%% file: structure_bilevel_learning.tex
\documentclass{article}
\usepackage[utf8]{inputenc}
\usepackage[short]{optidef}
\usepackage{booktabs}
\usepackage{amssymb}
\usepackage[margin=1.0in]{geometry}
\usepackage[textsize=scriptsize]{todonotes}
\usepackage[backend=biber]{biblatex}
\bibliography{structure_bilevel_learning.bib}
\usepackage{subcaption}
\usepackage{bm}
\usepackage{hyperref}
\usepackage{cleveref}
\usepackage{amsthm}
\usepackage{algorithm}
\usepackage{algorithmic}

\newcommand{\R}{\mathbb{R}}
\newcommand{\Kbb}{\mathbb{K}}
\newcommand{\Ical}{\mathcal{I}}
\newcommand{\Acal}{\mathcal{A}}

\newcommand{\Ccal}{\mathcal{C}}
\newcommand{\Bcal}{\mathcal{B}}

\newcommand{\Tcal}{\mathcal{T}}

\newcommand{\Ecal}{\mathcal{E}}

\newcommand*{\email}[1]{
    \small\href{mailto:#1}{#1}\par
}

\newcommand{\scalar}[2]{\langle #1 , #2 \rangle}
\newcommand{\bigscalar}[2]{\left\langle #1 , #2 \right\rangle}
\DeclareMathOperator*{\argmin}{arg\,min}

\DeclareMathOperator*{\range}{range}
\DeclareMathOperator*{\spann}{span}
\DeclareMathOperator*{\gph}{gph}

\usepackage{amsfonts}
\usepackage{graphicx}
\usepackage{pgfplots}
\pgfplotsset{compat=1.7}
\usetikzlibrary{pgfplots.groupplots}
\usepackage{algorithmic}
\ifpdf
  \DeclareGraphicsExtensions{.eps,.pdf,.png,.jpg}
\else
  \DeclareGraphicsExtensions{.eps}
\fi

\usepackage{enumitem}
\setlist[enumerate]{leftmargin=.5in}
\setlist[itemize]{leftmargin=.5in}

\newtheorem{remark}{Remark}

\newtheorem{corollary}{Corollary}

\theoremstyle{definition}
\newtheorem{definition}{Definition}[section]
\newtheorem{theorem}{Theorem}[section]
\newtheorem{lemma}[theorem]{Lemma}

\newcommand\numberthis{\addtocounter{equation}{1}\tag{\theequation}}

\title{Optimality Conditions for Bilevel Imaging Learning Problems with Total Variation Regularization}

\author{Juan Carlos De los Reyes\thanks{Research Center for Mathematical Modelling (MODEMAT), Escuela Polit\'ecnica Nacional, Quito, Ecuador
  (\email{juan.delosreyes@epn.edu.ec}, \email{david.villacis01@epn.edu.ec}, \url{http://www.modemat.epn.edu.ec/}).}
\and David Villac\'{\i}s\footnotemark[2]}

\usepackage{amsopn}

\begin{document}
\maketitle
\begin{abstract}
    We address the problem of optimal scale-dependent parameter learning in total variation image denoising. Such problems are formulated as bilevel optimization instances with total variation denoising problems as lower-level constraints. For the bilevel problem, we are able to derive M-stationarity conditions, after characterizing the corresponding Mordukhovich generalized normal cone and verifying suitable constraint qualification conditions. We also derive B-stationarity conditions, after investigating the Lipschitz continuity and directional differentiability of the lower-level solution operator. A characterization of the Bouligand subdifferential of the solution mapping, by means of a properly defined linear system, is provided as well. Based on this characterization, we propose a two-phase non-smooth trust-region algorithm for the numerical solution of the bilevel problem and test it computationally for two particular experimental settings.
\end{abstract}


  {\bf AMS:} 49K99, 90C33, 68U10, 68T99, 65K10



\input{sections/introduction}
\input{sections/problem_formulation}
\input{sections/mordukhovich_stationarity}
\input{sections/bouligand_stationarity.tex}
\input{sections/bouligand_subdifferential.tex}
\input{sections/trust_region}
\input{sections/experiments}

\printbibliography
\end{document}

%% file: sections/introduction.tex
\section{Introduction}

In this work we address the bilevel optimal selection of parameters in total variation image denoising models, from a variational and nonsmooth analysis perspectives. Bilevel techniques were proposed for optimal parameter selection of variational models in the seminal work \cite{de2013image}, provided a training dataset to learn from. The variational models considered there were based on the total variation (TV) seminorm and different noise statistics models were taken into account. Thereafter, apart of variational denoising problems, the bilevel learning framework has been considered for other imaging applications such as image restoration \cite{de2016structure,de2017bilevel}, blind image deconvolution \cite{hintermuller2015bilevel}, image segmentation \cite{ochs2016techniques,ranftl2014deep}, nonlocal models \cite{d2021bilevel} and kernel parameter estimation for support vector machines \cite{klatzer2015continuous}.
The bilevel methodology was also extended by Cao et al. \cite{van2017learning,calatroni2017bilevel}, Hintermüller et al. \cite{hintermuller2015bilevel,dong2011automated} and Strong et al. \cite{strong1996spatially} for learning different optimal \emph{scale-dependent} total variation (TV) and total generalized variation (TGV) denoising parameters. Either using a training set or noise statistics of the image, scale-dependent problems face the risk of leading to overfitted optimal parameters (see, e.g., \cite{dlrvillacis2021}). To gain generalization, alternative parameter functions may be considered that lie between a scalar and a fully scale-dependent weight. Examples of those intermediate functions are patch-dependent parameters, dictionary weights or basis functions coefficients.

One of the most challenging aspects in bilevel imaging learning, either with scalar or scale-dependent parameters, is the derivation of necessary optimality conditions that characterize the optimal solutions sharply and that may be used for numerical purposes. The approach pursued in \cite{de2013image,kunisch2013bilevel,van2017learning} was based on a local regularization of the nonsmooth terms and an asymptotic analysis thereafter, yielding a \emph{C-stationarity system} that does not fully characterize optimal parameters. A related approach, based on a dual reformulation of the lower level instance, was studied in \cite{hintermuller2017analytical,hintermuller2019generating}. Alternatively, a direct nonsmooth approach was considered in \cite{hintermuller2015bilevel} and \cite{dlrvillacis2021} to learn point spread functions in blind deconvolution models and the weight in front of the fidelity term in denoising models, respectively. In both such cases, the parameter affects the fidelity term and does not alter the structure of the primal-dual reformulation of the lower-level problem. Based on variational analysis tools, an M-stationarity system was then obtained.

Our aim in this paper consists in deriving sharp optimality conditions of \emph{Mordukhovich (M-)} and \emph{Bouligand (B-)} type for total variation bilevel learning problems, when the scale-dependent parameter appears within the regularizer. In such cases, differently from \cite{hintermuller2015bilevel,dlrvillacis2021}, the nonnegativity constraint of the parameter alters the structure of the primal-dual characterization of the variational denoising model, leading to the appearance of not only a biactive, but also a \emph{triactive set}. This involved nonsmoothness significantly complicates the derivation of the corresponding Fr\'echet and Mordukhovich normal cones, as well as the nonsmooth analysis of the lower-level solution mapping.

Firstly, after introducing the bilevel problem, we are able to reformulate it as a \emph{generalized mathematical program with equilibrium constraints (GMPEC)}. We rigorously characterize the corresponding Mordukhovich generalized normal cone (\Cref{lemma:mordukhovich_cone}) and verify a constraint qualification condition. This leads to the derivation of an \emph{M-stationarity} system for the characterization of local minima (\Cref{thm:m stationarity}).
Thereafter, we focus on \emph{B-stationarity} conditions, which require the directional differentiability of the solution mapping. In that respect, we are able to prove the Lipschitz continuity and directional differentiability of the solution operator, and characterize the directional derivative by means of a variational inequality (\Cref{thm: directional derivative}). In case of strict complementarity, the directional derivative becomes the Fr\'echet one and is characterized by a linear system. Also the Bouligand subdifferential of the solution operator is characterized by means of a general linear problem, based on a splitting of the biactive set and a properly defined subspace (\Cref{teo: G is a solution of the linear system}).

A further challenge is related with the numerical solution of the bilevel learning problems. Based on the properties of a local regularization technique, a numerical optimization algorithm of second-order type was considered in \cite{de2013image} and further explored in \cite{calatroni2013dynamic}, where dynamic sampling techniques were taken into account to accelerate the algorithm in the case of large datasets. A new approach handling the bilevel problem using a non-smooth setting can be found in the work of Ochs et. al. \cite{ochs2016techniques}, where the authors make use of algorithmic differentiation to find optimal parameters in variational models. An additional contribution of this paper is related to the solution of the bilevel instances by means of a nonsmooth trust-region algorithm. Specifically, based on the results in \cite{christof2017non} and the Bouligand subdifferential characterization, we devise a two-phase trust-region algorithm for the solution of the problems (\Cref{algo:tr_algorithm}). In the first phase, the Bouligand subdifferential is used in the quadratic trust-region subproblem, while in the second one, when the trust-region radius becomes small enough, a regularized gradient is utilized. We verify the performance of the proposed algorithm by means of scalar and scale-dependent experiments.

The paper is organized as follows. We will describe the bilevel learning problem in \Cref{sec:problem_formulation} along with the standing assumptions. In \Cref{sec:mordukhovich_stationarity} we will characterize the Mordhukovich generalized normal cone and verify an appropriate constraint qualification condition, that will enable us to derive an M-stationary optimality system for the bilevel learning problem. In \Cref{sec:bouligand_stationarity} we study the Lipschitz continuity and directional differentiability of the solution operator, with a proper characterization of the directional derivative, to obtain B-stationarity conditions for the bilevel problem. Thereafter, in \Cref{sec:bouligand_subdifferential} we characterize the Bouligand subdifferential of the solution mapping by means of a generalized linear system. This characterization of the Bouligand subdifferential will be of use in \Cref{sec:tr}, where we propose a non-smooth trust-region solution algorithm.  Finally, detailed numerical experiments will be provided in \Cref{sec:experiments}.

%% file: sections/problem_formulation.tex
\section{Problem Formulation}
\label{sec:problem_formulation}
Variational image denoising is a known \textit{ill-posed} inverse problem. Indeed, by considering an image as a $m_1\times m_2$ pixel matrix and mapping this matrix into a vector of length $m=m_1m_2$, a variational denoising model can be formulated as follows:
\begin{equation}
  \min_{u\in\R^m} \;\Ecal(u) = \phi(u,f) + \mathcal{R}(u,\alpha),
\end{equation}
where $f$ is the noise contaminated image, $\alpha$ is a parameter that balances a \textit{data fidelity} term $\phi$ and a \textit{regularization} term $\mathcal{R}$.


This paper will focus on the case where the regularization term corresponds to the \textit{Isotropic Total Variation Semi-norm} (TV). To characterize the TV regularizer, let us introduce the discrete gradient operator $\Kbb : \R^m\to\R^{n\times 2}$
\begin{equation*}
  \Kbb u = (\Kbb_x u, \Kbb_y u)\in \R^n\times\R^n.
\end{equation*}
The matrix $\Kbb_x$ computes the difference in all neighboring pixel intensities in the $x$-direction and $\Kbb_y$ in the $y$-direction. Taking $u\in\R^m$ as a given image, the isotropic discrete total variation is then defined by $\|\Kbb u\|_{2,1} := \sum_{j=1}^n \|(\Kbb u)_j\|,$
where $\|\cdot\|$ is the euclidean norm and $(\Kbb u)_j$ corresponds to the $j$-th row of $\Kbb u$.
This regularizer induces sparsity on the gradients of the image, which favors piecewise constant images with sparse edges \cite{strong2003edge}.

Now, regarding the balance parameter $\alpha$, it can be considered as a scalar value in the case we assume a uniform noise distribution across the entire image. However, such assumption is known to be unrealistic for most applications. Therefore, we will explore different types of balance parameters, starting with a \textit{scale-dependent} parameter $\alpha\in\R^n_+$ and the following form of the regularization term:
\begin{equation}\label{eq:data_fidelity_term}
  \mathcal{R}(u,\alpha) = \sum_{j=1}^n\alpha_j\|(\Kbb u)_j\|.
\end{equation}
With these considerations in mind, and assuming a $C^2$ and strongly convex data fidelity term $\phi:\R^m\to\R$, our variational denoising model reads as follows:
\begin{equation}\label{eq:sd_denoising_model}
  \min_{u\in\R^n} \;\mathcal{E}(u) := \phi(u,f) + \sum_{j=1}^n\alpha_j\|(\Kbb u)_j\|.
\end{equation}
In certain circumstances, such scale-dependent parameter may lead to an overfitted solution, which is why an intermediate approach is often desirable. One of such approaches considers a parameter $\alpha\in\R^p$ with $p<<n$ as piecewise constant in different patches of the image. Along the paper, we will derive the theoretical results only for the scale-dependent model, since they easily extend to the  patch-based and scalar models.

We will focus on bilevel strategies for finding optimal regularization parameters for the image denoising model \cref{eq:sd_denoising_model}. That is, given a \textit{training set} containing $N$ pairs of damaged and ``noise-free'' pairs of images, we will find an optimal regularization parameter $\alpha\in\R^n_+$ that minimizes a continuously differentiable, proper and strongly convex loss function $J:\R^m\to\R$, that measures the quality of the reconstruction with respect to the ground truth images in the training set.

Using training set pairs $(\bar{u}_i,f_i)$, $i=1,\dots,N$, with noisy images $f_i$ and ground truth images $\bar{u}_i$, our bilevel parameter learning problem reads as follows:
\begin{mini!}
  {\substack{\alpha\in\R^n, \alpha \ge 0}}{\sum_{i=1}^N\;J(u_i,\bar{u}_i)}
  {\label{eq:bilevel_problem}}{\label{eq:bilevel_problem_high}}
  \addConstraint{u_i \in}{\argmin_{u\in\R^m} \phi(u,f_i) + \sum_{j=1}^n\alpha_j\|(\Kbb u)_j\| \label{eq:bilevel_problem_lower},}{\qquad i=1,\dots,N}
\end{mini!}


Regarding the spatially dependent TV denoising problem \cref{eq:sd_denoising_model}, we can guarantee the existence of a unique solution as well as characterize it through a necessary and sufficient optimality condition.

\begin{theorem}\label{teo:unique_solution_vi}
  Problem \cref{eq:sd_denoising_model} has a unique solution $u^*\in\R^m$. Moreover, a necessary and sufficient condition is given by the following variational inequality of the second kind
  \begin{equation}\label{eq:nec_suff_condition_vi}
    \scalar{\phi'(u^*)}{v-u^*} + \sum_{j=1}^n\alpha_j\|(\Kbb v)_j\| - \sum_{j=1}^n\alpha_j\|(\Kbb u^*)_j\|\ge 0,\;\forall v \in\R^m.
  \end{equation}
\end{theorem}
\begin{proof}
  The strong convexity of $\phi$, along with the convexity of the total variation semi-norm, yields the strong convexity of the lower level optimization problem. Consequently this problem has a unique optimizer $u^*$. The variational inequality then follows by standard arguments (see, e.g., \cite[Theorem 6.1]{de2015numerical})
\end{proof}

Using duality techniques \cite{ekeland1999convex}, the variational inequality of the second kind \eqref{eq:nec_suff_condition_vi} can be equivalently written in primal-dual form, yielding the following reformulation of the lower-level problem:
\begin{subequations}\label{eq:fenchel_dual_lower_level}
  \begin{align}
    \phi'(u) + \Kbb^\top q &= 0,\\
    \scalar{q_j}{(\Kbb u)_j} - \alpha_j\|(\Kbb u)_j\|&=0,\;&\forall j=1,\dots,n,\\
    \|q_j\|-\alpha_j&\le 0,\;&\forall j=1,\dots,n,
  \end{align}
\end{subequations}
Consequently, the bilevel parameter learning problem, for a single training pair, can be written as:
\begin{mini!}
  {\substack{\alpha\in\R^n}}{J(u,\bar{u})}{\label{eq:bilevel_problem_reformulated}}{}
  \addConstraint{\phi'(u) + \Kbb^\top q }{= 0}{}{}
  \addConstraint{\scalar{q_j}{(\Kbb u)_j} - \alpha_j\|(\Kbb u)_j\|}{= 0,}{\qquad \forall j=1,\dots,n}{}
  \addConstraint{\|q_j\|-\alpha_j}{\le 0,}{\qquad \forall j=1,\dots,n}{}
  \addConstraint{\alpha_j}{\ge 0,}{\qquad \forall j=1,\dots,n.}{}
\end{mini!}
For the sake of clarity in the exposition, we restrict hereafter the analysis to the case of a single training pair. The results are, however, easily extendable to larger training sets.

%% file: sections/mordukhovich_stationarity.tex
\section{Mordukhovich Stationarity}\label{sec:mordukhovich_stationarity}
In this section we address stationarity conditions for the bilevel problem \cref{eq:bilevel_problem_reformulated} using variational analysis tools. Specifically, we reformulate the bilevel problem as a generalized mathematical program with equilibrium constraints (GMPEC) and verify a constraint qualification condition based on the variational geometry of the solution set.



A generalized mathematical program with equilibrium constraints may be formulated in the following general form:
\begin{mini}
  {}{f(x,y)}{}{\label{eq:outrata_cq_problem}}
  \addConstraint{0\in F_1(x,y) + Q(F_2(x,y))}
  \addConstraint{(x,y)\in\omega,}
\end{mini}
where $f$ is locally Lipschitz on $\R^n\times\R^m$, $F_1: \R^n\times\R^m \to \R^m$ and $F_2: \R^n\times\R^m \to \R^l$ are continuously differentiable, $\omega\subset \R^n\times\R^m$ is non-empty and closed and $Q:\R^l \rightrightarrows \R^m$ is a multifunction with closed graph. A constraint qualification condition for GMPEC that guarantee the existence of KKT multipliers and its corresponding stationarity system is privede in the next result.
\begin{theorem}[Outrata \cite{outrata2000generalized}]
  \label{teo:outrata_cq}
  Let $(x^*,y^*)$ be a local solution of \cref{eq:outrata_cq_problem} and the following constraint qualification
  \begin{multline}\label{eq:outrata_cq_conditions}
    \left.
    \begin{aligned}
      \begin{bmatrix}
        (\nabla_x F_2(x^*,y^*))^\top & -(\nabla_x F_1(x^*,y^*))^\top\\
        (\nabla_y F_2(x^*,y^*))^\top & -(\nabla_y F_1(x^*,y^*))^\top
      \end{bmatrix}
      \begin{bmatrix}
        w\\
        z
      \end{bmatrix}\in -N^M_\omega(x^*,y^*),\\
      (w,z)\in N^M_{\gph Q}(F_2(x^*,y^*),-F_1(x^*.y^*))
    \end{aligned}\right\}\text{implies}\begin{cases}
      w=0,\\
      z=0
    \end{cases}
  \end{multline}
  holds true. Then there exist a pair $(\xi,\eta)\in\partial f(x^*,y^*)$, a pair $(\gamma,\delta)\in N^M_\omega(x^*,y^*)$, and a KKT pair $(w^*,z^*)\in N^M_{\gph Q}(F_2(x^*,y^*),-F_1(x^*,y^*))$ such that
  \begin{align*}
    0 = \xi + (\nabla_x F_2(x^*,y^*))^\top w^* - (\nabla_x F_1(x^*,y^*))^\top z^* + \gamma,\\
    0 = \eta + (\nabla_y F_2(x^*,y^*))^\top w^* - (\nabla_y F_1(x^*,y^*))^\top z^* + \delta.
  \end{align*}
  where $N^M_{\gph Q}$ stands for Mordukhovich generalized normal cone to the graph of $Q$.
\end{theorem}

In our case, using the primal-dual formulation \cref{eq:fenchel_dual_lower_level}, the constraints may be written as
\begin{equation}\label{eq:generalized_equation}
  0\in \phi'(u) + Q(\alpha,u),
\end{equation}
where $Q:\R^n_+\times\R^m \rightrightarrows \R^m$ is the set-valued operator associated to the subdifferential of the Euclidean norm, i.e.,
\begin{equation}\label{eq:set_valued_Q}
  Q(\alpha,u) := \left\{\Kbb^\top q:q\in \R^{n\times 2} , \begin{cases}
    q_j = \alpha_j\frac{(\Kbb u)_j}{\|(\Kbb u)_j\|},\;&\text{ if } \|(\Kbb u)_j\| \neq 0,\alpha_j \ge 0,\\
    \|q_j\|\le \alpha_j,\;&\text{ if }\|(\Kbb u)_j\|=0,\alpha_j \ge 0.
  \end{cases}\right\}
\end{equation}
Equivalently, by making use of the definition of the graph of the multifunction $Q$ we can rewrite \cref{eq:generalized_equation} as follows
\begin{subequations}\label{eq:generalized_equation_2}
  \begin{align}
    \phi'(u) + \Kbb^\top q &= 0,\\
    (\alpha,u,\Kbb^\top q)&\in \gph\;Q,\\
    (\alpha,u)&\in\omega:=\R^n_+\times\R^m,
  \end{align}
\end{subequations}
where $\gph Q:=\{(\alpha,u,\Kbb^\top q)\in\R^n_+\times\R^m\times\R^m: \Kbb^\top q \in Q(\alpha,u)\}$.

The constraint qualification in \cref{teo:outrata_cq} guarantees the existence of multipliers that allow the derivation of a stationarity system. In \cite{hintermuller2015bilevel} such constraint qualification condition was circumvented by using Robinson strong regularity of the set valued equation. In contrast, the structure of the multifunction \cref{eq:set_valued_Q}, depending both on the regularization parameter $\alpha$ and the image $u$, prevent us from using the same shortcut.

Using the structure of the set valued operator $Q$ presented in \cref{eq:set_valued_Q}, let us introduce the following notation for the inactive, strongly active, biactive, zero-inactive and triactive sets, respectively:
\begin{align*}
  \Ical(\alpha,u) &:=\{j\in\{1,\dots,n\}:(\Kbb u)_j\neq 0,\; \alpha_j>0\},\\
  \Acal_s(\alpha,u) &:=\{j\in\{1,\dots,n\}:\|q_j\|<\alpha_j\},\\
  \Bcal(\alpha,u) &:= \{j\in\{1,\dots,n\}:\|q_j\|=\alpha_j,\;(\Kbb u)_j = 0,\;\alpha_j > 0\},\\
  \Ical_0(\alpha,u) &:=\{j\in\{1,\dots,n\}:(\Kbb u)_j\neq 0,\; \alpha_j=0\},\\
  \Tcal(\alpha,u) &:= \{j\in\{1,\dots,n\}:\|q_j\|=\alpha_j,\;(\Kbb u)_j = 0,\;\alpha_j = 0\}.
\end{align*}
We will omit the arguments in the set notation whenever they can be inferred from the context. Let us note that the condition in the strongly active set $\Acal_s$ implies a strict positive parameter $\alpha_j>0$ for this index set.

The constraint qualification condition \cref{teo:outrata_cq} makes use of key fundamental definitions from Mordukhovich's generalized calculus. In particular, the Mordukhovich normal cone to the graph of the mutifunction $Q$. For the reader's convenience we will lay some basic concepts in variational geometry that will be used through this section. For a more rigorous review refer to, e.g., \cite{rockafellar2009variational}

\begin{definition}[Bouligand Tangent Cone]\label{def:bouligand_tangent_cone}
  Let $C\subset\R^n$ and $\bar{x}\in cl C$. The Bouligand Tangent Cone to $C$ at $x$ is defined as
  \begin{equation}
    T_C(\bar{x}):=\{d\in\R^n:\exists t_k \to 0,\exists x_k\in C:\frac{x_k-\bar{x}}{t_k}\to d\}.
  \end{equation}
\end{definition}

\begin{definition}[Fréchet Normal Cone]\label{def:frechet_normal_cone}
  Let $C\subset\R^n$ and $\bar{x}\in cl C$. The Fréchet Normal Cone is defined as the polar to $T_C(\bar{x})$, i.e.,
  \begin{equation}
    N^F_C(\bar{x})=\left\{d\in\R^n : \limsup_{x\overset{C}{\to}\bar{x}}\frac{\scalar{d}{x-\bar{x}}}{\|x-\bar{x}\|}\le 0\right\}=[T_C(\bar{x})]^\circ.
  \end{equation}
\end{definition}

\begin{definition}[Mordukhovich Normal Cone]\label{def:mordukhovich_normal_cone}
  Let $C\subset\R^n$ and $\bar{x}\in cl C$. The Mordukhovich Normal Cone is defined as
  \begin{equation}
    N^M_C(\bar{x}):=\{d\in\R^n:\exists x_k\to\bar{x},d_k\in N^F_C(x_k)\text{ s.t. }d_k\to d\}.
  \end{equation}
\end{definition}

\begin{remark}\label{rem:mordukhovich_convex}
  If the set $C$ in \cref{def:mordukhovich_normal_cone} is convex, $N^M_C(\bar{x})$ amounts to the standard normal cone to $C$ at $\bar{x}$. Otherwise, this cone is in general non-convex.
\end{remark}

In the following lemmata we obtain the Bouligand tangent cone, the Fr\'echet normal cone and the Mordukhovich normal cone to the graph of the multifunction $Q$.
\begin{lemma}\label{lemma:tangent_cone}
  The Bouligand tangent cone to the graph of $Q$, described in \cref{eq:generalized_equation_2}, is given by
  \begin{multline*}\label{eq:contingent_cone}
    T_{\gph Q}(\alpha,u,\Kbb^\top q) =\\ \left\{ (\delta_\alpha,\delta_u,\Kbb^\top\delta_q):
    \left \{
    \begin{aligned}
      & (\delta_q)_j - (\delta_\alpha)_j\frac{(\Kbb u)_j}{\|(\Kbb u)_j\|} - \alpha_jT_j(\Kbb \delta_u)_j=0,&&\text{if }j\in\Ical,\\
      & (\Kbb\delta_u)_j = 0, &&\text{if }j\in\Acal_s,\\
      & \left.
      \begin{aligned}
          & (\Kbb\delta_u)_j = 0,\scalar{(\delta_q)_j}{q_j}\le \alpha_j(\delta_\alpha)_j\;\vee \\
          & (\Kbb\delta_u)_j = \tilde{c}q_j (\tilde{c}\ge 0),\scalar{(\delta_q)_j}{q_j}=\alpha_j(\delta_\alpha)_j
      \end{aligned} \right\}  &&\text{if }j\in\Bcal,\\
      & (\delta_q)_j - (\delta_\alpha)_j\frac{(\Kbb u)_j}{\|(\Kbb u)_j\|}=0,\;(\delta_\alpha)_j\ge 0&&\text{if }j\in\Ical_0,\\
      & \left.
      \begin{aligned}
          & (\delta_\alpha)_j \ge 0,(\Kbb\delta_u)_j\in\R^2\backslash\{0\},(\delta_q)_j - (\delta_\alpha)_j\frac{(\Kbb \delta_u)_j}{\|(\Kbb \delta_u)_j\|}=0\;\vee \\
          & (\delta_\alpha)_j \ge 0, (\Kbb \delta_u)_j = 0, \|(\delta_q)_j\| - (\delta_\alpha)_j \le 0\\
      \end{aligned} \right\}  &&\text{if }j\in\Tcal,
    \end{aligned} \right.
    \right\}
  \end{multline*}
  where
  \begin{equation*}
    T_j(\Kbb v)_j = \frac{(\Kbb v)_j}{\|(\Kbb u)_j\|} - \frac{(\Kbb u)_j(\Kbb u)_j^\top(\Kbb v)_j}{\|(\Kbb u)_j\|^3}, \text{ for } v\in\R^n.
  \end{equation*}
\end{lemma}
\begin{proof}
  The tangent cone to the graph of the multifunction $Q$ is defined as
  \begin{multline*}
    T_{\gph Q}(\alpha,u,\Kbb^\top q) = \{(\delta_\alpha,\delta_u,\Kbb^\top \delta_q)\in\R^n\times\R^m\times\R^m: \exists t_k\to 0, (\alpha_k,u_k,\Kbb^\top q_k)\in \gph Q :\\\frac{1}{t_k}((\alpha_k,u_k,\Kbb^\top q_k)-(\alpha,u,\Kbb^\top q))\to (\delta_\alpha,\delta_u,\Kbb^\top \delta_q)\}.
  \end{multline*}
  In order to calculate this cone, we split the analysis into different cases, according to the definition of the multifunction $Q$.
  \begin{description}
  \item[Case 1: $\bm{j\in\Ical.}$]
  In this index set the dual variable can be uniquely characterized. According to \cref{eq:set_valued_Q}, the following equation is fulfilled:
  \begin{equation}\label{eq:lyusternik_manifold}
    h_j(\alpha,u,\Kbb^\top q) = q_j-\alpha_j\frac{(\Kbb u)_j}{\|(\Kbb u)_j\|}=0.
  \end{equation}
  Using Lyusternik's theorem \cite[Theorem 4.21]{jahn2020introduction}, the $j-$th component of the tangent direction then satisfies
  \begin{equation*}
    (\delta_q)_j-(\delta_\alpha)_j\frac{(\Kbb u)_j}{\|(\Kbb u)_j\|} - \alpha_jT_j(\Kbb\delta_u)_j=0.
  \end{equation*}
  \item[Case 2: $\bm{j\in\Acal_s.}$]
  In this index set we know $\|q_j\|<\alpha_j$. Therefore, this point can only be approximated by taking sequences in the strongly active set. For $n$ sufficiently large we then take sequences such that $(\Kbb u_k)_j=0$, $\|(q_k)_j\|<(\alpha_k)_j$. Taking the limit in $(\Kbb u_k)_j=0$, as $k \to \infty$, yields $(\Kbb \delta_u)_j = 0$. For the dual variable, let us take the sequence $(q_k)_j = q_j + t_k d$ with arbitrary $d\in\R^2$. It then follows that
  \begin{equation*}
    (\delta_q)_j = \lim_{t_k\to 0} \frac{(q_k)_j-q_j}{t_k} = d.
  \end{equation*}
  Since we took an arbitrary direction $d$ it yields $(\delta_q)_j\in\R^2$. Similarly, we get that $(\delta_\alpha)_j\in\R$.
  \item[Case 3: $\bm{j\in\Bcal.}$]
  There are three possible approximations to a point in this index set, via inactive, strongly active or biactive sequences. Since $\alpha_j>0$ in all these three cases, we can approximate it by sequences $(\alpha_k)_j<\alpha_j$ or $(\alpha_k)_j>\alpha_j$. Consequently, we get $(\delta_\alpha)_j\in\R$.

  Now, if we approach a biactive point using a sequence in the inactive set, this sequence satisfies $(\Kbb u_k)_j\neq 0$. Taking in particular $(\Kbb u_k)_j = (\Kbb u)_j + t_k (\Kbb\delta_{u_k})_j = t_k (\Kbb\delta_{u_k})_j$, then the sequence of dual variables has the following form
  \begin{equation}\label{eq:tan_cone_b_3}
    (q_k)_j = (\alpha_k)_j\frac{(\Kbb u_k)_j}{\|(\Kbb u_k)_j\|} = (\alpha_k)_j\frac{(\Kbb\delta_{u_k})_j}{\|(\Kbb\delta_{u_k})_j\|}.
  \end{equation}
  Furthermore, considering the following product
  \begin{equation*}
    \scalar{(q_k)_j}{(\Kbb\delta_{u_k})_j} =\bigscalar{(\alpha_k)_j\frac{(\Kbb\delta_{u_k})_j}{\|(\Kbb\delta_{u_k})_j\|}}{(\Kbb\delta_{u_k})_j} = (\alpha_k)_j \|(\Kbb\delta_{u_k})_j\|,
  \end{equation*}
  and taking the limit as $k \to \infty$ we get $\scalar{q_j}{(\Kbb\delta_u)_j} = \alpha\|(\Kbb\delta_u)_j\|$. Recalling that in this index set $\|q_j\|=\alpha_j > 0$, we know both vectors are collinear, i.e.,
  \begin{equation}\label{eq:tan_cone_b_0}
    (\Kbb\delta_u)_j=\tilde{c} q_j,\;\text{ for some }\tilde{c}\ge 0.
  \end{equation}

  Using that $\|q_j\|=\alpha_j$, the following product holds
  \begin{align*}
    \bigscalar{\frac{(q_k)_j-q_j}{t_k}}{q_j} &= \frac{1}{t_k}(\scalar{(q_k)_j}{q_j}-\scalar{q_j}{q_j})= \frac{1}{t_k}(\scalar{(q_k)_j}{q_j} - \scalar{(q_k)_j}{(q_k)_j} + \scalar{(q_k)_j}{(q_k)_j} - \alpha_j^2),\\
    &= \bigscalar{(q_k)_j}{\frac{q_j-(q_k)_j}{t_k}} + \frac{(\alpha_k)_j^2-\alpha_j^2}{t_k}.
  \end{align*}
  The last equality holds since from \cref{eq:tan_cone_b_3}, we get the product $\scalar{(q_k)_j}{(q_k)_j} = (\alpha_k)_j^2$. Taking the limit as $t_k\to 0$ we get
  \begin{equation}\label{eq:tan_cone_b_1}
    \scalar{(\delta_q)_j}{q_j} = \alpha_j(\delta_\alpha)_j.
  \end{equation}

  Now, if the approximation is done through a sequence of strongly active points, we know the sequence satisfies $(\Kbb u_k)_j=0$ and $\|q_k\|<(\alpha_k)_j$. In this case we know $(\Kbb \delta_u)_j = 0$ and, using the Cauchy-Schwarz inequality, we get
  \begin{equation*}
    \bigscalar{\frac{(q_k)_j-q_j}{t_k}}{q_j} \le \frac{\alpha_j}{t_k}(\|(q_k)_j\|-\|q_j\|) < \alpha_j\left(\frac{(\alpha_k)_j-\alpha_j}{t_k}\right),
  \end{equation*}
  which implies that $\scalar{(\delta_q)_j}{q_j} \le \alpha_j(\delta_\alpha)_j$.

  Finally, approximating through a biactive set sequence, we know again $(\Kbb\delta_u)_j=0$ and, estimating the product
  \begin{equation*}
    \bigscalar{\frac{(q_k)_j-q_j}{t_k}}{q_j} \le \frac{\alpha_j}{t_k}(\|(q_k)_j\|-\|q_j\|) = \alpha_j\left(\frac{(\alpha_k)_j-\alpha_j}{t_k}\right),
  \end{equation*}
  we get
  \begin{equation}\label{eq:tan_cone_b_2}
    \scalar{(\delta_q)_j}{q_j} \le \alpha_j(\delta_\alpha)_j.
  \end{equation}

  \item[Case 4: $\bm{j\in\Ical_0.}$]
  We can approximate a point in the zero-inactive set by sequences in the inactive set. Therefore, considering a sequence $(\Kbb u_k)_j = (\Kbb u)_j + t_k v$ with $v\in\R^2$ arbitrary we have
  \begin{equation}\label{eq:kdu_zero_inactive}
    (\Kbb\delta_u)_j = \lim_{t_k\to 0}\frac{(\Kbb u_k)_j-(\Kbb u)_j}{t_k} = v.
  \end{equation}
  Since we took $v\in\R^2$ arbitrary, it follows that $(\Kbb\delta_u)\in\R^2$. Furthermore, since $q_j = \alpha_j((\Kbb u)_j/\|(\Kbb u)_j\|)$ and $\alpha_j=0$, then $q_j=0$ in this index set. Considering the sequence $(\alpha_k)_j= \alpha_j + t_k h$ we obtain
  \begin{equation}\label{eq:tan_cone_I0_2}
    (\delta_q)_j = \lim_{t_k\to 0}\frac{(q_k)_j}{t_k} = \lim_{t_k\to 0}h\frac{(\Kbb u_k)_j}{\|(\Kbb u_k)_j\|} = (\delta_\alpha)_j\frac{(\Kbb u)_j}{\|(\Kbb u)_j\|}.
  \end{equation}
  Since $\alpha_j=0$, the only valid approximations are the ones coming from positive elements, thus $(\delta_\alpha)_j\ge 0$. Another possible approximation can be done through zero-inactive points, meaning $(q_k)_j=0$ and $(\alpha_k)_j=0$. This case implies $(\delta_q)_j=0$, $(\delta_\alpha)_j=0$ and $(\Kbb\delta_u)_j\in\R^2$, which is a particular case of \cref{eq:tan_cone_I0_2}.

  \item[Case 5: $\bm{j\in\Tcal.}$]
  There are four ways to approach a triactive point. As in the zero-inactive case, all valid approximations come from $(\alpha_k)_j\ge 0$, which again implies $(\delta_\alpha)_j\ge 0$. Similarly to the zero-inactive case \cref{eq:kdu_zero_inactive}, we also get $(\Kbb\delta_u)_j\in\R^2$.

  Approximating through a sequence in the inactive set, we obtain
  \begin{equation}\label{eq:tan_cone_I0_1}
    (\delta_q)_j = \lim_{t_k\to 0}\frac{1}{t_k}\left((\alpha_k)_j\frac{(\Kbb u_k)_j}{\|(\Kbb u_k)_j\|}\right) = (\delta_\alpha)_j\frac{(\Kbb\delta_u)_j}{\|(\Kbb\delta_u)_j\|}.
  \end{equation}
  Likewise, the approximation can be made using zero-inactive points. In this case $(\Kbb\delta_u)_j \neq 0$, $(q_k)_j=0$ and $(\alpha_k)_j=0$. From this sequence we can derive $(\delta_q)_j=0$, $(\delta_\alpha)_j=0$ and $(\Kbb \delta_u)_j\in\R^2$, which is included in \cref{eq:tan_cone_I0_1}.

  Moving forward, we can also approximate through strongly active points, i.e., $(\Kbb u_k)_j=0$ and $\|(q_k)_j\|<(\alpha_k)_j$. From this sequence we know $(\Kbb \delta_u)_j=0$ and $(\delta_\alpha)_j\ge 0$ and the dual variable will have the following form
  \begin{equation}\label{eq:tan_cone_I0_3}
    \|(\delta_q)_j\| = \left\|\lim_{t_k\to 0}\frac{(q_k)_j}{t_k}\right\| = \lim_{t_k\to 0} \frac{1}{t_k}\|(q_k)_j\| \le \lim_{t_k\to 0}\frac{1}{t_k}(\alpha_k)_j = (\delta_\alpha)_j,
  \end{equation}
  yielding $\|(\delta_q)_j\|\le(\delta_\alpha)_j$.

  Finally, we consider an approximation through biactive points, meaning $(\Kbb u_k)_j=0$ and $\|(q_k)_j\|=(\alpha_k)_j$. We the obtain $(\Kbb \delta_u)_j =0$, $(\delta_\alpha)_j\ge 0$ and
  \begin{equation*}
    \|(\delta_q)_j\| = \left\|\lim_{t_k\to 0}\frac{(q_k)_j}{t_k}\right\| = \lim_{t_k\to 0} \frac{1}{t_k}\|(q_k)_j\| = \lim_{t_k\to 0}\frac{1}{t_k}(\alpha_k)_j = (\delta_\alpha)_j,
  \end{equation*}
  which is a particular case of \cref{eq:tan_cone_I0_3}. Consequently, $\|(\delta_q)_j\|=(\delta_\alpha)_j$.
\end{description}
\end{proof}

\begin{lemma}\label{lemma:frechet_cone}
  The Fréchet normal cone to the graph of $Q$, described in \cref{eq:generalized_equation_2}, is given by
  \begin{multline}\label{eq:frechet_normal_cone}
    N^F_{\gph Q}(\alpha,u,\Kbb^\top q) =\\ \left\{(\vartheta,\Kbb^\top \mu,p):
    \left\{
    \begin{aligned}
    & \mu_j + \alpha_jT_j(\Kbb p)_j = 0,&&\text{if }j\in\Ical,\\
    & \vartheta_j+\frac{\scalar{(\Kbb u)_j}{(\Kbb p)_j}}{\|(\Kbb u)_j\|}=0,&&\text{if }j\in\Ical,\\
    & \vartheta_j=0,\;(\Kbb p)_j=0  &&\text{if }j\in\Acal_s,\\
    & \vartheta_j+c\alpha_j = 0,(\Kbb p)_j = c q_j (c\ge 0),\scalar{\mu_j}{q_j} \le 0,&&\text{if }j\in\Bcal,\\
    & \vartheta_j + \frac{\scalar{(\Kbb u)_j}{(\Kbb p)_j}}{\|(\Kbb u)_j\|} \le 0,\quad \mu_j = 0,&&\text{if }j\in\Ical_0,\\
    & \vartheta_j\le 0,\;(\Kbb p)_j = 0,\;\mu_j=0,&&\text{if }j\in\Tcal.
    \end{aligned}
    \right.
    \right\}
  \end{multline}
\end{lemma}
\begin{proof}
  Using the definition of the Fréchet normal cone for this problem
  \begin{equation*}
    N^F_{\gph Q}(\alpha,u,\Kbb^\top q) = \{(\vartheta,\Kbb^\top \mu,p)\in\R^n\times\R^m\times\R^m:\scalar{(\vartheta,\Kbb^\top \mu,p)}{(\delta_\alpha,\delta_u,\Kbb^\top \delta_q)}\le 0\}.
  \end{equation*}
  Indeed, we can rewrite this inequality as
  \begin{equation*}
    \sum_{j=1}^n ((\delta_\alpha)_j\vartheta_j + \scalar{(\Kbb \delta_u)_j}{\mu_j}+\scalar{(\delta_q)_j}{(\Kbb p)_j})\le 0.
  \end{equation*}
  Using this representation, along with the tangent cone characterization from \cref{lemma:tangent_cone}, we analyze the different cases according to their index set.
  \begin{description}
    \item[Case 1: $\bm{j\in\Ical.}$]
    Using the characterization of the elements in the tangent cone, we have
  \begin{equation*}
    (\delta_\alpha)_j\vartheta_j+\scalar{(\Kbb \delta_u)_j}{\mu_j} + \bigscalar{(\delta_\alpha)_j\frac{(\Kbb u)_j}{\|(\Kbb u)_j\|}}{(\Kbb p)_j} + \scalar{\alpha_j T_j(\Kbb \delta_u)_j}{(\Kbb p)_j}\le 0.
  \end{equation*}
  Rearranging the terms and using the symmetry of $T_j$,
  \begin{equation*}
    (\delta_\alpha)_j \left(\vartheta_j+\frac{\scalar{(\Kbb u)_j}{(\Kbb p)_j}}{\|(\Kbb u)_j\|}\right) + \scalar{(\Kbb \delta_u)_j}{\mu_j+\alpha_jT_j(\Kbb p)_j}\le 0.
  \end{equation*}
  Since $(\Kbb \delta_u)_j\in\R^2$ and $(\delta_\alpha)_j\in\R$, it necessarily must hold
  \begin{equation*}
    \vartheta_j+\frac{\scalar{(\Kbb u)_j}{(\Kbb p)_j}}{\|(\Kbb u)_j\|}=0,\qquad \mu_j+\alpha_jT_j(\Kbb p)_j = 0.
  \end{equation*}

  \item[Case 2: $\bm{j\in\Acal_s.}$]
     In this index set we know that $(\Kbb \delta_u)_j=0$, $(\delta_\alpha)_j\in\R$ and $(\delta_q)_j\in\R^2$. Consequently the product reads
  \begin{equation*}
    (\delta_\alpha)_j\vartheta_j + \scalar{(\delta_q)_j}{(\Kbb p)_j}\le 0,
  \end{equation*}
  and we obtain that $(\Kbb p)_j=0$ and $\vartheta_j=0$.

  \item[Case 3: $\bm{j\in\Bcal.}$]
    In this index set there are two conditions on the normal directions. For the first one we take $(\Kbb \delta_u)_j=0$ and the cone inequality reads
    \begin{equation}\label{eq:frechet_bcal_6}
      (\delta_\alpha)_j\vartheta_j + \scalar{(\delta_q)_j}{(\Kbb p)_j}\le 0,\quad\forall (\delta_\alpha)_j,(\delta_q)_j\text{ s.t. }\scalar{(\delta_q)_j}{q_j}\le \alpha_j(\delta_\alpha)_j.
    \end{equation}
    Taking in particular, $(\delta_\alpha)_j=0$ we get
    \begin{equation*}
      \scalar{(\delta_q)_j}{(\Kbb p)_j}\le 0,\forall (\delta_q)_j\text{ s.t. }\scalar{(\delta_q)_j}{q_j}\le 0.
    \end{equation*}
    Therefore, $(\Kbb p)_j = c q_j$ with $c\ge 0$. Using this result in \cref{eq:frechet_bcal_6} for the particular case $(\delta_\alpha)_j = \frac{1}{\alpha_j}\scalar{(\delta_q)_j}{q_j}$, we obtain
    \begin{equation*}
      0\ge(\delta_\alpha)_j\vartheta_j + \scalar{(\delta_q)_j}{c q_j} =(\delta_\alpha)_j(\vartheta_j + c\alpha_j),
    \end{equation*}
    from where $\vartheta_j +c\alpha_j=0$ holds. The resulting cone reads
    \begin{equation}\label{eq:frechet_bcal_1}
      \vartheta_j+c\alpha_j=0,\quad (\Kbb p)_j = c q_j,\quad c\ge 0, \quad \mu_j\in\R^2.
    \end{equation}

    For the second case we take $(\Kbb\delta_u) = c q_j\;(c\ge 0)$ in \cref{eq:frechet_bcal_6},
    \begin{equation}\label{eq:frechet_bcal_7}
      (\delta_\alpha)_j\vartheta_j + \scalar{c q_j}{\mu_j}+\scalar{(\Kbb p)_j}{(\delta_q)_j}\le 0,\forall (\delta_\alpha)_j\in\R,\;(\delta_q)_j\text{ s.t. }\scalar{(\delta_q)_j}{q_j}= \alpha_j(\delta_\alpha)_j,
    \end{equation}
    Again, considering $(\delta_\alpha)_j=0$ and $(\Kbb\delta_u)_j=0$, we get
    \begin{equation*}
      \scalar{(\delta_q)_j}{(\Kbb p)_j}\le 0,\forall (\delta_q)_j\text{ s.t. }\scalar{(\delta_q)_j}{q_j}= 0.
    \end{equation*}
    Consequently, $(\Kbb p)_j = c q_j$ with $c\in\R$. Using this result in \cref{eq:frechet_bcal_7}, while keeping $(\delta_\alpha)_j=0$, yields $\tilde{c}\scalar{q_j}{\mu_j}\le 0$. Thanks to the positiveness of $\tilde{c}$ we then get $\scalar{q_j}{\mu_j}\le 0$. Now, applying all previous results in \cref{eq:frechet_bcal_7} we get
    \begin{equation*}
      0\ge(\delta_\alpha)_j\vartheta_j +\scalar{c q_j}{(\delta_q)_j} = (\delta_\alpha)_j\vartheta_j +c\alpha_j(\delta_\alpha)_j,\forall (\delta_\alpha)_j\in\R, c\in\R,
    \end{equation*}
    yielding $\vartheta_j\in\R$. Therefore, the resulting cone for the second case reads
    \begin{equation}\label{eq:frechet_bcal_2}
      \vartheta_j \in\R,\quad (\Kbb p)_j = c q_j,\quad c\in\R.
    \end{equation}
    Finally, considering both cases we obtain
  \begin{equation*}
    \vartheta_j+c\alpha_j = 0\wedge\scalar{\mu_j}{q_j} \le 0 \wedge (\Kbb p)_j = c q_j \wedge c\ge 0.
  \end{equation*}

  \item[Case 4: $\bm{j\in\Ical_0.}$]
  By using the characterization of the tangent cone in this index set, we have
  \begin{equation*}
    (\delta_\alpha)_j\left(\vartheta_j+\frac{\scalar{(\Kbb u)_j}{(\Kbb p)_j}}{\|(\Kbb u)_j\|}\right) + \scalar{(\Kbb \delta_u)_j}{\mu_j} \le 0.
  \end{equation*}
  This relationship must hold for all $(\Kbb \delta_u)_j\in\R^2$ and $(\delta_\alpha)_j\ge 0$, which implies
  \begin{equation*}
    \vartheta_j + \frac{\scalar{(\Kbb u)_j}{(\Kbb p)_j}}{\|(\Kbb u)_j\|}\le 0,\quad \scalar{(\Kbb \delta_u)_j}{\mu_j} \le 0.
  \end{equation*}
  Since $(\Kbb \delta_u)_j\in\R^2$ we get $\mu_j=0$.

  \item[Case 5: $\bm{j\in\Tcal.}$]
   For the first case in this index set we know the elements of the tangent cone satisfy
   \begin{equation*}
    (\delta_q)_j = (\delta_\alpha)_j\frac{(\Kbb \delta_u)_j}{\|(\Kbb \delta_u)_j\|},\;(\delta_\alpha)_j\ge 0,\;(\Kbb\delta_u)_j\in\R^2\backslash\{0\}.
   \end{equation*}
   Replacing these terms into the normal cone inequality,
   \begin{equation*}
     (\delta_\alpha)_j\left(\vartheta_j+\frac{\scalar{(\Kbb p)_j}{(\Kbb\delta_u)_j}}{\|(\Kbb\delta_u)_j\|}\right) + \scalar{\mu_j}{(\Kbb\delta_u)_j}\le 0,\;\forall (\delta_\alpha)_j\ge 0,(\Kbb\delta_u)_j\in\R^2\backslash\{0\}.
   \end{equation*}
   In particular, for $(\delta_\alpha)_j = 0$, we get that $\scalar{\mu_j}{(\Kbb\delta_u)_j}\le 0$, for all $(\Kbb \delta_u)_j\in\R^2\backslash\{0\}$, which implies that $\mu_j=0$. Moreover, thanks to the positiveness of $(\delta_\alpha)_j\ge 0$, we get
   \begin{equation}\label{eq:frechet_bcal_4}
    \vartheta_j + \frac{\scalar{(\Kbb p)_j}{(\Kbb\delta_u)_j}}{\|(\Kbb\delta_u)_j\|}\le 0,\forall (\Kbb\delta_u)_j\in\R^2\backslash\{0\}.
   \end{equation}
   Testing this inequality with $\pm(\Kbb\delta_u)_j$, we get $\vartheta_j\le\scalar{(\Kbb p)_j}{(\Kbb\delta_u)_j}\le -\vartheta_j$, which implies that the product $\scalar{(\Kbb p)_j}{(\Kbb\delta_u)_j}=0$ for all $(\Kbb\delta_u)_j\in\R^2\backslash\{0\}$, consequently $(\Kbb p)_j=0$. Using this result in \cref{eq:frechet_bcal_4} yields also $\vartheta_j\le 0$.

   Now, regarding the second condition, we know $\|(\delta_q)_j\| - (\delta_\alpha)_j\le 0, ~(\Kbb \delta_u)_j =0$ and $(\delta_\alpha)_j\ge 0.$
   Since $(\Kbb \delta_u)_j=0$, it follows $\mu_j\in\R^2$. Using the normal cone inequality,
   \begin{equation*}
     0 \ge (\delta_\alpha)_j\vartheta_j + \scalar{(\delta_q)_j}{(\Kbb p)_j}\ge (\delta_\alpha)_j\vartheta_j - \|(\delta_q)_j\|\|(\Kbb p)_j\|\ge (\delta_\alpha)_j\vartheta_j - (\delta_\alpha)_j\|(\Kbb p)_j\|,
   \end{equation*}
   from where we derive $(\delta_\alpha)_j(\vartheta_j - \|(\Kbb p)_j\|)\le 0$. Along with the positivity of $(\delta_\alpha)_j$, this implies $\vartheta_j \le \|(\Kbb p)_j\|$. Finally, considering both conditions, the result is obtained.
  \end{description}
\end{proof}

\begin{lemma}\label{lemma:mordukhovich_cone}
  The Mordukhovich normal cone to the graph of $Q$, described in \cref{eq:generalized_equation_2}, is given by
\begin{multline*}
  N_{\gph Q}^{M}(\alpha,u,\mathbb K^\top q) =\\ \left\{ (\vartheta,\mathbb K^\top \mu,p):
  \left \{
  \begin{aligned}
    & \mu_j + \alpha_jT_j(\Kbb p)_j = 0,&&\text{if }j\in\Ical,\\
    & \vartheta_j+\frac{\scalar{(\Kbb u)_j}{(\Kbb p)_j}}{\|(\Kbb u)_j\|}=0,&&\text{if }j\in\Ical,\\
    & \vartheta_j=0,(\Kbb p)_j=0, &&\text{if }j\in\Acal_s,\\
    & \left.
    \begin{aligned}
      & \vartheta_j=0,(\Kbb p)_j=0,\;\vee\\
      & (\mathbb K p)_j=c q_j (c\in\R), ~\langle \mu_j,q_j \rangle =0,\;\vee \\
      & \vartheta_j+c\alpha_j=0,(\mathbb K p)_j = c q_j (c\ge 0),\langle \mu_j,q_j \rangle\le 0.
    \end{aligned} \right\}  &&\text{if }j\in\Bcal,\\
    & \vartheta_j + \frac{\scalar{(\Kbb u)_j}{(\Kbb p)_j}}{\|(\Kbb u)_j\|} \le 0,\quad \mu_j = 0,&&\text{if }j\in\Ical_0,\\
    & \left.
    \begin{aligned}
      & \vartheta_j=0,(\Kbb p)_j = 0,\;\vee \\
      & \vartheta_j\le\|(\Kbb p)_j\|,\mu_j=0,\;\vee \\
      & |\vartheta_j|\le \|(\Kbb p)_j\|,\scalar{\mu_j}{(\Kbb p)_j}\le 0
    \end{aligned} \right\}  &&\text{if }j\in\Tcal.
  \end{aligned} \right.
  \right\}
\end{multline*}
\end{lemma}
\begin{proof}
  Let us recall the definition of the Mordukhovich normal cone for our problem
  \begin{multline*}
    N^M_{\gph Q}(\alpha,u,\Kbb^\top q) = \{(\vartheta,\Kbb^\top \mu,p):(\vartheta_k,\Kbb^\top \mu_k,p_k)\in N^F_{\gph Q}(\alpha_k,u_k,\Kbb^\top q_k):\\(\vartheta_k,\Kbb^\top \mu_k,p_k)\to (\vartheta,\Kbb^\top \mu,p),(\alpha_k,u_k,\Kbb^\top q_k)\to (\alpha,u,\Kbb^\top q)\}.
  \end{multline*}
  Considering limiting sequences to the inactive, strongly active and zero-inactive sets, the same directions as for the Fréchet normal cone are obtained. The differences lie in the biactive and triactive sets, where several approximations may be considered.
  \begin{description}
    \item[Case 1: $\bm{j\in\Bcal.}$]
      By taking approximation sequences in the inactive set, from \cref{lemma:frechet_cone} we know
      \begin{equation}\label{eq:m_cone_b0_6}
        0 = (\mu_k)_j + (\alpha_k)_j\frac{(\Kbb p_k)_j}{\|(\Kbb u_k)_j\|} - (\alpha_k)_j\frac{(\Kbb u_k)\scalar{((\Kbb u_k)_j)}{(\Kbb p_k)_j}}{\|(\Kbb u_k)_j\|^3}.
      \end{equation}
      Multiplying \cref{eq:m_cone_b0_6} with $(\Kbb p_k)_j$ yields
      \begin{equation*}
        \scalar{(\mu_k)_j}{(\Kbb p_k)_j} = (\alpha_k)_j\frac{\scalar{(\Kbb u_k)_j}{(\Kbb p_k)_j}^2}{\|(\Kbb u_k)_j\|^3} - (\alpha_k)_j\frac{\scalar{(\Kbb p_k)_j}{(\Kbb p_k)_j}}{\|(\Kbb u_k)_j\|}
      \end{equation*}
      Again, multiplying by $(\alpha_k)_j\|(\Kbb u_k)_j\|$ on both sides and recalling that $(q_k)_j = (\alpha_k)_j\frac{(\Kbb u_k)_j}{\|(\Kbb u_k)_j\|}$, we get
      \begin{equation*}
        (\alpha_k)_j\|(\Kbb u_k)_j\|\scalar{(\mu_k)_j}{(\Kbb p_k)_j} = \scalar{(q_k)_j}{(\Kbb p_k)_j}^2 - (\alpha_k)_j^2\|(\Kbb p_k)_j\|^2.
      \end{equation*}
      Taking the limit as $k\to\infty$ and recalling $(\alpha_k)_j=\|(q_k)_j\|$ in this index set, we obtain
      \begin{equation*}
        \scalar{q_j}{(\Kbb p)_j}^2 = \|q_j\|^2\|(\Kbb p)_j\|^2,
      \end{equation*}
      which implies that $(\Kbb p)_j = cq_j(c\in\R)$. Now, multiplying \cref{eq:m_cone_b0_6} with $(q_k)_j$ we get the following product
      \begin{align*}
      \scalar{(\mu_k)_j}{(q_k)_j} &= (\alpha_k)_j\frac{\scalar{(q_k)_j}{(\Kbb u_k)_j}\scalar{(\Kbb u_k)_j}{(\Kbb p_k)_j}}{\|(\Kbb u_k)_j\|^3} - (\alpha_k)_j\frac{\scalar{(q_k)_j}{(\Kbb p_k)_j}}{\|(\Kbb u_k)_j\|},\\
      &= (\alpha_k)_j^2\frac{\scalar{(\Kbb u_k)_j}{(\Kbb p_k)_j}}{\|(\Kbb u_k)_j\|^2} - (\alpha_k)_j^2\frac{\scalar{(\Kbb u_k)_j}{(\Kbb p_k)_j}}{\|(\Kbb u_k)_j\|^2}=0.\numberthis \label{eq:m_cone_b0_7}
      \end{align*}
      Taking the limit we get that $\scalar{\mu_j}{q_j}=0$.

      Regarding $\vartheta_j$ we have $(\alpha_k)_j(\vartheta_k)_j + \scalar{(q_k)_j}{(\Kbb p)_j} = 0.$
      Taking the limit as $k\to \infty$ we obtain
      \begin{equation*}
        0 = \alpha_j\vartheta_j + \scalar{q_j}{(\Kbb p)_j}= \alpha_j\vartheta_j + c\|q_j\|^2,(c\in\R),
      \end{equation*}
      which implies that $\vartheta_j=-c\alpha_j\in\R$.

      Finally, when taking the approximation through the strongly active and biactive sets, the cone directions coincide with the Fréchet normal ones.

    \item[Case 2: $\bm{j\in\Tcal.}$]
    This index set can be approximated by sequences belonging either to the inactive, biactive, strongly active or zero-inactive sets.
    Considering strongly active sequences, $(\Kbb p_k)_j=0$ and $(\vartheta_k)_j=0$. Taking the limit as $k\to\infty$ we get $(\Kbb p)_j=0$ and $\vartheta_j=0$ as well.

    Likewise, when taking biactive sequences we get $(\vartheta_k)_j+c(\alpha_k)_j=0$, $(\Kbb p_k)_j=c(q_k)_j(c\ge 0)$ and $\scalar{(\mu_k)_j}{(q_k)_j}\le 0$. Again, taking the limit as $k\to\infty$ we get, since $q_j=0$ and $\alpha_j=0$, that $\mu_j\in\R^2$, $(\Kbb p)_j=0$ and $\vartheta_j=0$.

    Furthermore, taking sequences in the zero-inactive set we have $(\mu_k)_j=0$, which implies that $\mu_j=0$. Using the Cauchy-Schwarz inequality we get
    \begin{equation*}
      0\ge (\vartheta_k)_j + \frac{\scalar{(\Kbb u_k)_j}{(\Kbb p_k)_j}}{\|(\Kbb u_k)_j\|} \ge (\vartheta_k)_j -\|(\Kbb p_k)_j\|.
    \end{equation*}
    Taking the limit as $k\to\infty$ yields $\vartheta_j\le\|(\Kbb p)_j\|$.

    When taking inactive sequences, we know that
    \begin{align}
      (\mu_k)_j + (\alpha_k)_j\left(\frac{I}{\|(\Kbb u_k)_j\|}-\frac{(\Kbb u_k)_j(\Kbb u_k)_j^\top}{\|(\Kbb u_k)_j\|^3}\right)(\Kbb p_k)_j &= 0,\label{eq:m_cone_b0_1}\\
      (\vartheta_k)_j + \frac{\scalar{(\Kbb u_k)_j}{(\Kbb p_k)_j}}{\|(\Kbb u_k)_j\|}&= 0.\label{eq:m_cone_b0_2}
    \end{align}
    Applying Cauchy-Schwarz in \cref{eq:m_cone_b0_2}, yields $|(\vartheta_k)_j|\le\|(\Kbb p_k)_j\|.$
    Now, multiplying \cref{eq:m_cone_b0_1} with $(\Kbb p_k)_j$ we obtain
    \begin{equation}\label{eq:m_cone_b0_4}
      \scalar{(\mu_k)_j}{(\Kbb p_k)_j} \le (\alpha_k)_j\left(\frac{\|(\Kbb p_k)_j\|^2\|(\Kbb u_k)_j\|^2}{\|(\Kbb u_k)_j\|^3}-\frac{\|(\Kbb p_k)_j\|^2}{\|(\Kbb u_k)_j\|}\right)=0.
    \end{equation}
    Taking the limit as $k\to\infty$ we get $\scalar{\mu_j}{(\Kbb p)_j}\le 0$, finishing the proof.
  \end{description}
\end{proof}

\begin{theorem}[M-Stationarity] \label{thm:m stationarity}
  Let $J:\R^m\to\R$ be continuously differentiable, $\phi:\R^m\to\R$ twice continuously differentiable and strongly convex, and $(\alpha^*,u^*,q^*)$ be a local solution to \cref{eq:bilevel_problem_reformulated}. Then there exist KKT multipliers $(\vartheta,\Kbb^\top \mu,p)$ such that
  \begin{subequations}\label{eq:m_stationary}
    \begin{align}
      \phi'(u^*)+\Kbb^\top q^* = 0,\\
      \scalar{q^*_j}{(\Kbb u^*)_j}-\alpha^*_j\|(\Kbb u^*)_j\| =0, &&&\forall j=1,\dots,n,\\
      \|q^*_j\| \le \alpha_j^*, &&&\forall j=1,\dots,n,\\
      \phi''(u^*)^\top p -\Kbb^\top \mu - \nabla J(u^{*})=0,\\
      \vartheta + \rho =0,\\
      0 \leq \alpha^* \perp \rho \geq 0,\\
      (\vartheta,\Kbb^\top \mu,p)\in N_{\gph Q}^M(\alpha^*,u^*,\Kbb^\top q^*)
    \end{align}
  \end{subequations}
\end{theorem}
\begin{proof}
  Referring to \cref{teo:outrata_cq} let us take $F_1(\alpha,u)=\phi'(u)\in\R^m$ and $F_2(\alpha,u) = (\alpha,u)\in\R^n_+\times\R^m$. Existence of the KKT multipliers is guaranteed if the following constraint qualification condition holds for $(\vartheta,\Kbb^\top \mu,p)\in N_{\gph Q}^M(\alpha^*,u^*,\Kbb^\top q^*)$
  \begin{equation}\label{eq:outrata_optimality_system_proof_1}
    \begin{bmatrix}
      I & \bm{0} & \bm{0} \\
      \bm{0} & I & -\phi''(u^*)^\top
    \end{bmatrix} \begin{bmatrix}
      \vartheta\\
      \Kbb^\top \mu\\
      p
    \end{bmatrix}\in -N^M_{\R^n_+}(\alpha^*)\times\{0\} \text{ implies }\vartheta = 0,\;\Kbb^\top\mu = 0,\;p = 0.
  \end{equation}
  Recalling \cref{rem:mordukhovich_convex} and using the expression of the Mordukhovich normal cone $N_{R^n_+}^M(\alpha^*)=N_{R^n_+}(\alpha^*)=\{v\in\R^2:\scalar{v}{\alpha^*} = 0,\;v\le 0\}$, condition \cref{eq:outrata_optimality_system_proof_1} can also be written as
  \begin{align}
    \Kbb^\top \mu - \phi''(u^*)^\top p &= 0,\label{eq:constraint_qualification_1}\\
    \scalar{\alpha^*}{\vartheta} &= 0,\label{eq:constraint_qualification_2}\\
    \vartheta &\ge 0.\label{eq:constraint_qualification_3}
  \end{align}
  Let us take $(\vartheta,\Kbb^\top\mu,p)\in N^M_{\gph Q}(\alpha^*,u^*,\Kbb^\top q^*)$ and let us multiply \cref{eq:constraint_qualification_1} by $p$ on the left. Recalling $(\Kbb p)_j=0$ in $\Acal_s$ and $\mu_j=0$ in $\Ical_0$, we have for each remaining index set
  \begin{align*}
    \scalar{p}{\phi''(u^*)^\top p}&=\scalar{p}{\Kbb^\top\mu} = \sum_{j\in\Ical}\scalar{\mu_j}{(\Kbb p)_j} + \sum_{j\in\Bcal}\scalar{\mu_j}{(\Kbb p)_j} + \sum_{j\in\Bcal_0}\scalar{\mu_j}{(\Kbb p)_j},\\
    &= \sum_{j\in\Ical}-\alpha_j\scalar{(\Kbb p)_j}{T_j(\Kbb p)_j} + \sum_{j\in\Bcal}c\underbrace{\scalar{\mu_j}{q_j}}_{\le 0} + \sum_{j\in\Bcal_0}\underbrace{\scalar{\mu_j}{(\Kbb p)_j}}_{\le 0} \le 0,
  \end{align*}
  where we used the characterization of the Mordukhovich normal cone. Furthermore, using the strong convexity of the function $\phi$ we have $\scalar{p}{\phi''(u^*)^\top p}\ge 0$. Both inequalities imply $p=0$ and, according to \cref{eq:constraint_qualification_1}, it also yields $\Kbb^\top\mu = 0$.
  Moreover, if we consider the index set $\Ical\cup\Acal_s\cup\Bcal$, we know in all these sets $\alpha^*_j>0$, and therefore, to satisfy equation \cref{eq:constraint_qualification_2} it must hold $\vartheta_j=0$. Since $p=0$ in $\Tcal$ we know $\vartheta_j = 0$ or $\vartheta_j \le \|(\Kbb p)_j\|$ for this index set, in both cases it leads to $\vartheta_j=0$. In $\Ical_0$, we have $\vartheta_j \le - \frac{\scalar{(\Kbb u)_j}{(\Kbb p)_j}}{\|(\Kbb u)_j\|} = 0$ and \cref{eq:constraint_qualification_3} yields $\vartheta_j=0$. Therefore, $\vartheta_j = 0$ for all $j$. Consequently, the existence of multipliers is guaranteed and there exists a vector $\rho\in N^M_{R^n_+}(\alpha^*)$ and KKT multipliers $(\vartheta,\Kbb^\top\mu,p)\in N_{\gph Q}^M(\alpha^*,u^*,\Kbb^\top q^*)$ such that
  \begin{align}
    0 &= \nabla_u J(u^*) + (\nabla_u F_2(\alpha,u))^\top \begin{bmatrix}
      \vartheta\\
      \Kbb^\top\mu
    \end{bmatrix} - (\nabla_u F_1(\alpha,u))^\top p,\\
    0 &= (\nabla_\alpha F_2(\alpha,u))^\top \begin{bmatrix}
      \vartheta\\
      \Kbb^\top\mu
    \end{bmatrix}- (\nabla_\alpha F_1(\alpha,u))^\top p + \rho.
  \end{align}
\end{proof}


%% file: sections/bouligand_stationarity.tex
\section{Bouligand Stationarity}\label{sec:bouligand_stationarity}
In this section we will study the Bouligand stationarity condition for \cref{eq:bilevel_problem}. With this goal in mind, let us introduce the solution operator for the lower-level problem $S:\R^n_+\ni\alpha\to u\in\R^m$ that maps each parameter $\alpha\in\R^n_+$ to the corresponding reconstruction $u\in\R^n$. If this mapping is bijective, we can make use of it to formulate \cref{eq:bilevel_problem} as a reduced optimization problem
\begin{mini}
  {\alpha\in\R^n_+}{j(\alpha) := J(S(\alpha)).}{\label{eq:reduced_cost_function}}{}
\end{mini}
Furthermore, if the solution operator is Bouligand (B)-differentiable, we can make use of the chain rule for B-differentiable functions to conclude that the composite mapping $J$, as a function of $\alpha$, is B-differentiable as well. In this case, its directional derivative in a direction $h$ is given by
\begin{equation}\label{eq:dir_derivative_reduced_cost}
  j'(\alpha;h) = \scalar{\nabla J(u)}{S'(\alpha;h)},
\end{equation}
where $S'(\alpha;h)$ is the directional derivative of the solution operator in direction $h$.
Moreover, if $\alpha^*$ is a local optimal solution and $u^*=S(\alpha^*)$ its corresponding reconstruction, then it satisfies the following necessary condition
\begin{equation}\label{eq:b_stationary_condition}
  j'(\alpha^*;\alpha-\alpha^*) = \scalar{\nabla J(u^*)}{S'(\alpha^*;\alpha-\alpha^*)}\ge 0,\;\forall \alpha\in\R^n_+.
\end{equation}

A point $\alpha^*$ satisfying the necessary condition \cref{eq:b_stationary_condition} is called \textit{Bouligand (B)-stationary}. This type of stationarity condition is based on the tangent cone to our feasible parameter set and can be interpreted as the counterpart of the implicit programming approach in the discussion of finite-dimensional MPECs, see \cite[Lemma 4.2.5]{luo1996mathematical}.

To obtain such B-stationarity condition for our problem, a sensitivity analysis of the solution mapping must be carried out, to prove that it is indeed Bouligand differentiable, i.e., locally Lipschitz continuous and directionally differentiable. Let us recall that in \Cref{sec:problem_formulation} we already argued the existence and uniqueness of the solution to the lower-level problem \cref{eq:bilevel_problem_lower}, implying that $S:\R^n_+\to \R^m$ is singled valued.

%
%

\begin{theorem}\label{teo:solution_operator_lipschitz}
  The solution operator for the lower-level problem \cref{eq:bilevel_problem_lower} $S:\R^n_+\ni\alpha\to u\in\R^m$ is Lipschitz continuous.
\end{theorem}
\begin{proof}
  Thanks to \Cref{teo:unique_solution_vi}, we know the lower-level problem has a unique solution. Moreover, $\alpha_1,\alpha_2\in \R^n_+$ and its corresponding solutions $u_1,u_2$ satisfy
  \begin{align*}
    \scalar{\phi'(u_1)}{v-u_1}+\sum_{j=1}^n(\alpha_1)_j\|(\Kbb v)_j\|-\sum_{j=1}^n(\alpha_1)_j\|(\Kbb u_1)_j\|&\ge 0,\;\forall v\in\R^m\\
    \scalar{\phi'(u_2)}{w-u_2}+\sum_{j=1}^n(\alpha_2)_j\|(\Kbb w)_j\|-\sum_{j=1}^n(\alpha_2)_j\|(\Kbb u_2)_j\|&\ge 0,\;\forall w\in\R^m.
  \end{align*}
  Taking in particular $v=u_2$ and $w=u_1$ and adding the inequalities, it yields
  \begin{equation*}
    \scalar{\phi'(u_2)-\phi'(u_1)}{u_2-u_1} \le \sum_{j=1}^n((\alpha_1)_j-(\alpha_2)_j)(\|(\Kbb u_2)_j\|-\|(\Kbb u_1)_j\|),
  \end{equation*}
  Moreover, given that $\phi$ is strongly convex and using the Cauchy-Schwarz inequality, it yields
  \begin{align*}
    c\|u_2-u_1\|^2 &\le \sum_{j=1}^n(\alpha_{2,j}-\alpha_{1,j})\|(\Kbb (u_2-u_1))_j\| \le \|\alpha_2-\alpha_1\| \sum_{j=1}^n\|(\Kbb (u_2-u_1))_j\|,\\
    &\le \|\alpha_2-\alpha_1\|\|\Kbb\|\|u_2-u_1\|,
  \end{align*}
  where $\|\Kbb\|$ is the operator norm of the linear operator $\Kbb$.
\end{proof}

\subsection{Directional Differentiability}
Now, we are interested in the differentiability properties of the solution operator for the lower-level problem \cref{eq:bilevel_problem_lower}. This will require a sensitivity analysis of the solution operator with respect to the regularization parameter.
By taking a perturbed regularization parameter $\alpha^t$ in the primal-dual formulation for the lower-level problem \cref{eq:fenchel_dual_lower_level} such that $\alpha_j^t=\alpha_j+th_j\ge 0$ we get the following perturbed lower-level problem
\begin{subequations}\label{eq:perturbed_lower_level}
  \begin{align}
    \phi'(u^t) + \Kbb^\top q^t &= 0,\\
    \scalar{q^t_j}{(\Kbb u^t)_j} - (\alpha_j+t h_j)\|(\Kbb u^t)_j\|&=0,\;&\forall j=1,\dots,n,\\
    \|q^t_j\|-(\alpha_j+t h_j)&\le 0,\;&\forall j=1,\dots,n.
  \end{align}
\end{subequations}
Thanks to the local Lipschitz continuity of the solution operator proved in \cref{teo:solution_operator_lipschitz} and the boundedness of $q^t$, see \cref{eq:perturbed_lower_level}, there exist a subsequence, denoted the same, so that $q^t\to\tilde{q}\in\R^{n\times 2}$, to some $\tilde{q}$. Additionally, we can guarantee the existence of a subsequence of $u^t$, denoted w.l.o.g. with the same symbol, satisfying the following limit
\begin{equation}\label{eq:eta_definition}
  \lim_{t\to 0} \frac{u^t-u}{t} \to \eta\in\R^m.
\end{equation}

\begin{theorem}
  The limit described in \cref{eq:eta_definition} satisfies $\eta \in \Ccal = \Ccal(\alpha,u)$ where
  \begin{equation}\label{eq:ccal}
    \Ccal(\alpha,u):= \left\{v\in\R^n:\begin{cases}
      (\Kbb v)_j = 0,&\forall j\in\Acal_s,\\
      \scalar{q_j}{(\Kbb v)_j}=\alpha_j\|(\Kbb v)_j\|,&\forall j\in\Bcal.
    \end{cases}
    \right\}
  \end{equation}
\end{theorem}
\begin{proof}
  By adding the complementarity relationships in  \cref{eq:perturbed_lower_level,eq:fenchel_dual_lower_level}, and dividing by $t$, we get
  \begin{equation*}
    \bigscalar{\frac{q^t_j-q_j}{t}}{(\Kbb u)_j} + \bigscalar{q_j^t}{\frac{(\Kbb u^t)_j-(\Kbb u)_j}{t}} -\alpha_j\left(\frac{\|(\Kbb u^t)_j\|-\|(\Kbb u)_j\|}{t}\right) - h_j\|(\Kbb u^t)_j\| = 0.
  \end{equation*}
  For $j\in\Acal_s\cup \Bcal$, taking the limit as $t\to 0$ and using the boundedness of the sequence $q^t$ along with the Bouligand differentiability of the Euclidean norm, it yields
  \begin{equation*}\label{eq:teo3_3_aux1}
    \scalar{\tilde{q}_j}{(\Kbb\eta)_j} - \alpha_j\|(\Kbb \eta)_j\| - h_j\|(\Kbb u)_j\| = 0.
  \end{equation*}
  Since $(\Kbb u)_j=0$ for $j\in\Acal_s\cup\Bcal$, we get that
  \begin{equation*}
    \scalar{\tilde{q}_j}{(\Kbb\eta)_j} - \alpha_j\|(\Kbb \eta)_j\|=0.
  \end{equation*}
  Moreover, for $j\in\Acal_s$, recalling $\alpha_j>0$ in this index set, we get
  \begin{equation*}
    \alpha_j\|(\Kbb \eta)_j\| = \scalar{\tilde{q}_j}{(\Kbb\eta)_j} \le \|\tilde{q}_j\|\|(\Kbb \eta)_j\| < \alpha_j\|(\Kbb\eta)_j\|,
  \end{equation*}
  which only holds if $(\Kbb \eta)_j= 0$ in this index set, finishing the proof.
\end{proof}

\begin{remark}\label{rem:Cu_independent_from_q}
  If $q^1$ and $q^2$ are two different slack variables associated with the solution $u$ in \cref{eq:fenchel_dual_lower_level}, then the two sets
  \begin{equation*}
    \Ccal_i := \left\{v\in\R^n:\begin{cases}
      (\Kbb v)_j=0,&\text{ if }\|q_j^i\| < \alpha_j,\\
      \scalar{q_j^i}{(\Kbb v)_j}=\alpha_j\|(\Kbb v)_j\|,&\text{ if }\;(\Kbb u)_j=0,\;\alpha_j>0,\;\|q_j^i\|=\alpha_j.
    \end{cases}\right\},
  \end{equation*}
  coincide, since $\Kbb^\top q^1 = -\phi'(u) = \Kbb^\top q^2$. As a consequence, the set $\Ccal(\alpha,u)$ does not depend on the slack variable, only on the solution $u$ and the parameter $\alpha$.
\end{remark}

\begin{lemma}
  The cone $\Ccal(\alpha,u)$ can alternatively be written as
  \begin{equation}\label{eq:ccal_cone}
    \Ccal(\alpha,u) = \left\{v\in\R^n:\scalar{\Kbb^\top q}{v}\ge\sum_{j\in\Ical}\bigscalar{\alpha_j\frac{(\Kbb u)_j}{\|(\Kbb u)_j\|}}{(\Kbb v)_j} + \sum_{j\in\Bcal}\alpha_j\|(\Kbb v)_j\|\right\}
  \end{equation}
\end{lemma}
\begin{proof}
  Let us denote the set in \cref{eq:ccal_cone} as $\mathcal{M}$. Taking $v\in \Ccal$, as in \cref{eq:ccal}, and using its definition, we obtain
  \begin{align*}
    \scalar{\Kbb^\top q}{v} &= \sum_{j\in\Ical}\scalar{q_j}{(\Kbb v)_j} + \sum_{j\in\Acal_s}\scalar{q_j}{(\Kbb v)_j} + \sum_{j\in\Bcal}\scalar{q_j}{(\Kbb v)_j},\\
    &= \sum_{j\in\Ical}\bigscalar{\alpha_j\frac{(\Kbb u)_j}{\|(\Kbb u)_j\|}}{(\Kbb v)_j} + \sum_{j\in\Acal_s} \scalar{q_j}{\underbrace{(\Kbb v)_j}_{=0}} + \sum_{j\in\Bcal}\alpha_j\|(\Kbb v)_j\|,
  \end{align*}
  and, consequently, $\Ccal\subset\mathcal{M}$.

  To prove the reverse inclusion, let us take $v\in\mathcal{M}$. For $j\in\Acal_s\cup\Bcal$ the following relation holds true
  \begin{equation}\label{eq:ccal_equivalence_1}
    \sum_{j\in\Bcal}\alpha_j\|(\Kbb v)_j\| \le \sum_{j\in\Acal_s\cup\Bcal}\scalar{q_j}{(\Kbb v)_j}\le \sum_{j\in\Acal_s\cup\Bcal}\alpha_j\|(\Kbb v)_j\|.
  \end{equation}
  where we used the Cauchy-Schwarz inequality and $\|q_j\|\le \alpha_j$. Regarding the left inequality in \cref{eq:ccal_equivalence_1} we know
  \begin{equation}\label{eq:ccal_equivalence_2}
    \sum_{j\in\Acal_s\cup\Bcal}\scalar{q_j}{(\Kbb v)_j} - \sum_{j\in\Bcal}\alpha_j\|(\Kbb v)_j\|\ge 0,
  \end{equation}
  using the Cauchy-Schwarz inequality we also know
  \begin{equation}\label{eq:ccal_equivalence_3}
    \sum_{j\in\Acal_s\cup\Bcal}\scalar{q_j}{(\Kbb v)_j} - \sum_{j\in\Acal_s\cup\Bcal}\alpha_j\|(\Kbb v)_j\|\le 0.
  \end{equation}
  Multiplying \cref{eq:ccal_equivalence_2} by $-1$ and adding it to \cref{eq:ccal_equivalence_3} we get $\sum_{j\in\Acal_s}\alpha_j\|(\Kbb v)_j\| = 0,$ which implies $(\Kbb v)_j=0$, for all $j\in\Acal_s$. Now, using this result in \cref{eq:ccal_equivalence_1} we get
  \begin{equation*}
    \sum_{j\in\Bcal}(\alpha_j\|(\Kbb v)_j\| - \scalar{q_j}{(\Kbb v)_j})=0
  \end{equation*}
  and, consequently, it holds $\scalar{q_j}{(\Kbb v)_j}=\alpha_j\|(\Kbb v)_j\|$.
\end{proof}

Now, to prove the directional differentiability of the solution operator for the lower-level problem \cref{eq:bilevel_problem_lower} we will first demonstrate the following lemmata.

\begin{lemma}\label{lemma:bound_1}
  Let $\R^n_+\ni\alpha\ge 0$ and $\R^n_+\ni\alpha+th\ge 0$. Then for every $v\in\Ccal$ it holds
  \begin{multline}
    \bigscalar{\Kbb^\top\left(\frac{q^t-q}{t}\right)}{v}\le \sum_{j\in\Ical}\frac{\alpha_j}{t}\bigscalar{\frac{(\Kbb u^t)_j}{\|(\Kbb u^t)_j\|}-\frac{(\Kbb u)_j}{\|(\Kbb u)_j\|}}{(\Kbb v)_j} \\+ \sum_{j\in\Ical} h_j\bigscalar{\frac{(\Kbb u^t)}{\|(\Kbb u^t)_j\|}}{(\Kbb v)_j}+\sum_{j\in\Bcal}h_j\|(\Kbb v)_j\| + \sum_{j\in\Tcal\cup\Ical_0}h_j\|(\Kbb v)_j\|.
  \end{multline}
\end{lemma}
\begin{proof}
  Given that $v\in\Ccal$, let us first bound the following product
  \begin{align*}
    \scalar{\Kbb^\top q^t}{v} &= \sum_{j\in\Ical}\scalar{q_j^t}{(\Kbb v)_j} + \sum_{j\in\Acal_s}\scalar{q_j^t}{\underbrace{(\Kbb v)_j}_{=0}} + \sum_{j\in\Bcal}\scalar{q_j^t}{(\Kbb v)_j} + \sum_{j\in\Tcal\cup\Ical_0}\scalar{q_j^t}{(\Kbb v)_j},\\
    &\le\sum_{j\in\Ical}\bigscalar{(\alpha_j+t h_j)\frac{(\Kbb u^t)_j}{\|(\Kbb u^t)_j\|}}{(\Kbb v)_j} + \sum_{j\in\Bcal}(\alpha_j+t h_j)\|(\Kbb v)_j\| + \sum_{j\in\Tcal\cup\Ical_0}t h_j\|(\Kbb v)_j\|,
  \end{align*}
  for $t$ sufficiently small, since $u^t\to u$ implies $\Ical(\alpha,u)\subset\Ical(\alpha+th,u^t)$, where we used the property $(\Kbb v)_j=0$ for $j\in\Acal_s$ and $\alpha_j = 0$ for $j\in\Tcal\cup\Ical_0$, along with Cauchy-Schwarz inequality and $\|q_j^t\|\le \alpha_j+th_j$. Now, as $v\in\Ccal$ we know the bound in \cref{eq:ccal_cone} holds, i.e.,
  \begin{align*}
    \scalar{\Kbb^\top q}{v}\ge\sum_{j\in\Ical}\bigscalar{\alpha_j\frac{(\Kbb u)_j}{\|(\Kbb u)_j\|}}{(\Kbb v)_j} + \sum_{j\in\Bcal}\alpha_j\|(\Kbb v)_j\|.
  \end{align*}
  Therefore,
  \begin{multline*}
    \scalar{\Kbb^\top(q^t-q)}{v} \le \sum_{j\in\Ical}\alpha_j\bigscalar{\frac{(\Kbb u^t)_j}{\|(\Kbb u^t)_j\|}-\frac{(\Kbb u)_j}{\|(\Kbb u)_j\|}}{(\Kbb v)_j} \\+\sum_{j\in\Ical}t h_j \bigscalar{\frac{(\Kbb u^t)_j}{\|(\Kbb u^t)_j\|}}{(\Kbb v)_j} + \sum_{j\in\Bcal}t h_j\|(\Kbb v)_j\|+ \sum_{j\in\Tcal\cup\Ical_0}t h_j\|(\Kbb v)_j\|.
  \end{multline*}
  Finally, dividing both sides by $t$ yields the result.
\end{proof}
\begin{lemma}\label{lemma:bound_2}
  Let $\R^n\ni\alpha\ge 0$ and $\R^n\ni\alpha+th\ge 0$. Then, it holds
  \begin{multline*}
    \bigscalar{\Kbb^\top\left(\frac{q^t-q}{t}\right)}{\frac{u^t-u}{t}} \ge \sum_{j\in\Ical} \frac{\alpha_j}{t}\bigscalar{\frac{(\Kbb u^t)_j}{\|(\Kbb u^t)_j\|}-\frac{(\Kbb u)_j}{\|(\Kbb u)_j\|}}{\frac{(\Kbb u^t)_j-(\Kbb u)_j}{t}} \\
    +\sum_{j\in\Ical} h_j\bigscalar{\frac{(\Kbb u^t)}{\|(\Kbb u^t)_j\|}}{\frac{(\Kbb u^t)_j-(\Kbb u)_j}{t}} + \frac{1}{t}\sum_{j\in\Acal\cup\Ical_0}h_j(\|(\Kbb u^t)_j\|-\|(\Kbb u)_j\|),
  \end{multline*}
  where $\Acal = \Acal_s\cup\Bcal\cup\Tcal$.
\end{lemma}
\begin{proof}
  For $t$ small enough, we can split the product by their index set
  \begin{multline*}
    \bigscalar{\Kbb^\top\left(\frac{q^t-q}{t}\right)}{\frac{u^t-u}{t}} =\\ \sum_{j\in\Ical} \frac{\alpha_j}{t}\bigscalar{\frac{(\Kbb u^t)_j}{\|(\Kbb u^t)_j\|}-\frac{(\Kbb u)_j}{\|(\Kbb u)_j\|}}{\frac{(\Kbb u^t)_j-(\Kbb u)_j}{t}}
    +\sum_{j\in\Ical}h_j\bigscalar{\frac{(\Kbb u^t)_j}{\|(\Kbb u^t)_j\|}}{\frac{(\Kbb u^t)_j-(\Kbb u)_j}{t}}\\+\frac{1}{t}\sum_{j\in\Acal_s\cup\Bcal} \bigscalar{q^t_j-q_j}{\frac{(\Kbb u^t)_j-(\Kbb u)_j}{t}}+\frac{1}{t}\sum_{j\in\Ical_0\cup\Tcal} \bigscalar{q^t_j-q_j}{\frac{(\Kbb u^t)_j-(\Kbb u)_j}{t}}.
  \end{multline*}
  Focusing, on the index set $\Acal_s\cup\Bcal$, the complementarity relations in \cref{eq:perturbed_lower_level,eq:fenchel_dual_lower_level} yield
  \begin{align*}
    &\frac{1}{t^2}\sum_{j\in\Acal_s\cup\Bcal}\scalar{q^t_j-q_j}{(\Kbb u^t)_j-(\Kbb u)_j}\\
    &=\frac{1}{t^2} \sum_{j\in\Acal_s\cup\Bcal}\scalar{q^t_j}{(\Kbb u^t)_j} - \scalar{q_j^t}{(\Kbb u)_j} - \scalar{q_j}{(\Kbb u^t)_j} + \scalar{q_j}{(\Kbb u)_j},\\
    &\ge \frac{1}{t^2} \sum_{j\in\Acal_s\cup\Bcal}(\alpha_j+t h_j)\|(\Kbb u^t)_j\| - \underbrace{\|q_j^t\|}_{\le \alpha_j+t h_j}\|(\Kbb u)_j\| - \underbrace{\|q_j\|}_{\le \alpha_j}\|(\Kbb u^t)_j\| + \alpha_j\|(\Kbb u)_j\|,\\
    &\ge \frac{1}{t}\sum_{j\in\Acal_s\cup\Bcal} h_j (\|(\Kbb u^t)_j\|-\|(\Kbb u)_j\|).
  \end{align*}
  Using the same analysis over the set $\Ical_0\cup\Tcal$ we get
  \begin{align*}
    &\frac{1}{t^2}\sum_{j\in\Ical_0\cup\Tcal}\scalar{q^t_j-q_j}{(\Kbb u^t)_j-(\Kbb u)_j} =\frac{1}{t^2} \sum_{j\in\Ical_0\cup\Tcal}\scalar{q^t_j}{(\Kbb u^t)_j} - \scalar{q_j^t}{(\Kbb u)_j},\\
    &\ge \frac{1}{t^2} \sum_{j\in\Ical_0\cup\Tcal}t h_j\|(\Kbb u^t)_j\| - \underbrace{\|q_j^t\|}_{\le t h_j}\|(\Kbb u)_j\| \ge \frac{1}{t}\sum_{j\in\Ical_0\cup\Tcal} h_j (\|(\Kbb u^t)_j\|-\|(\Kbb u)_j\|).
  \end{align*}

\end{proof}

\begin{theorem} \label{thm: directional derivative}
  Let $\alpha\in\R^n_+$ and $h\in\R^n$ be a direction such that $\alpha+th\ge 0$ for $t$ small enough. The solution operator $S:\alpha\to S(\alpha)=u\in\R^m$ is directionally differentiable and its directional derivative $\eta\in\Ccal(\alpha,u)$ at $u$, in direction $h$, is given by the solution of the following variational inequality
    \begin{multline}\label{eq:directional_derivative_vi}
      \scalar{\phi''(u)\eta}{v-\eta} + \sum_{j\in\Ical}\alpha_j\scalar{T_j(\Kbb \eta)_j}{(\Kbb v)_j-(\Kbb\eta)_j} + h_j\bigscalar{\frac{(\Kbb u)_j}{\|(\Kbb u)_j\|}}{(\Kbb v)_j-(\Kbb \eta)_j} \\+ \sum_{j\in\Bcal} \frac{h_j}{\alpha_j}\scalar{q_j}{(\Kbb v)_j-(\Kbb\eta)_j} + \sum_{j\in\Ical_0\cup\Tcal} h_j (\|(\Kbb v)_j\|-\|(\Kbb\eta)_j\|)\ge 0,\;\forall v\in\Ccal,
    \end{multline}
  where $T_j(\Kbb v)_j = \frac{(\Kbb v)_j}{\|(\Kbb u)_j\|} - \frac{(\Kbb u)_j(\Kbb u)_j^\top(\Kbb v)_j}{\|(\Kbb u)_j\|^3}$ for $v\in\R^m$.
\end{theorem}
\begin{proof}
  To verify the variational inequality, let us take \cref{eq:perturbed_lower_level}, \cref{eq:fenchel_dual_lower_level} and test them with $v-\frac{u^t-u}{t}$, with $v\in\Ccal(\alpha,u)$
  \begin{align*}
    0&=\bigscalar{\frac{\phi'(u^t)-\phi'(u)}{t}}{v-\frac{u^t-u}{t}} + \bigscalar{\Kbb^\top\left(\frac{q^t-q}{t}\right)}{v-\frac{u^t-u}{t}},\\
    &=\bigscalar{\frac{\phi'(u^t)-\phi'(u)}{t}}{v-\frac{u^t-u}{t}} + \bigscalar{\Kbb^\top\left(\frac{q^t-q}{t}\right)}{v} - \bigscalar{\Kbb^\top\left(\frac{q^t-q}{t}\right)}{\frac{u^t-u}{t}}
  \end{align*}
  Now, applying the bounds in \cref{lemma:bound_1,lemma:bound_2} we have
  \begin{multline*}
    0\le\bigscalar{\frac{\phi'(u^t)-\phi'(u)}{t}}{v-\frac{u^t-u}{t}}\\ + \sum_{j\in\Ical}\frac{\alpha_j}{t}\bigscalar{\frac{(\Kbb u^t)_j}{\|(\Kbb u^t)_j\|}-\frac{(\Kbb u)_j}{\|(\Kbb u)_j\|}}{(\Kbb v)_j - \frac{(\Kbb u^t)_j-(\Kbb u)_j}{t}} \\+ h_j\bigscalar{\frac{(\Kbb u^t)}{\|(\Kbb u^t)_j\|}}{(\Kbb v)_j- \frac{(\Kbb u^t)_j-(\Kbb u)_j}{t}} + \sum_{j\in\Bcal} h_j\|(\Kbb v)_j\|\\ + \sum_{j\in\Ical_0\cup\Tcal} h_j\|(\Kbb v)_j\| - \frac{1}{t}\sum_{j\in\Acal_s\cup\Bcal\cup\Ical_0\cup\Tcal}h_j(\|(\Kbb u^t)_j\|-\|(\Kbb u)_j\|).
  \end{multline*}
  Taking the limit $t\to 0$, as well as, the differentiability of the term $x/\|x\|$ in the inactive set, and given that $(\Kbb\eta)_j=0$ in the strongly active set $\Acal_s$, it yields

  \begin{multline*}
    0 \le \scalar{\phi''(u)\eta}{v-\eta} + \sum_{j\in\Ical}\alpha_j\bigscalar{\left(\frac{I}{\|(\Kbb u)_j\|} - \frac{(\Kbb u)_j(\Kbb u)_j^\top}{\|(\Kbb u)_j\|^3}\right)(\Kbb\eta)_j}{(\Kbb v)_j-(\Kbb\eta)_j} \\+ h_j\bigscalar{\frac{(\Kbb u)_j}{\|(\Kbb u)_j\|}}{(\Kbb v)_j - (\Kbb \eta)_j} + \sum_{j\in\Bcal}h_j(\|(\Kbb v)_j\|-\|(\Kbb \eta)_j\|) + \sum_{j\in\Ical_0\cup\Tcal}h_j(\|(\Kbb v)_j\| - \|(\Kbb \eta)_j\|).
  \end{multline*}
  Using the definition for $T_j$ and recalling $v,\eta\in\Ccal$, the inequality takes the form in \cref{eq:directional_derivative_vi}.

  Now it is required to verify the uniqueness of the limit. For this purpose, let us note that \cref{eq:directional_derivative_vi} is a variational inequality
  \begin{multline*}
    \scalar{\phi''(u)\eta}{v-\eta} + \sum_{j\in\Ical}\alpha_j\scalar{T_j(\Kbb\eta)_j}{(\Kbb v)_j-(\Kbb\eta)_j} + \sum_{j\in\Bcal}\frac{h_j}{\alpha_j}\scalar{q_j}{(\Kbb v)_j-(\Kbb\eta)_j}\\ + \sum_{j\in\Ical_0\cup\Tcal}h_j(\|(\Kbb v)_j\|-\|(\Kbb \eta)_j\|) \ge  -\sum_{j\in\Ical} h_j\bigscalar{\frac{(\Kbb u)_j}{\|(\Kbb u)_j\|}}{(\Kbb v)_j - (\Kbb \eta)_j},\forall v\in\Ccal
  \end{multline*}
  Now, given that the function $f(z) := \sum_{j=1}^n \|(\Kbb z)_j\|$ is indeed convex, lower semicontinuous and proper, the right hand side is continuous and linear, and finally, the bilinear form in the left hand side is V-elliptic, we know by \cite{glowinski1985numerical}, that there exists a unique solution for this variational inequality.
\end{proof}

Using the demonstrated Bouligand differentiability of the solution operator, and the cooresponding  characterization of the directional derivative described in this section, we have proven the following result.
\begin{theorem}
  Let $\alpha^*\in\R^n_+$ be a local optimal solution of \cref{eq:reduced_cost_function} and $u^*=S(\alpha^*)$. Then $\alpha^*$ is a B-stationary point, i.e., it satisfies the following inequality
  \begin{equation}\label{eq:b_stationarity}
    \scalar{\nabla J(u^*)}{S'(\alpha^*; \alpha- \alpha^*)}\ge 0, \quad \forall \alpha \in \R^n_+,
  \end{equation}
  where $S'(\alpha^*; \alpha- \alpha^*)=: \eta$ is the unique solution to \Cref{eq:directional_derivative_vi}.
\end{theorem}

\subsection{Strict Complementarity}
The characterization of the directional differentiability can take different formulations if any of the active sets becomes empty. For instance, assuming the zero-inactive and triactive sets empty, i.e., $\Ical_0\cup\Tcal=\emptyset$, then the directional derivative of the solution operator can be written as the following variational inequality of the first kind
\begin{multline}\label{eq:dir_der_vi_1st_kind}
  \scalar{\phi''(u)\eta}{v-\eta} + \sum_{j\in\Ical}\alpha_j\scalar{T_j(\Kbb \eta)_j}{(\Kbb v)_j-(\Kbb\eta)_j} \\ + h_j\bigscalar{\frac{(\Kbb u)_j}{\|(\Kbb u)_j\|}}{(\Kbb v)_j-(\Kbb \eta)_j} + \sum_{j\in\Bcal} \frac{h_j}{\alpha_j} (\scalar{q_j}{(\Kbb v)_j-(\Kbb \eta)_j})\ge 0,\;\forall v\in\Ccal.
\end{multline}
Furthermore, assuming an empty biactive set and $\alpha_j=0$, for all $j$, we obtain that the solution operator is Fréchet differentiable as stated in the following theorem.

\begin{theorem}
  Let us assume the index set $\Bcal \cup \Ical_0 \cup \Tcal$ is empty. Then, the solution operator is Fréchet differentiable and the derivative can be computed as the solution of the following system of equations
  \begin{subequations}\label{eq:sd_linearized_system}
    \begin{align}
      \phi''(u)\eta + \Kbb^\top \lambda &= 0,\\
      \lambda_j - \alpha_jT_j(\Kbb\eta)_j - \frac{h_j}{\alpha_j}q_j &= 0,\;&&\forall j\in\Ical,\\
      (\Kbb\eta)_j &= 0,\;&&\forall j\in\Acal_s.
    \end{align}
  \end{subequations}
\end{theorem}
\begin{proof}
  Using the empty biactive set assumption, we get that the cone $\Ccal$ becomes the following linear subspace $\Ccal(\alpha,u) = \{v\in\R^m:(\Kbb v)_j=0 \text{ if } (\Kbb u)_j=0\}$. Thus, the variational inequality in \cref{eq:dir_der_vi_1st_kind} becomes the following variational equation
  \begin{equation}\label{eq:strict_comp_var_eq}
    \scalar{\phi''(u)\eta}{v-\eta} + \sum_{j\in\Ical}\alpha_j\scalar{T_j(\Kbb \eta)_j}{(\Kbb v)_j-(\Kbb\eta)_j} + h_j\bigscalar{\frac{(\Kbb u)_j}{\|(\Kbb u)_j\|}}{(\Kbb v)_j-(\Kbb \eta)_j}= 0,\;\forall v\in\Ccal.
  \end{equation}
  \Cref{eq:strict_comp_var_eq} guarantees that the directional derivative of the solution operator is a linear mapping w.r.t. the direction $h$. Since $S$ is Bouligand differentiable, it implies the Fréchet differentiability \cite[Proposition 3.1.2]{scholtes2012introduction}. Furthermore, \cref{eq:strict_comp_var_eq} is equivalent to the following optimization problem

  \begin{mini}
    {\eta\in\Ccal} {\frac{1}{2}\scalar{\eta}{\phi''(u)\eta} + \sum_{j\in\Ical}\alpha_j\left(\frac{\|(\Kbb\eta)_j\|^2}{\|(\Kbb u)_j\|} - \frac{\scalar{(\Kbb u)_j}{(\Kbb\eta)_j}^2}{\|(\Kbb u)_j\|^3}\right)+h_j\bigscalar{(\Kbb \eta)_j}{\frac{(\Kbb u)_j}{\|(\Kbb u)_j\|}}}{}{}
  \end{mini}

  Then the KKT-optimality conditions for this problem look as follows
  \begin{align*}
    \scalar{\phi''(u)\eta}{v} + \sum_{j\in\Ical}\alpha_j\scalar{T_j(\Kbb \eta)_j}{(\Kbb v)_j} + h_j\bigscalar{\frac{(\Kbb u)_j}{\|(\Kbb u)_j\|}}{(\Kbb v)_j} + \sum_{j\in\Acal_s}\scalar{\nu_j}{(\Kbb v)_j} = 0,\\
    (\Kbb \eta)_j = 0,\forall j\in\Acal_s,
  \end{align*}
  with Lagrange multipliers $\nu_j\in\R^2$. Since all the constraints are linear, the Abadie constraint qualification condition is satisfied. By introducing $\lambda \in \R^{n\times 2}$ as
  \begin{equation*}
    \lambda_j :=\begin{cases}
      \nu_j,&\forall j\in\Acal_s\\
      \alpha_j T_j(\Kbb\eta)_j+\frac{h_j}{\alpha_j}q_j,&\forall j\in\Ical
    \end{cases}
  \end{equation*}
  the result is obtained.
\end{proof}

%% file: sections/bouligand_subdifferential.tex
\section{Bouligand Subdifferential}\label{sec:bouligand_subdifferential}
Even though the Bouligand stationarity condition presented in \Cref{sec:bouligand_stationarity} holds for any local optimal solution without requiring any constraint qualification, its purely primal form is in general not amenable for algorithmic purposes; this limitation is related to the non-linearity of the directional derivative. As a remedy, in this section we will focus on the study of the Bouligand subdifferential of the solution operator $S$. The characterization of the linear elements of this subdifferential turns out to be useful when devising a numerical algorithm to solve the bilevel problem.

Thanks to the local Lipschitz continuity of $S$, showed in \cref{sec:bouligand_stationarity}, and Rademacher's theorem, we know the solution operator is differentiable almost everywhere. Denoting the set of points where this function is differentiable as $D_S$, the Bouligand subdifferential $\partial_B S(\alpha)$ is defined as follows.
\begin{definition}[Bouligand subdifferential]
  Let $S:\R^n\to\R^m$ be a locally Lipschitz function, and $\alpha\in\R^n$ arbitrary but fixed. The set
  \begin{equation}
    \partial_B S(\alpha) := \{G\in\R^{m\times n}:\exists\{\alpha_k\}\subset D_S,\;\alpha_k\to \alpha \wedge S'(\alpha_k)\to G\},
  \end{equation}
  is called \textit{Bouligand subdifferential} of $S$ at $\alpha$.
\end{definition}

In the next result a characterization of the elements of the Bouligand subdifferential is provided. We assume along this section that $\alpha_j>0$, i.e., $\Tcal \cup\Ical_0=\emptyset$.
\begin{theorem}\label{teo: G is a solution of the linear system}
  Let $G\in\partial_B S(\alpha)$. There exists a partition of the biactive set $\Bcal = \Bcal_1 \cup \Bcal_2$ such that, for any $h$ such that $\alpha+th\ge0$, $\tilde{\eta} :=Gh$ is the unique solution of the system
  \begin{subequations}\label{eq:G_linear_system}
    \begin{align}
      \scalar{\phi''(u)\tilde{\eta}}{v} + \sum_{j\in\Ical\cup\Bcal_2}\scalar{\tilde{\lambda}_j}{(\Kbb v)_j} &= 0,&\forall v\in V \label{eq:G_linear_system_1}\\
      \tilde{\lambda}_j - \alpha_jT_j(\Kbb \tilde{\eta})_j - \frac{h_j}{\alpha_j}q_j &= 0,&\forall j\in\Ical,\\
      \tilde{\lambda}_j - \frac{h_j}{\alpha_j}q_j&=0,&\forall j\in\Bcal_2
    \end{align}
  \end{subequations}
  where $V := \{v\in\R^m:\;(\Kbb v)_j = 0,\;\forall j\in \Acal_s\cup\Bcal_1;(\Kbb v)_j\in \spann(q_j),\forall j\in\Bcal_2\}$.
\end{theorem}
\begin{proof}
  We know the solution operator is locally Lipschitz continuous (see \cref{teo:solution_operator_lipschitz}), which implies it is differentiable almost everywhere. Let us consider a sequence $\{\alpha_k\}\subset D_S$ such that $\alpha_k\to \alpha$ and $S'(\alpha_k)\to G$. Due to the differentiability in the elements of this sequence, the parameter sequence fulfills $(\alpha_k)_j>0$, for all $j=1,\dots,n$. Moreover, thanks to the Lipschitz continuity of $S$ we know that
  \begin{align*}
    u_k = S(\alpha_k) &\to S(\alpha) = u,\\
    \Kbb^\top q_k = -\phi'(u_k) &\to -\phi'(u) = \Kbb^\top q.
  \end{align*}
  This last statement follows from the fact that $\{q_k\}$ is also bounded and, therefore, has a converging subsequence. Now, each of this subsequence elements $(u_k,q_k)$ define their respective inactive $\Ical^k:=\Ical(\alpha_k,u_k)$ and strongly active $\Acal_s^k:=\Acal_s(\alpha_k,u_k)$ sets.\\
  By continuity, we know that $\Ical \subset\Ical^k$ and $\Acal_s \subset\Acal_s^k$, for $k$ sufficiently large. Since $\{\alpha_k\}\subset D_S$, it then follows that $\eta_k:=S'(\alpha_k)h$ satisfies the system
  \begin{subequations}\label{eq:differentiable_sequence}
    \begin{align}
      \phi''(u_k)\eta_k + \Kbb^\top \lambda_k&= 0 , \label{eq:differentiable_sequence_1}\\
      (\lambda_k)_j - (\alpha_k)_j(T_k)_j(\Kbb\eta_k)_j &= \frac{h_j}{(\alpha_k)_j} \Kbb_j^\top  (q_k)_j,&\forall j\in\Ical^k,\label{eq:differentiable_sequence_2}\\
      (\Kbb\eta_k)_j &= 0,&\forall j\in\Acal_s^k,
    \end{align}
  \end{subequations}
  or equivalently,
  \begin{equation}\label{eq:differentiable_sequence2}
    \scalar{\bm{\phi}''(u_k)\eta_k}{v}+\sum_{j\in\Ical^k}\bigscalar{(\alpha_k)_j(T_k)_j(\Kbb \eta_k)_j}{(\Kbb v)_j} + \frac{h_j}{(\alpha_k)_j} \scalar{(q_k)_j}{(\Kbb v)_j} = 0,\forall v\in V^k,
  \end{equation}
  with $V^k:=\{v\in\R^m: (\Kbb v)_j=0, \forall j\in\Acal_s^k\}$. From the definition of the Bouligand subdifferential it follows that $\tilde{\eta} = \lim_{k\to\infty}\eta_k$. Moreover, since for $j\in\Ical$ the sequence $\{(\lambda_k)_j\}$ is bounded, then there exists a  subsequence that converges to a limit point $\tilde{\lambda}_j$. Therefore, up to a subsequence, by passing to the limit we get
  \begin{subequations}
    \begin{align*}
        \tilde{\lambda}_j-\alpha_jT_j(\Kbb\tilde{\eta})_j - \frac{h_j}{\alpha_j}q_j &= 0,\;\forall j\in \Ical,\\
        (\Kbb \tilde{\eta})_j &= 0,\;\forall j\in\Acal_s.
    \end{align*}
\end{subequations}
Let us now consider a partition of the biactive set $\Bcal = \Bcal_1 \cup \Bcal_2$, with
\begin{align*}
  \Bcal_1(\alpha,u) := \{j\in\Bcal(\alpha,u):\exists \{u_{k_l}\}:(\Kbb u_{k_l})_j = 0,\forall l\} \quad \text{and} \quad
  \Bcal_2(\alpha,u) := \Bcal(\alpha,u)\backslash\Bcal_1(\alpha,u).
\end{align*}
For the index set $\Bcal_1$ we know the subsequence $(\Kbb \eta_{k_l})_j=0$, for all $k$. Since $\eta_k\to \tilde{\eta}$, we get
\begin{equation*}
  (\Kbb \tilde{\eta})_j = 0,\;\forall j\in\Acal_s\cup\Bcal_1.
\end{equation*}

Considering the partition $\Bcal_2$, we approach a biactive point by a sequence of points such that $(\Kbb u_k)_j\neq 0$, i.e., $j\in\Ical^k$. Let us first notice that the term on the right hand side of \Cref{eq:differentiable_sequence_2} is uniformly bounded and, therefore, as $k \to \infty$,
\begin{equation*}
  \sum_{j\in\Ical^k}\frac{h_j}{(\alpha_k)_j} \scalar{(q_k)_j}{(\Kbb v)_j} \to \sum_{j\in\Ical \cup \Bcal_2}\frac{h_j}{\alpha_j} \scalar{q_j}{(\Kbb v)_j}, \quad \forall v \in V.
\end{equation*}
In addition, defining $(\zeta_k)_j=(\alpha_k)_j(T_k)_j(\Kbb\eta_k)_j$, for $j \in \Ical^k$, we get that
\begin{equation*}
  \scalar{(\zeta_k)_j}{(\Kbb \eta_k)_j} = \frac{(\alpha_k)_j}{\|(\Kbb u_k)_j\|} \left( \|(\Kbb \eta_k)_j\|^2 -\frac{1}{\|(\Kbb u_k)_j\|^2}\scalar{(\Kbb \eta_k)_j}{(\Kbb u_k)_j}^2 \right) \geq 0, ~\forall j \in \Ical^k.
\end{equation*}
Thanks to \Cref{eq:differentiable_sequence2}, we also get that
\begin{equation*}
  0 \leq \scalar{(\zeta_k)_j}{(\Kbb \eta_k)_j}  \leq \scalar{\phi''(u_k)\eta_k}{\eta_k} + \sum_{j\in\Ical^k}\scalar{(\zeta_k)_j}{(\Kbb \eta_k)_j} \leq \sum_{j\in\Ical^k} |h_j| \|(\Kbb \eta_k)_j\|,
\end{equation*}
which, since $\eta_k \to \tilde{\eta}$, as $k \to \infty,$ implies that $\scalar{(\zeta_k)_j}{(\Kbb \eta_k)_j}$ is uniformly bounded. Since for $j \in \Bcal_2$ we know that $(\Kbb u_k)_j \to 0$, it follows from the previous relations that
$$\alpha_j^2 \|(\Kbb \tilde \eta)_j\|^2 -\scalar{q_j}{(\Kbb \tilde \eta)_j}^2 = \lim_{k \to \infty} (\alpha_k)_j^2 \|(\Kbb \eta_k)_j\|^2 -\scalar{(q_k)_j}{(\Kbb \eta_k)_j}^2 =0,$$
which implies that $(\Kbb \tilde \eta)_j \in \spann(q_j),\forall j\in\Bcal_2$.

Finally, for any $v\in V$, we obtain that
\begin{align*}
  \lim_{k\to\infty}\scalar{(\zeta_k)_j}{(\Kbb v)_j} &= \lim_{k \to\infty}\scalar{(\zeta_k)_j}{c(q_k)_j} = \lim_{k \to\infty}\bigscalar{(\zeta_k)_j}{c(\alpha_k)_j\frac{(\Kbb u_k)_j}{\|(\Kbb u_k)_j\|}},\\
  &= c\lim_{k \to\infty}\bigscalar{(\alpha_k)_j\frac{(\Kbb \eta_k)_j}{\|(\Kbb u_k)_j\|}}{(\alpha_k)_j\frac{(\Kbb u_k)_j}{\|(\Kbb u_k)_j\|}}\\
  & \hspace{2cm} - \bigscalar{(\alpha_k)_j\frac{\scalar{(\Kbb\eta_k)_j}{(\Kbb u_k)_j}(\Kbb u_k)_j}{\|(\Kbb u_k)_j\|^3}}{(\alpha_k)_j\frac{(\Kbb u_k)_j}{\|(\Kbb u_k)_j\|}},\\
  &=c\lim_{k \to\infty}\frac{(\alpha_k)_j^2}{\|(\Kbb u_k)_j\|^2}\scalar{(\Kbb \eta_k)_j}{(\Kbb u_k)_j}\\
  & \hspace{2cm} -\frac{(\alpha_k)_j^2}{\|(\Kbb u_k)_j\|^4}\scalar{(\Kbb \eta_k)_j}{(\Kbb u_k)_j}\scalar{(\Kbb u_k)_j}{(\Kbb u_k)_j}
  =0.
\end{align*}
By passing to the limit in \cref{eq:differentiable_sequence2} the result is obtained.
\end{proof}

\begin{corollary}\label{cor: G is a solution of the linear system}
  Let $G\in\partial_B S(\alpha)$. There exists a partition of the biactive set $\Bcal = \Bcal_1 \cup \Bcal_2$ and a multiplier $\theta \in \R^m$ such that, for any $h$ such that $\alpha+th\ge0$, $\tilde{\eta}:= Gh$ is the unique solution of the system
  \begin{subequations}\label{eq:full G_linear_system}
    \begin{align}
      \phi''(u)\tilde{\eta} + \Kbb^T \theta &= 0\\
      \theta_j - \alpha_jT_j(\Kbb \tilde{\eta})_j - \frac{h_j}{\alpha_j} q_j &= 0, &&\forall j\in\Ical,\\
      \scalar{\theta_j}{q_j}-\alpha_j h_j&=0, &&\forall j\in\Bcal_2.
    \end{align}
  \end{subequations}
\end{corollary}
\begin{proof}
  Let us consider the functional $\mathcal F \in \R^n$ defined by
  \begin{equation*}
    (\mathcal F ,v):= (\phi''(u)\tilde{\eta},v) + \sum_{j\in\Ical}\scalar{\alpha_jT_j(\Kbb \tilde{\eta})_j}{(\Kbb v)_j} +\sum_{j\in\Ical \cup \Bcal_2} \frac{h_j}{\alpha_j} \scalar{q_j}{(\Kbb v)_j},\;\forall v\in V.
  \end{equation*}
  \Cref{eq:G_linear_system_1} can then be written as $\mathcal F \in V^\perp$. Thanks to the structure of the linear subspace $V$, it can be represented in a separated way as
$
    V= \left( \bigcap_{j \in \Acal_S \cup \Bcal_1} V^1_j \right) \cap \left( \bigcap_{j \in \Bcal_2} V^2_j \right),
$
  where
  \begin{align*}
    & V^1_j:= \{ v \in \R^n: (\Kbb v)_j =0 \}, &&  j \in \Acal_S \cup \Bcal_1,\\
    & V^2_j:= \{ v \in \R^n: (\Kbb v)_j \in \spann(q_j) \}, &&  j \in \Bcal_2.
  \end{align*}
  Consequently, $V^\perp = \sum_{j \in \Acal_S \cup \Bcal_1} (V^1_j)^\perp + \sum_{j \in \Bcal_2} (V^2_j)^\perp.$

  For $j \in \Acal_S \cup \Bcal_1$, we get that $(V^1_j)^\perp= \ker (\Kbb_j)^\perp$. Thanks to the orthogonality relations, it follows that $\ker (\Kbb_j)^\perp=\range (\Kbb_j^T)$. Hence, for any $\xi_j \in (V^1_j)^\perp$, there exist $\pi_j$ such that $\xi_j= \Kbb_j^T \pi_j$. Consequently,
  \begin{equation*}
    \sum_{j \in \Acal_S \cup \Bcal_1} (V^1_j)^\perp = \sum_{j \in \Acal_S \cup \Bcal_1} \Kbb_j^T \pi_j, \quad \pi_j \in \R^2.
  \end{equation*}

  For $j \in \Bcal_2$,
  any $v \in V^2_j$ can be represented as
  \begin{equation*}
    v= \phi + \varphi, \quad \text{ with }(\Kbb_j \varphi)=0 \text{ and } \phi \in \range (\Kbb_j^T).
  \end{equation*}
  Since $(\Kbb v)_j \in \spann(q_j)$ and $(\Kbb_j \varphi)=0$, it follows that $(\Kbb v)_j \in \spann(q_j)$ as well.
  Let us now consider $w_j \in (V^2_j)^\perp$, which can be represented as $w_j=\tilde{w}_j+\hat{w}_j$, where $\tilde{w}_j \in \range(\Kbb_j^T)$ and $\hat{w}_j \in \range(\Kbb_j^T)^\perp = \ker(\Kbb_j)$. Consequently, there exists $\psi_j$ such that
  \begin{equation*}
    w_j = \Kbb_j^T \psi_j + \hat{w}_j, \quad \text{ with }\Kbb_j \hat{w}_j=0.
  \end{equation*}
  Taking the scalar product with $v_j \in  V^2_j$ we get
  \begin{align*}
    (w_j,v_j)  = (\Kbb_j^T \psi_j + \hat{w}_j, \phi + \varphi) = \scalar{\psi_j}{\Kbb_j \phi} +  (\hat{w}_j, \Kbb_j^T \psi) + (\hat{w}_j, \varphi)= c \scalar{\psi_j}{q_j} + (\hat{w}_j, \varphi),
  \end{align*}
  since $\Kbb_j \varphi =\Kbb_j \hat{w}_j= 0$. For the product to be zero, it is then required that $(\hat{w}_j, \varphi)=0, \forall \varphi \in \ker(\Kbb_j)$ and $\scalar{\psi_j}{q_j}=0.$
  Since $\hat{w}_j$ belongs to $\ker(\Kbb_j)$ as well, it follows that $\hat{w}_j=0.$
  Consequently,
  \begin{equation*}
    \sum_{j \in \Bcal_2} (V^2_j)^\perp = \sum_{j \in \Bcal_2} \Kbb_j^T \psi_j, \quad \psi_j \in \R^2: \scalar{\psi_j}{q_j}=0.
  \end{equation*}

  Altogether, we then obtain that there exist multipliers $\pi_j$ and $\psi_j$ such that
  \begin{equation*}
    \mathcal F + \sum_{j \in \Acal_S \cup \Bcal_1} \Kbb_j^T \pi_j + \sum_{j \in \Bcal_2} \Kbb_j^T \psi_j =0,
  \end{equation*}
  with $\scalar{\psi_j}{q_j}=0$. Defining
\begin{equation*}
  \theta_j := \begin{cases}
    \alpha_jT_j(\Kbb \tilde{\eta})_j + \frac{h_j}{\alpha_j}q_j, & j\in\Ical,\\
    \pi_j, & j \in \Acal_S \cup \Bcal_1,\\
    \psi_j + \frac{h_j}{\alpha_j}q_j, & j \in \Bcal_2,
\end{cases}
\end{equation*}
the result is obtained.
\end{proof}

Next we verify that, along a given direction, there exists a solution of system \cref{eq:G_linear_system} which coincides with the directional derivative. When properly characterized, this enables us to use of a linear representative of the (otherwise nonlinear) directional derivative within a solution algorithm.

\begin{theorem}\label{teo:exist_linearized_element}
  For any $\alpha\in\R^n_+$ and $h\in\R^n$ such that $\alpha + th \ge 0$, there exists a linearized element $\tilde{\eta}=Gh$ such that $S'(\alpha)h=Gh$.
\end{theorem}
\begin{proof}
  Let us recall that, since by assumption $\mathcal T \cup \mathcal I_0 = \emptyset$, the directional derivative of the solution mapping, in direction $h$, is given by the unique $\eta\in\Ccal(\alpha,u)$ solution of
  \begin{multline} \label{eq: directional derivative representation}
    \scalar{\bm{\phi}''(u)\eta}{v-\eta} +\sum_{j\in\Ical} \scalar{\alpha_jT_j(\Kbb\eta)_j}{(\Kbb v)_j-(\Kbb\eta)_j}\ge\\ -\sum_{j\in\Ical}h_j\bigscalar{\frac{(\Kbb u)_j}{\|(\Kbb u)_j\|}}{(\Kbb v)_j-(\Kbb\eta)_j} - \sum_{j\in\Bcal} \frac{h_j}{\alpha_j} \scalar{q_j}{(\Kbb v)_j - (\Kbb \eta)_j},
  \end{multline}
  for all $v\in\Ccal(u)$. Considering the sets $\Bcal_1:=\{j\in\Bcal:\;(\Kbb\eta)_j = 0\}$ and $\Bcal_2:=\Bcal\backslash\Bcal_1$, and since $\eta\in\Ccal(u)$, it also follows that $(\Kbb \eta)=c_j q_j$, for all $j\in\Bcal_2$, for some $c_j > 0$. Consequently $\eta$ belongs to the subspace
  \begin{equation*}
    V := \{v\in\R^n:\;(\Kbb v)_j = 0,\;\forall j\in\Acal_s\cup\Bcal_1;\;(\Kbb v)_j\in \spann(q_j),\;\forall j\in\Bcal_2\}.
  \end{equation*}
Moreover, for any $w \in V$ it follows that, for $t$ sufficiently small, $\eta \pm t w \in \Ccal(u)$. Testing \eqref{eq: directional derivative representation} with these vectors we then get that
\begin{equation*}
  \scalar{\bm{\phi}''(u)\eta}{w} +\sum_{j\in\Ical} \scalar{\alpha_jT_j(\Kbb\eta)_j}{(\Kbb w)_j}= - \sum_{j\in \Ical \cup \Bcal_2} \frac{h_j}{\alpha_j} \scalar{q_j}{(\Kbb w)_j}, \quad \forall w \in V,
\end{equation*}
and, consequently, the directional derivative takes the form $\eta = Gh$, solution of \cref{eq:G_linear_system}, with $\Bcal_2$ as defined above.
\end{proof}

%% file: sections/trust_region.tex
\section{Trust Region Algorithm}\label{sec:tr}
In this section we will describe the numerical algorithm used for finding optimal parameters of \cref{eq:bilevel_problem}. Thanks to the Bouligand subdifferential characterization given in \cref{sec:bouligand_subdifferential}, it is possible to compute a linearized representative of the directional derivative via the linear system \cref{eq:G_linear_system}. Using this information we can make use of a descent-like algorithm for its numerical solution. Indeed, by using the uniqueness properties of the solution operator, we can write a \textit{reduced optimization problem}
\begin{mini}
{\substack{\alpha \ge 0}}{j(u(\alpha))}{}{}
\end{mini}
where $u(\alpha)$ is the image reconstruction corresponding to a particular value of $\alpha$. With this reduced problem, we can make use of the stationarity condition for the bilevel problem described in \cref{eq:dir_derivative_reduced_cost} and the directional derivative characterization \cref{eq:sd_linearized_system}. By using the definition of the directional derivative for the reduced optimization problem, we get
\begin{equation}\label{eq:generalized_gradient_rep}
    \scalar{j'(\alpha)}{h} = \scalar{\nabla J(u)}{S'(\alpha;h)} = \scalar{\nabla J(u)}{\tilde{\eta}}.
\end{equation}
where $\tilde{\eta}$ is a solution of system \cref{eq:G_linear_system} for a particular partition of the biactive set $\Bcal = \Bcal_1\cup\Bcal_2$. Let us now define a \emph{generalized adjoint} $p\in\R^m$ as the solution of the following system
\begin{align*}
    \scalar{\phi''(u)^\top p}{v} + \sum_{j\in\Ical}\scalar{\mu_j}{(\Kbb v)_j} - \scalar{\nabla J(u)}{v} &= 0,&\forall v\in V,\\
    \mu_j - \alpha_jT_j(\Kbb p)_j &= 0,&\forall j\in\Ical,
\end{align*}
where $V$ is defined as in \Cref{teo: G is a solution of the linear system}. Moreover, using the results in \Cref{teo:exist_linearized_element}, we know that $\tilde{\eta}\in V$ is a linear representative of the directional derivative. Consequently, \cref{eq:generalized_gradient_rep} reads
\begin{equation*}
    \scalar{j'(\alpha)}{h} = \scalar{\nabla J(u)}{\tilde{\eta}} = \scalar{\phi''(u)^\top p}{\tilde{\eta}} + \sum_{j\in\Ical}\scalar{\alpha_j T_j(\Kbb p)_j}{(\Kbb\tilde{\eta})_j},
\end{equation*}
Rearranging the terms we get
\begin{equation*}
    \scalar{j'(\alpha)}{h} = \scalar{p}{\phi''(u)\tilde{\eta}} + \sum_{j\in\Ical}\scalar{(\Kbb p)_j}{\alpha_j T_j(\Kbb\tilde{\eta})_j} = \scalar{p}{\phi''(u)\tilde{\eta}} + \sum_{j\in\Ical}\scalar{(\Kbb p)_j}{\tilde{\lambda}_j},
\end{equation*}
Finally, using \cref{eq:sd_linearized_system} we get
\begin{equation}
    \scalar{j'(\alpha)}{h}= -\sum_{j\in\Ical\cup\Bcal_2}\frac{h_j}{\alpha_j}\scalar{q_j}{(\Kbb p)_j}\label{eq:b_subgrad_equation}
\end{equation}


In this work, we will rely on a nonsmooth trust region method in the spirit of \cite{christof2017non}. In general, a trust-region algorithm works by defining for an iteration $k$ a radius $\Delta_k$ and replaces the function with a \textit{model function} $m_k(\alpha_k)$ within a \textit{trust-region}. This region in this paper will be represented by a $l_\infty$ ball, and the radius $\Delta_k$ is updated according to a \textit{quality measure} based on the predicted decrease of the model function \textit{pred} and the actual cost function reduction \textit{ared}.

The model function we will be using has the following form
\begin{equation*}
    m_k(\alpha_k) = j(\alpha_k) + g_k^\top d_k + \frac{1}{2}d_k^\top B_k d_k,
\end{equation*}
where $B_k$ is a second order matrix built using BFGS updates, and $g_k$ is obtained either by using the Bouligand subdifferential element obtained using \cref{eq:b_subgrad_equation} or a gradient obtained from the regularized model, $g_{\gamma,k}$, if the radius falls below a threshold value, see \Cref{subsec:local_regularization}. The idea of differentiating two phases for the trust region subproblem approximation was taken from \cite{christof2017non}, where this strategy is presented for optimizing nonsmooth, non-convex and locally-Lipschitz functions. By defining an appropriate model function, dependent on the radius value, convergence to a Clarke (C)-stationary points may be verified \cite[Proposition 2.10]{christof2017non}. The complete steps are provided in \cref{algo:tr_algorithm}.

\begin{algorithm}
    \caption{Non-smooth Trust-Region Algorithm}\label{algo:tr_algorithm}
    \begin{algorithmic}[1]
        \STATE{Choose initial parameter $\alpha_0$, radius $\Delta_0$, $0<\eta_1\le\eta_2<1$, $0<\gamma_1\le\gamma_2<1$ and $tol>0$}
        \STATE{Choose initial second order matrix $B_0$ and a threshold radius $\Delta_t$}
        \STATE{$k=0$}
        \WHILE{$\Delta_k > tol$}
        \IF{$\Delta_k>=\Delta_t$}
            \STATE{$m_k(\alpha_k) = j(\alpha_k) + g_k^\top d_k + \frac{1}{2}d_k^\top B_kd_k$}
        \ELSE
            \STATE{$m_k(\alpha_k) = j(\alpha_k) + g_{\gamma,k}^\top d_k + \frac{1}{2}d_k^\top B_kd_k$}
        \ENDIF
        \STATE{Compute a step $s_k$ that ``sufficiently'' reduces the model $m_k$ such that $\alpha_k+s_k\in B_{\Delta_k}$\label{algo:step_calculation}}
        \STATE{Update second order matrix $B_k$ using limited memory BFGS.}
        \STATE{Compute the quality measure $\rho_k$}
        \STATE{$\alpha_{k+1} = \begin{cases}
            \alpha_k &\text{if }\rho_k\le\eta_1,\\
            \alpha_k+s_k &\text{otherwise}.
            \end{cases}$}
        \STATE{$\Delta_{k+1} = \begin{cases}
                [\Delta_k,\infty)&\text{if }\rho_k\ge\eta_2,\\
                [\gamma_2\Delta_k,\Delta_k]&\text{if }\rho_k\in(\eta_1,\eta_2),\\
                [\gamma_1\Delta_k,\gamma_2\Delta_k]&\text{if }\rho_k<\eta_1.
            \end{cases}$}
        \STATE{$k \leftarrow k+1$}
        \ENDWHILE
        \RETURN $\alpha_k$
    \end{algorithmic}
\end{algorithm}

\subsection{Trust Region Subproblem}

Regarding step \ref{algo:step_calculation} in \cref{algo:tr_algorithm}, we need to approximate the solution of the following trust region sub-problem
\begin{mini}
    {\substack{s\in\R^n}}{j(\alpha)+g^\top s + \frac{1}{2}s^\top Bs}{}{}
    \addConstraint{\|s\|_{\infty}\le \Delta}
    \addConstraint{s+\alpha\ge 0},
\end{mini}
which corresponds to a classical trust-region sub-problem with additional positivity constraints (see, e.g., \cite{voglis2004rectangular,xu2007astral}). The main idea is to reformulate the problem taking advantage of the $l_\infty$ norm used for the ball at the point $x_k$
\begin{mini}
    {\substack{s\in\R^n}}{j(\alpha)+g^\top s + \frac{1}{2}s^\top Bs}{}{}
    \addConstraint{\max(-\alpha_j,-\Delta) \le s_j \le \Delta,\;\forall j=1,\dots,n}.
\end{mini}
For performance purposes it is desirable to solve this problem approximately in such a way that we can guarantee a descent on the cost function. With that goal in mind we will make use of a dogleg strategy that takes into account a Newton step $s_N$ and a Cauchy step $s_C$. In the context of this constrained problem, let us take $\tilde{B}_\Delta = B_\Delta\cap\R^n_+$ and distinguish the following three cases
\begin{enumerate}
    \item $s_N\in \tilde{B}_\Delta$,
    \item $s_C\in \tilde{B}_\Delta$ and $s_N\notin \tilde{B}_\Delta$,
    \item $s_C\notin \tilde{B}_\Delta$ and $s_N\notin \tilde{B}_\Delta$.
\end{enumerate}
For case 1  we take the Newton step; for case 2 a dogleg strategy for box constraints is used; for case 3 we make use of a scaled Cauchy direction as described in \cref{algo:projected_cauchy_algorithm}.
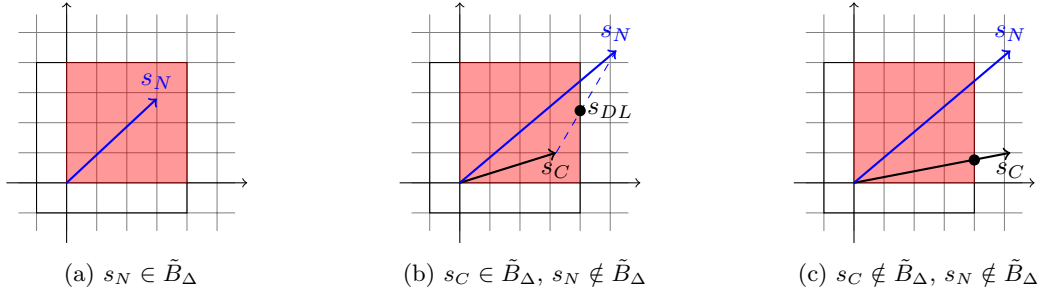
\begin{figure}
    \centering
    \begin{subfigure}{0.31\textwidth}
        \centering
        \input{figures/vectors_by_index_set_1.tex}
        \caption{$s_N\in \tilde{B}_\Delta$}
        \label{fig:sub-first}
    \end{subfigure}
    \begin{subfigure}{0.31\textwidth}
        \centering
        \input{figures/vectors_by_index_set_2.tex}
        \caption{$s_C\in \tilde{B}_\Delta$, $s_N\notin \tilde{B}_\Delta$}
        \label{fig:sub-second}
    \end{subfigure}
    \begin{subfigure}{0.31\textwidth}
        \centering
        \input{figures/vectors_by_index_set_3.tex}
        \caption{$s_C\notin \tilde{B}_\Delta$, $s_N\notin \tilde{B}_\Delta$}
        \label{fig:sub-third}
    \end{subfigure}
    \caption{Projected Cauchy Step:
    {\scriptsize The figure depicts a two dimensional example for the three different possible cases when approximating the trust region subproblem using a dogleg strategy with $l_\infty$ norm and positivity constraints.}}
    \label{fig:step_selection}
\end{figure}
\begin{algorithm}
    \caption{Dogleg Step for Box Constraints}\label{algo:projected_cauchy_algorithm}
    \begin{algorithmic}[1]
        \STATE{Calculate Newton step by solving the linear system $B_k s_N=-g_k$}.
        \IF{$s_N \in \tilde{B}_\Delta$}
        \RETURN $s_N$
        \ENDIF
        \STATE{Calculate $s_C = -\frac{\|g_k\|^2}{g_k^\top*B_k*g_k}g_k$}
        \IF{$s_C \in \tilde{B}_\Delta$}
        \STATE{Calculate the intersection of $s_N-s_C$ with $\tilde{B}_\Delta$ and obtain $s_{DL}$}
        \RETURN $s_{DL}$
        \ENDIF
        \STATE{Find $t$ such that $t*s_C/\|s_C\|$ remains in $\tilde{B}_\Delta$.}
        \RETURN $t*\frac{s_{C}}{\|s_C\|}$
    \end{algorithmic}
\end{algorithm}

\subsection{Local Regularization}\label{subsec:local_regularization}
As a safeguard, for the cases where the trust-region radius fall below a threshold value $\Delta_t$ we will rely on a switching strategy between the nonsmooth model presented earlier and a regularized model. Let us consider the following optimization problem
\begin{mini}
    {\substack{\alpha\in\R^n_+}}{J(\bar{u}(\alpha),u^{train})}{\label{eq:regularized_bilevel_problem}}{}
    \addConstraint{\bar{u}(\alpha) = \argmin_u \phi(u) + \sum_{j=1}^n \alpha_j\|(\Kbb u)_j\|_\gamma}
\end{mini}
where $\|\cdot\|_\gamma$ is a local regularization of the euclidean norm given by
\begin{equation*}
    \|z\|_\gamma = \begin{cases}
        \|z\| - \frac{1}{2\gamma} &\text{if }\|z\| \ge \frac{1}{\gamma},\\
        \frac{\gamma}{2}\|z\|^2 &\text{if } \|z\| < \frac{1}{\gamma}.
    \end{cases}
\end{equation*}
Since the problem is now differentiable, we can define the following optimality condition for the lower level problem is obtained
\begin{equation*}
    \scalar{\phi'(u)}{v} + \sum_{j=1}^n\alpha_j\scalar{h_\gamma ((\Kbb u)_j)}{(\Kbb v)_j} = 0,\;\forall v\in\R^m.
\end{equation*}
Here, $h_\gamma$ correspond to the first derivative of the regularized euclidean norm. Moreover, regarding the bilevel problem, it is possible to formulate the KKT conditions for this problem. Let $(\alpha^*,u^*)\in\R^n_+\times\R^m$ be a stationary point for the regularized bilevel problem \cref{eq:regularized_bilevel_problem}. Then there exist Lagrange multipliers $(p,\sigma)\in\R^m\times\R^n$ such that the following optimality system holds true
\begin{align}
    \scalar{\phi'(u)}{v} + \sum_{j=1}^n\alpha_j^*\scalar{h_\gamma ((\Kbb u^*)_j)}{(\Kbb v)_j} &= 0,&\;\forall v\in\R^m,\\
    \scalar{\phi''(u^*)^\top p}{v} + \sum_{j=1}^n\alpha_j\scalar{h'^*_\gamma ((\Kbb u^*)_j)(\Kbb p)_j}{(\Kbb v)_j} &= -\scalar{J'(u)}{v},&\;\forall v\in\R^m,\label{eq:regularized_adjoint}\\
    \scalar{h_\gamma((\Kbb u^*)_j)}{(\Kbb p)_j} &= \sigma_j,&\;\forall j=1,\dots,n,\\
    0\le \sigma \perp \alpha^* \ge 0.
\end{align}
Moreover, with help of the adjoint equation \cref{eq:regularized_adjoint} it is possible to derive a gradient formula for the reduced cost function $j(\alpha) = J(u_\gamma(\alpha))$ as
\begin{equation*}
    (j'(\alpha))_j = \scalar{h_\gamma((\Kbb u)_j)}{(\Kbb p)_j}.
\end{equation*}
According to \cref{algo:tr_algorithm}, if the value for the radius runs below a predetermined threshold, the algorithm will make use of this gradient for its calculations.

%% file: figures/vectors_by_index_set_1.tex
\begin{tikzpicture}[scale=0.8]
    \draw[step=0.5cm,gray,very thin] (-1.8,-1.8)grid(1.8,1.8);
    \draw[->](-2.0,-1.0)-- (2.0,-1.0);
    \draw[->](-1.0,-2.0)-- (-1.0,2.0);
    \draw(-1.5,-1.5)rectangle(1.0,1.0);
    \fill[color=red,opacity=0.4] (-1.0,-1.0) rectangle (1.0,1.0);
    \draw[->,color=blue,thick](-1.0,-1.0)--(0.5,0.4) node[above] {$s_N$};
    \draw[color=black](1.0,0);
\end{tikzpicture}

%% file: figures/vectors_by_index_set_2.tex
\begin{tikzpicture}[scale=0.8]
    \draw[step=0.5cm,gray,very thin] (-1.8,-1.8)grid(1.8,1.8);
    \draw[->](-2.0,-1.0)-- (2.0,-1.0);
    \draw[->](-1.0,-2.0)-- (-1.0,2.0);
    \draw(-1.5,-1.5)rectangle(1.0,1.0);
    \fill[color=red,opacity=0.4] (-1.0,-1.0) rectangle (1.0,1.0);
    \draw[->,color=black,thick](-1.0,-1.0)--(0.6,-0.5) node[below] {$s_C$};
    \draw[->,color=blue,thick](-1.0,-1.0)--(1.6,1.2) node[above] {$s_N$};
    \draw[color=blue,dashed](0.6,-0.5)--(1.6,1.2);
    \node at (1.0,0.2)[circle,fill,inner sep=1.5pt]{};
    \node at (1.5,0.2){$s_{DL}$};
\end{tikzpicture}

%% file: figures/vectors_by_index_set_3.tex
\begin{tikzpicture}[scale=0.8]
    \draw[step=0.5cm,gray,very thin] (-1.8,-1.8)grid(1.8,1.8);
    \draw[->](-2.0,-1.0)-- (2.0,-1.0);
    \draw[->](-1.0,-2.0)-- (-1.0,2.0);
    \draw(-1.5,-1.5)rectangle(1.0,1.0);
    \fill[color=red,opacity=0.4] (-1.0,-1.0) rectangle (1.0,1.0);
    \draw[->,color=black,thick](-1.0,-1.0)--(1.6,-0.5) node[below] {$s_C$};
    \draw[->,color=blue,thick](-1.0,-1.0)--(1.6,1.2) node[above] {$s_N$};
    \node at (1.0,-0.62)[circle,fill,inner sep=1.5pt]{};
\end{tikzpicture}

%% file: sections/experiments.tex
\section{Experiments}\label{sec:experiments}
In this section we report on the performance of \cref{algo:tr_algorithm} presented in \Cref{sec:tr} to find optimal parameters. With this goal in mind we prepared two training image datasets. 
\subsection{Single Training Pair}
The first dataset we will explore is a single 128 by 128 pixel image pair dataset based on the cameraman image and a corrupted version, obtained by adding gaussian noise with zero mean and standard deviation $\sigma=0.05$. \Cref{fig:cameraman_dataset} shows this training image pair along with the optimal parameter obtained using the trust-region algorithm for both a scalar and a two-dimensional patch parameter. The improvement when using a two-dimensional patch parameter, according to the SSIM value of the image reconstructions, in \Cref{fig:cost_cameraman}. Moreover, the cost function when using a scalar and a two dimensional parameter is shown, respectively. These plots show the non-convexity of the cost function, and the cost corresponding to the optimal parameter.

\begin{figure}
    \captionsetup[subfigure]{labelformat=empty}
    \centering
    \begin{subfigure}{0.18\textwidth}
        \centering
        \includegraphics[width=0.9\textwidth]{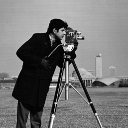}
        \caption{\scriptsize Original\\.}
    \end{subfigure}
    \begin{subfigure}{0.18\textwidth}
        \centering
        \includegraphics[width=0.9\textwidth]{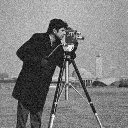}
        \caption{\scriptsize Noisy\\SSIM=0.6091}
    \end{subfigure}
    \begin{subfigure}{0.18\textwidth}
        \centering
        \includegraphics[width=0.9\textwidth]{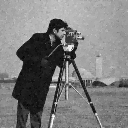}
        \caption{\scriptsize$\alpha^*=0.0155$\\SSIM=0.7711}
    \end{subfigure}
    \begin{subfigure}{0.18\textwidth}
        \centering
        \begin{tikzpicture}
            \node (image) at (0,0) {
                \includegraphics[width=0.9\textwidth]{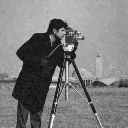}
            };
            \draw[green,very thick] (-1.35,0.0) rectangle (1.35,1.35) node[below left,black,fill=green]{$\alpha_1$};
            \draw[green,very thick] (-1.35,-1.35) rectangle (1.35,0.0) node[below left,black,fill=green]{$\alpha_2$};
        \end{tikzpicture}
        \caption{\scriptsize$\alpha^*=[0.0233; 0.0126]$\\ SSIM=0.8007}
    \end{subfigure}
    \caption{Cameraman Dataset\\
    \scriptsize
    Optimal reconstructions using a scalar regularization parameter and a 2 dimensional regularization parameter.}
    \label{fig:cameraman_dataset}
\end{figure}

\begin{figure}
    \captionsetup[subfigure]{labelformat=empty}
    \centering
    \begin{subfigure}{0.35\textwidth}
        \centering
        \input{experiments/cameraman_128_5/cameraman_128_5_cost_plot.tex}
    \end{subfigure}
    \begin{subfigure}{0.35\textwidth}
        \centering
        \input{experiments/cameraman_128_5/cameraman_128_5_cost_plot_2d.tex}
    \end{subfigure}
    \caption{Cost Function - Cameraman Dataset\\
    \scriptsize
    Values for the $l_2$ squared cost function using a scalar regularization parameter and a two dimensional regularization parameter using the Cameraman single image dataset.}
    \label{fig:cost_cameraman}
\end{figure}

In order to explore the behavior of the algorithm with more degrees of freedom we will make use of a so called \textit{patch-based} parameter. This parameter will be a piecewise constant parameter. We split the image size into a grid of different size according to the number of patches, i.e., the patch size will be smaller as we require more patches within the image domain. For instance a 4x4 patch will calculate an optimal parameter of size 16.

Regarding the behavior of the algorithm, \Cref{table:cameraman_table} shows the performance on the cameraman training dataset when using different number of patches, along with the number of iterations, cost, quality measures and residue. Indeed, the improvement on the reconstructed images when using more patches can be verified from the data. Furthermore, \Cref{fig:patch_cameraman} shows the optimal reconstructions along with the optimal parameter for different number of patches. In this experiment it can be seen how the regularization parameter can adjust different regularization levels according to the training pair.


\begin{table}
    \centering
    \begin{tabular}{@{}lccccr@{}} \toprule
        \multicolumn{2}{c}{Item}&\multicolumn{2}{c}{}&\multicolumn{2}{c}{Reconstruction} \\ \cmidrule(r){5-6}
        Patch & Iterations & $\|\alpha_{k+1}-\alpha_k\|$ & COST & SSIM & PSNR \\ \midrule
        Scalar & 19 & 7.629e-7 & 34.1206 & 0.7712 & 26.3794 \\
        2x1 & 20 & 4.883e-5 & 33.6915  & 0.8007 & 25.6222 \\
        2x2 & 19 & 4.639e-5 & 33.6625  & 0.8034 & 25.8982 \\
        4x4 & 22 & 1.461e-4 & 33.0175 & 0.8189 & 25.8303 \\
        8x8 & 28 & 5.614e-5 & 32.2040 & 0.8283 & 24.4066 \\
        16x16 & 36 & 9.119e-4 & 31.6976 & 0.8455 & 25.8264 \\\bottomrule
    \end{tabular}
    \caption{Trust Region Algorithm behavior on the Cameraman dataset.}
    \label{table:cameraman_table}
\end{table}

\begin{figure}
    \captionsetup[subfigure]{labelformat=empty}
    \centering
    \begin{subfigure}{0.18\textwidth}
        \centering
        \includegraphics[width=0.99\textwidth]{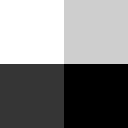}
        \caption{2x2}
    \end{subfigure}
    \begin{subfigure}{0.18\textwidth}
        \centering
        \includegraphics[width=0.99\textwidth]{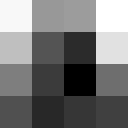}
        \caption{4x4}
    \end{subfigure}
    \begin{subfigure}{0.18\textwidth}
        \centering
        \includegraphics[width=0.99\textwidth]{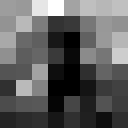}
        \caption{8x8}
    \end{subfigure}
    \begin{subfigure}{0.18\textwidth}
        \centering
        \includegraphics[width=0.99\textwidth]{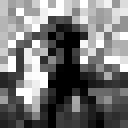}
        \caption{16x16}
    \end{subfigure}

    \begin{subfigure}{0.18\textwidth}
        \centering
        \includegraphics[width=0.99\textwidth]{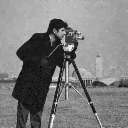}
        \caption{SSIM=0.8034}
    \end{subfigure}
    \begin{subfigure}{0.18\textwidth}
        \centering
        \includegraphics[width=0.99\textwidth]{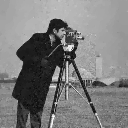}
        \caption{SSIM=0.8189}
    \end{subfigure}
    \begin{subfigure}{0.18\textwidth}
        \centering
        \includegraphics[width=0.99\textwidth]{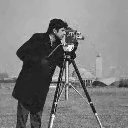}
        \caption{SSIM=0.8283}
    \end{subfigure}
    \begin{subfigure}{0.18\textwidth}
        \centering
        \includegraphics[width=0.99\textwidth]{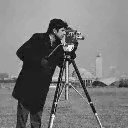}
        \caption{SSIM=0.8455}
    \end{subfigure}
    \caption{Optimal Patch Parameter - Cameraman Dataset\\
    \scriptsize
    Values for the $l_2$ squared cost function using a scalar regularization parameter and a two dimensional regularization parameter using the Cameraman dataset.}
    \label{fig:patch_cameraman}
\end{figure}

\subsection{Multiple Training Pairs}
For the second experiment, we used ten image pairs containing images of faces to generate a training dataset and ten different image pairs to generate a testing dataset; both datasets were based on the CelebA dataset \cite{liu2015faceattributes}. These images are of size 128 by 128 pixels and in both datasets, the degenerated pairs were generated by adding gaussian noise with zero-mean and standard deviation $\sigma = 0.1$. A subset of the training dataset is depicted in \Cref{fig:faces_dataset}.

\begin{figure}
    \captionsetup[subfigure]{labelformat=empty}
    \centering
    \begin{subfigure}{0.115\textwidth}
        \centering
        \includegraphics[width=0.99\textwidth]{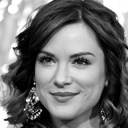}
    \end{subfigure}
    \begin{subfigure}{0.115\textwidth}
        \centering
        \includegraphics[width=0.99\textwidth]{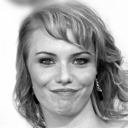}
    \end{subfigure}
    \begin{subfigure}{0.115\textwidth}
        \centering
        \includegraphics[width=0.99\textwidth]{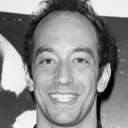}
    \end{subfigure}
    \begin{subfigure}{0.115\textwidth}
        \centering
        \includegraphics[width=0.99\textwidth]{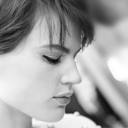}
    \end{subfigure}
    \begin{subfigure}{0.115\textwidth}
        \centering
        \includegraphics[width=0.99\textwidth]{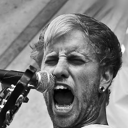}
    \end{subfigure}
    \begin{subfigure}{0.115\textwidth}
        \centering
        \includegraphics[width=0.99\textwidth]{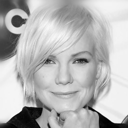}
    \end{subfigure}
    \begin{subfigure}{0.115\textwidth}
        \centering
        \includegraphics[width=0.99\textwidth]{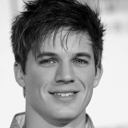}
    \end{subfigure}
    \begin{subfigure}{0.115\textwidth}
        \centering
        \includegraphics[width=0.99\textwidth]{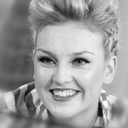}
    \end{subfigure}

    \begin{subfigure}{0.115\textwidth}
        \centering
        \includegraphics[width=0.99\textwidth]{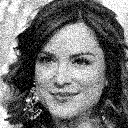}
    \end{subfigure}
    \begin{subfigure}{0.115\textwidth}
        \centering
        \includegraphics[width=0.99\textwidth]{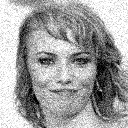}
    \end{subfigure}
    \begin{subfigure}{0.115\textwidth}
        \centering
        \includegraphics[width=0.99\textwidth]{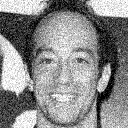}
    \end{subfigure}
    \begin{subfigure}{0.115\textwidth}
        \centering
        \includegraphics[width=0.99\textwidth]{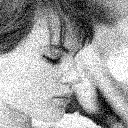}
    \end{subfigure}
    \begin{subfigure}{0.115\textwidth}
        \centering
        \includegraphics[width=0.99\textwidth]{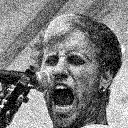}
    \end{subfigure}
    \begin{subfigure}{0.115\textwidth}
        \centering
        \includegraphics[width=0.99\textwidth]{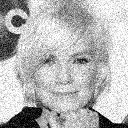}
    \end{subfigure}
    \begin{subfigure}{0.115\textwidth}
        \centering
        \includegraphics[width=0.99\textwidth]{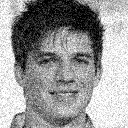}
    \end{subfigure}
    \begin{subfigure}{0.115\textwidth}
        \centering
        \includegraphics[width=0.99\textwidth]{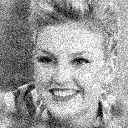}
    \end{subfigure}
    \caption{Faces Dataset\\
    \scriptsize
    A subset of the CelebA dataset corrupted with gaussian noise.}
    \label{fig:faces_dataset}
\end{figure}

In \cref{fig:cost_faces} we plot the cost function corresponding to a scalar parameter and two-dimensional patch parameter along with the cost function corresponding to the optimal value calculated by the algorithm.
It is worth mentioning that when considering a patch-dependent parameter, as it was the case with the single image dataset, a scale dependent structure appears (see \cref{fig:patch_faces}).

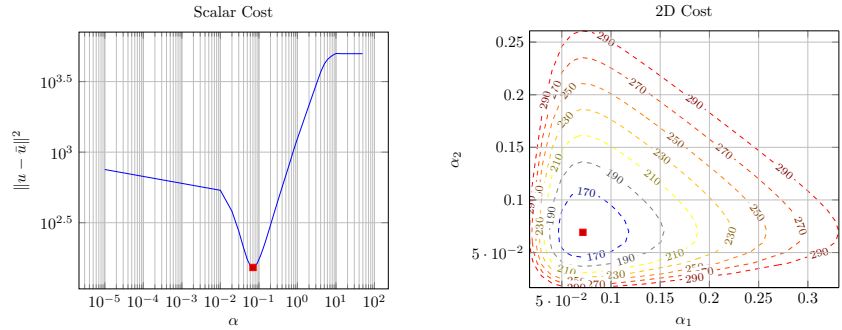
\begin{figure}
    \captionsetup[subfigure]{labelformat=empty}
    \centering
    \begin{subfigure}{0.35\textwidth}
        \centering
        \input{experiments/faces_train_128_10/faces_train_128_10_cost_plot.tex}
    \end{subfigure}
    \begin{subfigure}{0.35\textwidth}
        \centering
        \input{experiments/faces_train_128_10/faces_train_128_10_cost_plot_2d.tex}
    \end{subfigure}
    \caption{Cost Function - Faces Dataset\\
    \scriptsize
    Values for the $l_2$ squared cost function using a scalar regularization parameter and a two dimensional regularization parameter using the Faces dataset.}
    \label{fig:cost_faces}
\end{figure}

\begin{figure}
    \captionsetup[subfigure]{labelformat=empty}
    \centering
    \begin{subfigure}{0.18\textwidth}
        \centering
        \includegraphics[width=0.99\textwidth]{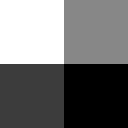}
        \caption{2x2}
    \end{subfigure}
    \begin{subfigure}{0.18\textwidth}
        \centering
        \includegraphics[width=0.99\textwidth]{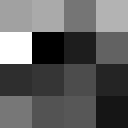}
        \caption{4x4}
    \end{subfigure}
    \begin{subfigure}{0.18\textwidth}
        \centering
        \includegraphics[width=0.99\textwidth]{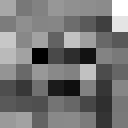}
        \caption{8x8}
    \end{subfigure}
    \begin{subfigure}{0.18\textwidth}
        \centering
        \includegraphics[width=0.99\textwidth]{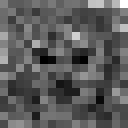}
        \caption{16x16}
    \end{subfigure}

    \caption{Optimal Patch Parameter - Faces Dataset\\
    \scriptsize
    Values for the optimal parameters calculated for different parameter patch sizes.}
    \label{fig:patch_faces}
\end{figure}

For the training dataset proposed, the quality reconstruction results for different number of patches is shown in \Cref{table:faces_algorithm_table}. This table shows the mean values for the SSIM and PSNR quality metric for the reconstructed images from the training dataset. Again, an improvement on the reconstruction quality can be seen as the degrees of freedom for the regularization parameter increases.

\begin{table}
    \centering
    \begin{tabular}{@{}lccccr@{}} \toprule
        \multicolumn{2}{c}{Item}&\multicolumn{2}{c}{}&\multicolumn{2}{c}{Reconstruction} \\ \cmidrule(r){5-6}
        Patch & Iterations & $\|\alpha_{k+1}-\alpha_k\|$ & COST & MSSIM & MPSNR \\ \midrule
        Scalar & 14 & 1.562e-3 & 152.3726 & 0.8141 & 27.0689 \\
        2x1 & 17 & 1.049e-3 & 152.3209 & 0.8142 & 27.0589\\
        2x2 & 21 & 1.762e-4 & 152.3120  & 0.8146 & 27.0739\\
        4x4 & 21 & 7.421e-4 & 152.0911 & 0.8154 & 27.0766\\
        8x8 & 24 & 7.050e-4 & 151.2011 & 0.8173 & 27.0787\\
        16x16 & 25 & 1.410e-3 & 149.9673 & 0.8193 & 27.1047\\
        32x32 & 35 & 3.285e-4 & 147.6573 & 0.8223 & 27.2141\\\bottomrule
    \end{tabular}
    \caption{Trust Region Algorithm behavior on the Faces dataset.}
    \label{table:faces_algorithm_table}
\end{table}

Finally, we can estimate the denoiser performance in images from the testing dataset. This experiment will show the \textit{overfitting} phenomena that may occur when dealing with large number of patches, including the case of scale-dependent parameters ($\alpha\in\R^n$). Indeed, it can be seen in the testing dataset an increment on the mean SSIM (MSSIM) for the reconstructed images from the testing dataset up to a 8x8 patch size. Any higher number of patches results in quality degradation. This is indeed the expected behavior when dealing with overfitting problems.

\begin{table}
    \begin{center}
        \begin{tabular}{@{}cccccccc@{}} \toprule
            img num & noisy & scalar & 2x2 & 4x4 & 8x8 & 16x16 & 32x32 \\ \midrule
            1 & 0.5247 & 0.7951 & 0.7948 & 0.7937 & 0.7956 & 0.7940 & 0.7930\\
            2 & 0.4588 & 0.6614 & 0.6618  & 0.6616 & 0.6625 & 0.6625 & 0.6592\\
            3 & 0.4267 & 0.7719 & 0.7721 & 0.7723 & 0.7747 & 0.7743 & 0.7710\\
            4 & 0.3836 & 0.7237 & 0.7229 & 0.7214 & 0.7235 & 0.7233 & 0.7202\\
            5 & 0.4580 & 0.7812 & 0.7799  & 0.7783 & 0.7798 & 0.7799 & 0.7768\\
            6 & 0.4263 & 0.7400 & 0.7403 & 0.7403 & 0.7420 & 0.7428 & 0.7423\\
            7 & 0.4547 & 0.6321 & 0.6322 & 0.6308 & 0.6305 & 0.6298 & 0.6262\\
            8 & 0.4117 & 0.7381 & 0.7386  & 0.7405 & 0.7416 & 0.7413 & 0.7398\\
            9 & 0.4655 & 0.7430 & 0.7409 & 0.7399 & 0.7405 & 0.7388 & 0.7354\\
            10 & 0.5081 & 0.8210 & 0.8210 & 0.8204 & 0.8211 & 0.8197 &0.8181\\\bottomrule
            \textbf{MSSIM} & & $\bm{0.7408}$ & $\bm{0.7405}$ & $\bm{0.7400}$ & $\bm{0.7412}$ & $\bm{0.7406}$ & $\bm{0.7382}$\\\bottomrule
        \end{tabular}
    \end{center}
    \caption{Faces Dataset SSIM Quality Measures in the validation dataset.}
\end{table}

\subsection{Learning Optimal Total Variation Discretization}
The selection of an adequate discretization of the total variation seminorm in the context of image reconstruction problems is still an open problem \cite{chambolle2020crouzeix,caillaud2020error}. In \cite{chambolle2020learning}, the authors propose a methodology for finding optimal discretization where instead of using hand-crafted discretization schemes for the total variation seminorm, a learning strategy is proposed.

The bilevel framework presented in this work, can also be used to learn optimal gradient discretization, by using different discretization schemes and their corresponding regularization parameters. We will make use of a training dataset to estimate the optimal regularization parameters for the contributions of each discretization scheme into the final solution. Let us define the following variational denoising model
\begin{equation}\label{eq:sumregs_lower_level}
    \min_{u\in\R^n} \Ecal(u) := \frac{1}{2}\|u-f\|^2+\sum_{j=1}^n(\alpha_1)_j\|(\Kbb_1 u)_j\|+\sum_{j=1}^n(\alpha_2)_j\|(\Kbb_2 u)_j\|+\sum_{j=1}^n(\alpha_3)_j\|(\Kbb_3 u)_j\|,
\end{equation}
where $\Kbb_1,\Kbb_2$ and $\Kbb_3$ are the forward, backward and centered finite differences discretization of the gradient operator respectively. The goal is to determine optimal parameters $(\alpha_1,\alpha_2,\alpha_3)^\top \in\R^{3\times n}_+$ that lead to an optimal discretization of the total variation operator. Indeed, we consider the following bilevel learning strategy
\begin{mini}
    {(\alpha_1,\alpha_2,\alpha_3)^\top \in\R^n_+\times\R^n_+\times\R^n_+}{\frac{1}{2}\|\bar{u}(\alpha)-u^{\text{\scriptsize train}}\|^2}{}{}
    \addConstraint{\bar{u} = \argmin_{u\in\R^m}\Ecal(u),}
\end{mini}
By extending the model presented previously, we can make use of a similar analysis using the Bouligand candidate presented in \cref{teo: G is a solution of the linear system} by defining the following adjoint state
\begin{align*}
    \scalar{p}{v} + \sum_{i=1}^3\sum_{j\in\Ical^i}\scalar{\mu_j^i}{(\Kbb^i v)_j} - \scalar{\nabla J(u)}{v} &= 0,\forall v\in \mathcal{V}\\
    \mu^i_j - \alpha_j^i T_j^i(\Kbb^i p)_j &= 0,\;\forall j\in\Ical^i,\;i=1,2,3,\\
\end{align*}
where $\mathcal{V}:=\{v\in\R^m:(\Kbb^i v)_j=0,\;\forall \Acal_s^i\cup\Bcal_1^i,\;(\Kbb^i v)_j\in span(q_j^i),\;\forall j\in\Bcal_2^i,\;i=1,2,3\}$, and the following gradient form
\begin{equation*}
    \scalar{j'(\alpha^i)}{h^i} = -\sum_{j\in\Ical^i\cup\Bcal_2^i}\frac{h_j^i}{\alpha_j^i}\scalar{q_j^i}{(\Kbb^i p)_j}\;\text{for }i=1,2,3.
\end{equation*}

The same procedure applies for the regularized problem. Using the KKT conditions for the differentiable problem we obtain
\begin{equation*}
    \scalar{p}{v} + \sum_{i=1}^3\sum_{j=1}^n \alpha_j^i\scalar{h'^*_\gamma ((\Kbb^i u)_j)(\Kbb^i p)_j}{(\Kbb^i v)_j}  = -\scalar{\nabla J(u)}{v},\;\forall v\in\R^m,\\
\end{equation*}
and its corresponding gradient characterization
\begin{equation*}
    (j'(\alpha^i))_j = \scalar{h_\gamma((\Kbb^i u)_j)}{(\Kbb^i p)_j},\;\forall i=1,2,3.
\end{equation*}

The optimal regularization parameters for 4x4 and 8x8 patch-parameters for each of the three regularizers considered is presented in \cref{fig:patch_cameraman_sumregs} for the cameraman dataset and \cref{fig:patch_faces_sumregs} for the faces dataset. Again, we can see an improvement on the image quality by using more patches in the training datasets. Finally, in \cref{table:faces_sumregs_validation} we can see a comparison between different patch sizes when using these three regularizers. Again, it is worth noticing that in the validation dataset the best reconstruction quality corresponds to an 8x8 patch-parameter size, and any further enlargement of the patch dimension drives a lower reconstruction quality.

\begin{figure}
    \captionsetup[subfigure]{labelformat=empty}
    \centering

    \begin{subfigure}{0.18\textwidth}
        \centering
        \includegraphics[width=0.99\textwidth]{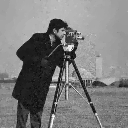}
        \caption{SSIM=0.8187}
    \end{subfigure}
    \begin{subfigure}{0.18\textwidth}
        \centering
        \includegraphics[width=0.99\textwidth]{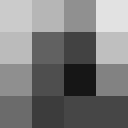}
        \caption{$\alpha_1$}
    \end{subfigure}
    \begin{subfigure}{0.18\textwidth}
        \centering
        \includegraphics[width=0.99\textwidth]{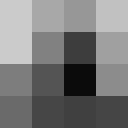}
        \caption{$\alpha_2$}
    \end{subfigure}
    \begin{subfigure}{0.18\textwidth}
        \centering
        \includegraphics[width=0.99\textwidth]{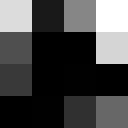}
        \caption{$\alpha_3$}
    \end{subfigure}

    \begin{subfigure}{0.18\textwidth}
        \centering
        \includegraphics[width=0.99\textwidth]{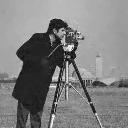}
        \caption{SSIM=0.8281}
    \end{subfigure}
    \begin{subfigure}{0.18\textwidth}
        \centering
        \includegraphics[width=0.99\textwidth]{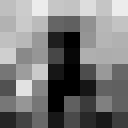}
        \caption{$\alpha_1$}
    \end{subfigure}
    \begin{subfigure}{0.18\textwidth}
        \centering
        \includegraphics[width=0.99\textwidth]{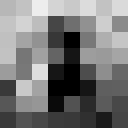}
        \caption{$\alpha_2$}
    \end{subfigure}
    \begin{subfigure}{0.18\textwidth}
        \centering
        \includegraphics[width=0.99\textwidth]{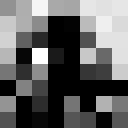}
        \caption{$\alpha_3$}
    \end{subfigure}

    \caption{Optimal Patch Parameters - Cameraman Dataset\\
    \scriptsize
    Values for the optimal parameter calculated for the Cameraman dataset for different patch sizes.}
    \label{fig:patch_cameraman_sumregs}
\end{figure}

\begin{figure}
    \captionsetup[subfigure]{labelformat=empty}
    \centering
    \begin{subfigure}{0.18\textwidth}
        \centering
        \includegraphics[width=0.99\textwidth]{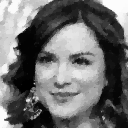}
        \caption{SSIM=0.8094}
    \end{subfigure}
    \begin{subfigure}{0.18\textwidth}
        \centering
        \includegraphics[width=0.99\textwidth]{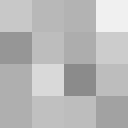}
        \caption{$\alpha_1$}
    \end{subfigure}
    \begin{subfigure}{0.18\textwidth}
        \centering
        \includegraphics[width=0.99\textwidth]{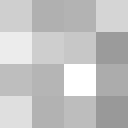}
        \caption{$\alpha_2$}
    \end{subfigure}
    \begin{subfigure}{0.18\textwidth}
        \centering
        \includegraphics[width=0.99\textwidth]{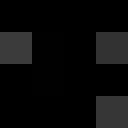}
        \caption{$\alpha_3$}
    \end{subfigure}

    \begin{subfigure}{0.18\textwidth}
        \centering
        \includegraphics[width=0.99\textwidth]{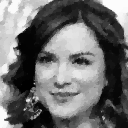}
        \caption{SSIM=0.8130}
    \end{subfigure}
    \begin{subfigure}{0.18\textwidth}
        \centering
        \includegraphics[width=0.99\textwidth]{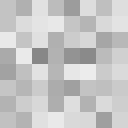}
        \caption{$\alpha_1$}
    \end{subfigure}
    \begin{subfigure}{0.18\textwidth}
        \centering
        \includegraphics[width=0.99\textwidth]{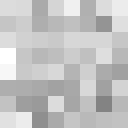}
        \caption{$\alpha_2$}
    \end{subfigure}
    \begin{subfigure}{0.18\textwidth}
        \centering
        \includegraphics[width=0.99\textwidth]{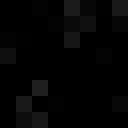}
        \caption{$\alpha_3$}
    \end{subfigure}

    \begin{subfigure}{0.18\textwidth}
        \centering
        \includegraphics[width=0.99\textwidth]{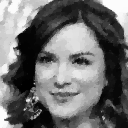}
        \caption{SSIM=0.8153}
    \end{subfigure}
    \begin{subfigure}{0.18\textwidth}
        \centering
        \includegraphics[width=0.99\textwidth]{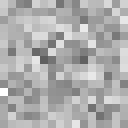}
        \caption{$\alpha_1$}
    \end{subfigure}
    \begin{subfigure}{0.18\textwidth}
        \centering
        \includegraphics[width=0.99\textwidth]{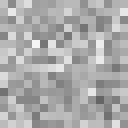}
        \caption{$\alpha_2$}
    \end{subfigure}
    \begin{subfigure}{0.18\textwidth}
        \centering
        \includegraphics[width=0.99\textwidth]{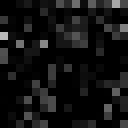}
        \caption{$\alpha_3$}
    \end{subfigure}

    \caption{Optimal Patch Parameters - Faces Dataset\\
    \scriptsize
    Values for the optimal parameters calculated for different parameter patch sizes.}
    \label{fig:patch_faces_sumregs}
\end{figure}

\begin{table}
    \begin{center}
    \begin{tabular}{@{}ccccccc@{}} \toprule
        img num & noisy & scalar & 2x2 & 4x4 & 8x8 & 16x16\\
        \midrule
        1 & 0.5247 & 0.7922 & 0.7952 & 0.7951 & 0.7977 & 0.7965\\
        2 & 0.4588 & 0.6568 & 0.6676 & 0.6663 & 0.6649 & 0.6653\\
        3 & 0.4267 & 0.7745 & 0.7742 & 0.7729 & 0.7748 & 0.7741\\
        4 & 0.3836 & 0.7237 & 0.7239 & 0.7226 & 0.7244 & 0.7206\\
        5 & 0.4580 & 0.7814 & 0.7796 & 0.7779 & 0.7806 & 0.7809\\
        6 & 0.4263 & 0.7394 & 0.7406 & 0.7414 & 0.7432 & 0.7441\\
        7 & 0.4547 & 0.6262 & 0.6372 & 0.6357 & 0.6338 & 0.6333\\
        8 & 0.4117 & 0.7377 & 0.7380 & 0.7413 & 0.7429 & 0.7421\\
        9 & 0.4655 & 0.7438 & 0.7481 & 0.7445 & 0.7432 & 0.7420\\
        10 & 0.5081 & 0.8185 & 0.8225 & 0.8217 & 0.8219 & 0.8211\\
        \hline
        \textbf{MSSIM} & & $\bm{0.7395}$ & $\bm{0.7420}$ & $\bm{0.7419}$ & $\bm{0.7428}$ & $\bm{0.7424}$\\
        \hline
    \end{tabular}
    \end{center}
    \caption{Faces Dataset SSIM Quality Measures - Optimal Gradient Discretization in the validation dataset}
    \label{table:faces_sumregs_validation}
\end{table}

%% file: experiments/cameraman_128_5/cameraman_128_5_cost_plot.tex
\begin{tikzpicture}[,
scale=0.6]
\begin{axis}[
  ylabel = {$\|u-\bar{u}\|^2$},
  title = {Scalar Cost},
  xlabel = {$\alpha$},
  grid=both,
  xmode = {log},
  ymode = {log}
]

\addplot+[
  mark = {none}
] coordinates {
  (0.0001, 39.35183975503215)
  (0.0011, 38.61021138087902)
  (0.0021, 37.93428758113962)
  (0.0031, 37.321827768821386)
  (0.0041, 36.770308396183665)
  (0.0051, 36.27740041877734)
  (0.0061, 35.841343072948135)
  (0.0071, 35.45993597402602)
  (0.0081, 35.13110783608035)
  (0.0091, 34.85182917866388)
  (0.0101, 34.62004168799655)
  (0.0111, 34.43370615295921)
  (0.0121, 34.28945274615154)
  (0.0131, 34.18524996336351)
  (0.0141, 34.11908724031764)
  (0.0151, 34.08880387637571)
  (0.0161, 34.091926255302525)
  (0.0171, 34.12585117144701)
  (0.0181, 34.18902190257802)
  (0.0191, 34.27882266167631)
  (0.0201, 34.394406752816295)
  (0.0211, 34.533629553177335)
  (0.0221, 34.69439469436213)
  (0.0231, 34.87535386728961)
  (0.0241, 35.07450094253086)
  (0.0251, 35.290041608639974)
  (0.0261, 35.52051014867768)
  (0.0271, 35.763836123594764)
  (0.0281, 36.018324573043095)
  (0.0291, 36.28312543397011)
  (0.0301, 36.55739802264843)
  (0.0311, 36.840579664644316)
  (0.0321, 37.1313855927374)
  (0.0331, 37.42844585362433)
  (0.0341, 37.73050429715157)
  (0.0351, 38.03636862038071)
  (0.0361, 38.34646353636534)
  (0.0371, 38.66051153924266)
  (0.0381, 38.97731738961112)
  (0.0391, 39.295656736172695)
  (0.0401, 39.6156940119174)
  (0.0411, 39.93668411208651)
  (0.0421, 40.25887812288611)
  (0.0431, 40.58228912841855)
  (0.0441, 40.9068439756975)
  (0.0451, 41.231291187602686)
  (0.0461, 41.55563507443086)
  (0.0471, 41.880169238241386)
  (0.0481, 42.204282612403496)
  (0.0491, 42.52886648762474)
  (0.0501, 42.85414556695718)
  (0.1501, 72.35345510365228)
  (0.2501, 93.65906927579985)
  (0.3501, 113.75760369909105)
  (0.4501, 128.56465411203482)
  (0.5501, 137.91774682825098)
  (0.6501, 145.05185925193064)
  (0.7501, 151.21353843916958)
  (0.8501, 157.5868450732938)
  (0.9501, 164.1345715697258)
  (1.0501, 170.85766959617274)
  (1.1501, 177.7463009408699)
  (1.2501, 184.78468107153773)
  (1.3501, 191.95235137571663)
  (1.4501, 199.19190695132622)
  (1.5501, 206.44555870470285)
  (1.6501, 213.77032313244592)
  (1.7501, 221.24064518896003)
  (1.8501, 228.83239296080527)
  (1.9501, 236.57170580666835)
  (2.0501, 244.50925582590975)
  (2.1501, 252.63012099461247)
  (2.2501, 260.93770055552426)
  (2.3501, 269.4273661714143)
  (2.4501, 278.08058983079025)
  (2.5501, 286.736376617044)
  (2.6501, 295.5807222389587)
  (2.7501, 304.56752510777625)
  (2.8501, 313.60660161830697)
  (2.9501, 322.55679244391325)
  (3.0501, 329.4002844612731)
  (3.1501, 336.10285810602255)
  (3.2501, 342.7257853477919)
  (3.3501, 349.38323085812175)
  (3.4501, 356.0232338165534)
  (3.5501, 362.43176103726506)
  (3.6501, 368.5694917132454)
  (3.7501, 374.62079163847557)
  (3.8501, 380.5303804205056)
  (3.9501, 386.28721056540616)
  (4.0501, 391.88746782910164)
  (4.1501, 397.37075846455764)
  (4.2501, 402.44621466485455)
  (4.3501, 407.4419630194092)
  (4.4501, 412.3739763354854)
  (4.5501, 417.15459483027905)
  (4.6501, 421.8512235997132)
  (4.7501, 426.47873480438517)
  (4.8501, 431.0528751598738)
  (4.9501, 435.57128504034193)
  (5.01, 438.2527372705775)
  (6.01, 482.4630189997896)
  (7.01, 518.9873313663533)
  (8.01, 517.7780773429049)
  (9.01, 512.7531644121789)
  (10.01, 509.7547639723466)
  (11.01, 507.9872647853066)
  (12.01, 507.20899733403564)
  (13.01, 507.0167777657284)
  (14.01, 507.01432136923995)
  (15.01, 507.01432136923995)
  (16.01, 507.01432136923995)
  (17.01, 507.01432136923995)
  (18.01, 507.01432136923995)
  (19.01, 507.01432136923995)
  (20.01, 507.01432136923995)
  (21.01, 507.01432136923995)
  (22.01, 507.01432136923995)
  (23.01, 507.01432136923995)
  (24.01, 507.01432136923995)
  (25.01, 507.01432136923995)
  (26.01, 507.01432136923995)
  (27.01, 507.01432136923995)
  (28.01, 507.01432136923995)
  (29.01, 507.01432136923995)
  (30.01, 507.01432136923995)
  (31.01, 507.01432136923995)
  (32.01, 507.01432136923995)
  (33.01, 507.01432136923995)
  (34.01, 507.01432136923995)
  (35.01, 507.01432136923995)
  (36.01, 507.01432136923995)
  (37.01, 507.01432136923995)
  (38.01, 507.01432136923995)
  (39.01, 507.01432136923995)
  (40.01, 507.01432136923995)
  (41.01, 507.01432136923995)
  (42.01, 507.01432136923995)
  (43.01, 507.01432136923995)
  (44.01, 507.01432136923995)
  (45.01, 507.01432136923995)
  (46.01, 507.01432136923995)
  (47.01, 507.01432136923995)
  (48.01, 507.01432136923995)
  (49.01, 507.01432136923995)
};

\addplot+[
] coordinates {
  (0.015497558593750004, 34.1077050130809)
};

\end{axis}
\end{tikzpicture}

%% file: experiments/cameraman_128_5/cameraman_128_5_cost_plot_2d.tex
\begin{tikzpicture}[,
  scale=0.6]
  \begin{axis}[
    ylabel = {$\alpha_2$},
    title = {2D Cost},
    xlabel = {$\alpha_1$},
    grid=both,
    view = {{0}{90}},
  ]
  
  \addplot3[
    contour prepared,
    dashed
  ] table {
    0.018563569859214587 0.00901 33.75
    0.01801 0.00968875047861055 33.75
    0.017819135099588393 0.01001 33.75
    0.017655721256832436 0.01101 33.75
    0.01788677545103021 0.01201 33.75
    0.01801 0.012211636490759471 33.75
    0.018683157525183156 0.01301 33.75
    0.01901 0.013254680735155917 33.75
    0.02001 0.013736481253490895 33.75
    0.02101 0.013971659551795898 33.75
    0.02201 0.013983720626570809 33.75
    0.02301 0.013790165805176394 33.75
    0.02401 0.013405114397480312 33.75
    0.024715780307770156 0.01301 33.75
    0.02501 0.01275126377384967 33.75
    0.02566599318040244 0.01201 33.75
    0.02601 0.011015058801183857 33.75
    0.02601146985806056 0.01101 33.75
    0.02601 0.01100194318158143 33.75
    0.02579385241130205 0.01001 33.75
    0.02501 0.009075703668560592 33.75
    0.024938788257928193 0.00901 33.75
    0.02401 0.008525304696667482 33.75
    0.02301 0.008170027994183869 33.75
    0.022089262763077572 0.00801 33.75
    0.02201 0.008000546187171075 33.75
    0.021316373235800094 0.00801 33.75
    0.02101 0.008016067130701956 33.75
    0.02001 0.008243041501319178 33.75
    0.01901 0.008697517163520827 33.75
    0.018563569859214587 0.00901 33.75
  
    0.013846994686358242 0.01801 34.5
    0.01401 0.018165722804559 34.5
    0.01501 0.01897030601212906 34.5
    0.015069943514555015 0.01901 34.5
    0.01601 0.019574658790674205 34.5
    0.01691352007089102 0.02001 34.5
    0.01701 0.020052839001169206 34.5
    0.01801 0.020394024107547676 34.5
    0.01901 0.02064073151571528 34.5
    0.02001 0.02079845316130236 34.5
    0.02101 0.020873932035981882 34.5
    0.02201 0.02087481071970599 34.5
    0.02301 0.0208067517333051 34.5
    0.02401 0.02067455667174725 34.5
    0.02501 0.020483569395303006 34.5
    0.02601 0.0202388727729086 34.5
    0.026797418884132885 0.02001 34.5
    0.02701 0.019943052124502283 34.5
    0.02801 0.0195844057006502 34.5
    0.02901 0.019186470604641178 34.5
    0.02941912154445923 0.01901 34.5
    0.03001 0.018729501627810057 34.5
    0.03101 0.018218154187954874 34.5
    0.03139312194075175 0.01801 34.5
    0.03201 0.017637346698057277 34.5
    0.032999376422837096 0.01701 34.5
    0.03301 0.017002304919833067 34.5
    0.03401 0.01624859234688737 34.5
    0.034315038408345586 0.01601 34.5
    0.03501 0.015372315787971873 34.5
    0.03539133949756405 0.01501 34.5
    0.03601 0.014285584270882104 34.5
    0.036238207964233216 0.01401 34.5
    0.036855049880648465 0.01301 34.5
    0.03701 0.012618060654088197 34.5
    0.03724495127523095 0.01201 34.5
    0.03740059692880591 0.01101 34.5
    0.03731576125094105 0.01001 34.5
    0.03701 0.009132525725512057 34.5
    0.036966260967821594 0.00901 34.5
    0.03632889676314659 0.00801 34.5
    0.03601 0.007668145609797478 34.5
    0.03537562481136062 0.00701 34.5
    0.03501 0.006726631836328734 34.5
    0.03405065495586912 0.00601 34.5
    0.03401 0.005985935427281625 34.5
    0.03301 0.005416524729264596 34.5
    0.032265839252390766 0.00501 34.5
    0.03201 0.004894925153261548 34.5
    0.03101 0.004467570385281977 34.5
    0.03001 0.0040663117549700045 34.5
    0.029858447041983677 0.00401 34.5
    0.02901 0.003743011911474349 34.5
    0.02801 0.0034535565501949694 34.5
    0.02701 0.0031936563189674857 34.5
    0.026201697430996786 0.00301 34.5
    0.02601 0.0029724983319251697 34.5
    0.02501 0.002808994532192067 34.5
    0.02401 0.0026828744598605663 34.5
    0.02301 0.0025975802739843998 34.5
    0.02201 0.002556850869883585 34.5
    0.02101 0.002563744251288129 34.5
    0.02001 0.0026218555438316644 34.5
    0.01901 0.0027362792716972887 34.5
    0.01801 0.002911468863036615 34.5
    0.017598878186647675 0.00301 34.5
    0.01701 0.003173701280926879 34.5
    0.01601 0.0035332210477947617 34.5
    0.01501 0.003979279897179986 34.5
    0.014953142291842395 0.00401 34.5
    0.01401 0.004610586325091224 34.5
    0.013480265913142142 0.00501 34.5
    0.01301 0.005439493709943665 34.5
    0.012474447914720011 0.00601 34.5
    0.01201 0.006631873950565695 34.5
    0.011764199316744155 0.00701 34.5
    0.011275223748272667 0.00801 34.5
    0.01101 0.008794987576962544 34.5
    0.010946152836695347 0.00901 34.5
    0.010775776666731707 0.01001 34.5
    0.010727369050577244 0.01101 34.5
    0.010790205465890048 0.01201 34.5
    0.010960064661094926 0.01301 34.5
    0.01101 0.013194900273105721 34.5
    0.011260608289965102 0.01401 34.5
    0.011677505546216384 0.01501 34.5
    0.01201 0.01565316394395151 34.5
    0.012222240497126441 0.01601 34.5
    0.01292287517302949 0.01701 34.5
    0.01301 0.01711849492303486 34.5
    0.013846994686358242 0.01801 34.5
  
    0.01745085447694348 1.0e-5 35.25
    0.01701 9.21392765055523e-5 35.25
    0.01601 0.0003328191523708677 35.25
    0.01501 0.0006312054986193404 35.25
    0.01401 0.0009923565665556956 35.25
    0.013968742832151803 0.00101 35.25
    0.01301 0.0014732406838173936 35.25
    0.01205719449070188 0.00201 35.25
    0.01201 0.002040397224980868 35.25
    0.01101 0.002779921238776831 35.25
    0.010736278826246054 0.00301 35.25
    0.01001 0.0037178397491905257 35.25
    0.00974427660039758 0.00401 35.25
    0.00901 0.004961179974202321 35.25
    0.008976307233869104 0.00501 35.25
    0.008406507001010012 0.00601 35.25
    0.00801 0.006884469819856188 35.25
    0.007958764784418947 0.00701 35.25
    0.0076517338731149406 0.00801 35.25
    0.007439478120413196 0.00901 35.25
    0.007317386218584749 0.01001 35.25
    0.007282194857004063 0.01101 35.25
    0.007326481820237479 0.01201 35.25
    0.007447041268059539 0.01301 35.25
    0.00763895883029353 0.01401 35.25
    0.007899305830986017 0.01501 35.25
    0.00801 0.015352827346753604 35.25
    0.008245894791423528 0.01601 35.25
    0.008668748183149737 0.01701 35.25
    0.00901 0.01771407953954384 35.25
    0.009170586201843148 0.01801 35.25
    0.009775596220728452 0.01901 35.25
    0.01001 0.019361386108284397 35.25
    0.010498689896616157 0.02001 35.25
    0.01101 0.020635409207111478 35.25
    0.011358646619594109 0.02101 35.25
    0.01201 0.021663205936432443 35.25
    0.01240792859879119 0.02201 35.25
    0.01301 0.022503503995741605 35.25
    0.013732530389680694 0.02301 35.25
    0.01401 0.02319496887034647 35.25
    0.01501 0.02375624064672648 35.25
    0.01555951306201906 0.02401 35.25
    0.01601 0.024209403796631272 35.25
    0.01701 0.02456426074924361 35.25
    0.01801 0.024836636417906517 35.25
    0.018891273949891454 0.02501 35.25
    0.01901 0.025032506990509403 35.25
    0.02001 0.025153547555292345 35.25
    0.02101 0.025211135407711464 35.25
    0.02201 0.025211083044622558 35.25
    0.02301 0.02515777207632801 35.25
    0.02401 0.025054660455234495 35.25
    0.02431052772500257 0.02501 35.25
    0.02501 0.024902169060799725 35.25
    0.02601 0.02470510255727998 35.25
    0.02701 0.024471051384786714 35.25
    0.02801 0.024204527136372375 35.25
    0.028667809966236642 0.02401 35.25
    0.02901 0.023904471316516705 35.25
    0.03001 0.02357028952248887 35.25
    0.03101 0.023210507011513566 35.25
    0.03153472969314366 0.02301 35.25
    0.03201 0.022819163171114022 35.25
    0.03301 0.02239730831222571 35.25
    0.033892417858518695 0.02201 35.25
    0.03401 0.021955187681352385 35.25
    0.03501 0.021471662495155457 35.25
    0.03593200051366363 0.02101 35.25
    0.03601 0.020968203798994563 35.25
    0.03701 0.02041555625340587 35.25
    0.03772738375831312 0.02001 35.25
    0.03801 0.01983690164744846 35.25
    0.03901 0.01921202611514009 35.25
  
    0.026439129556310947 1.0e-5 35.25
    0.02701 9.584555434792398e-5 35.25
    0.02801 0.0002680008899992149 35.25
    0.02901 0.0004598350939741348 35.25
    0.03001 0.0006683456678374547 35.25
    0.03101 0.0008935866468243621 35.25
    0.03149527165792019 0.00101 35.25
    0.03201 0.001149690961432178 35.25
    0.03301 0.001435292992942094 35.25
    0.03401 0.0017329976499036602 35.25
    0.03490497593196448 0.00201 35.25
    0.03501 0.002047225555832687 35.25
    0.03601 0.0024149617873064767 35.25
    0.03701 0.002794772593726129 35.25
    0.037562878447766526 0.00301 35.25
    0.03801 0.0032122538993923535 35.25
    0.03901 0.0036746652094103672 35.25
  
    0.03901 0.02390896200609117 36.0
    0.03880091882144055 0.02401 36.0
    0.03801 0.02437659908272052 36.0
    0.03701 0.024831005407985725 36.0
    0.03660732990222901 0.02501 36.0
    0.03601 0.025266049891016867 36.0
    0.03501 0.025681645795242382 36.0
    0.0341923938560923 0.02601 36.0
    0.03401 0.026081040363823558 36.0
    0.03301 0.026456602091677703 36.0
    0.03201 0.026817328607169743 36.0
    0.031448878696946486 0.02701 36.0
    0.03101 0.02715697136999304 36.0
    0.03001 0.02747212977296803 36.0
    0.02901 0.0277647526269293 36.0
    0.02810155355606478 0.02801 36.0
    0.02801 0.028034133379940325 36.0
    0.02701 0.028271728801797846 36.0
    0.02601 0.02848054819893933 36.0
    0.02501 0.028656485678412693 36.0
    0.02401 0.028794316571555255 36.0
    0.02301 0.028890102922704865 36.0
    0.02201 0.028940015541394915 36.0
    0.02101 0.028940707580878456 36.0
    0.02001 0.028888236836763576 36.0
    0.01901 0.02877709190543661 36.0
    0.01801 0.028602647034500123 36.0
    0.01701 0.028360972996706552 36.0
    0.01601 0.02804620708041264 36.0
    0.01591778464648437 0.02801 36.0
    0.01501 0.02764469301915163 36.0
    0.01401 0.027155370004157145 36.0
    0.013759908488474687 0.02701 36.0
    0.01301 0.026562728745702764 36.0
    0.012217419818578403 0.02601 36.0
    0.01201 0.02586071808896743 36.0
    0.01101 0.025032635535059636 36.0
    0.010985989143837704 0.02501 36.0
    0.01001 0.024055109911071588 36.0
    0.009969175617314578 0.02401 36.0
    0.00910180035956694 0.02301 36.0
    0.00901 0.02289865941886554 36.0
    0.008355678592185981 0.02201 36.0
    0.00801 0.02151071582459328 36.0
    0.007698166707519003 0.02101 36.0
    0.0071169641073041716 0.02001 36.0
    0.00701 0.019810350920728245 36.0
    0.0066221822830727275 0.01901 36.0
    0.006182707245597417 0.01801 36.0
    0.00601 0.01757194161351368 36.0
    0.005808456710627414 0.01701 36.0
    0.005495604784000094 0.01601 36.0
    0.005230019482268864 0.01501 36.0
    0.0050158889922109 0.01401 36.0
    0.00501 0.013972681986792145 36.0
    0.004871194863000248 0.01301 36.0
    0.004780751817629484 0.01201 36.0
    0.004747703427232494 0.01101 36.0
    0.004774441902847115 0.01001 36.0
    0.004866511984990914 0.00901 36.0
    0.00501 0.008112832599795346 36.0
    0.005027995628459062 0.00801 36.0
    0.005281098655768927 0.00701 36.0
    0.005617453241106985 0.00601 36.0
    0.00601 0.0050814751934415165 36.0
    0.006043203219577258 0.00501 36.0
    0.006605843053814988 0.00401 36.0
    0.00701 0.003400280358482171 36.0
    0.007295812124412778 0.00301 36.0
    0.00801 0.00216889236739699 36.0
    0.008159833555987695 0.00201 36.0
    0.00901 0.0012213788786867913 36.0
    0.009264834564577858 0.00101 36.0
    0.01001 0.0004628625209106008 36.0
    0.01070547000612054 1.0e-5 36.0
  
    0.037710413426376356 1.0e-5 36.0
    0.03801 9.984120880980495e-5 36.0
    0.03901 0.00040640522777887577 36.0
  
    0.007277533734935226 1.0e-5 36.75
    0.00701 0.00025396754177609124 36.75
    0.006259428884507691 0.00101 36.75
    0.00601 0.0012938180782670933 36.75
    0.005437056826436073 0.00201 36.75
    0.00501 0.002620271456311612 36.75
    0.004760686463112165 0.00301 36.75
    0.0042089993912472635 0.00401 36.75
    0.00401 0.004434995120450802 36.75
    0.0037628168316270853 0.00501 36.75
    0.003407902660696522 0.00601 36.75
    0.0031255200391389964 0.00701 36.75
    0.00301 0.007553612023197098 36.75
    0.002920551528053842 0.00801 36.75
    0.0027850418750347316 0.00901 36.75
    0.002706981488813808 0.01001 36.75
    0.002684235093305693 0.01101 36.75
    0.0027120599363442506 0.01201 36.75
    0.0027884470049944887 0.01301 36.75
    0.002910346696255101 0.01401 36.75
    0.00301 0.014612042751993182 36.75
    0.0030814362657013115 0.01501 36.75
    0.0033040730694219728 0.01601 36.75
    0.0035662454803990014 0.01701 36.75
    0.0038668702364050137 0.01801 36.75
    0.00401 0.018437082625048267 36.75
    0.004219205435957668 0.01901 36.75
    0.004621777912624871 0.02001 36.75
    0.00501 0.020898912778152985 36.75
    0.005063100994460726 0.02101 36.75
    0.005575416816662258 0.02201 36.75
    0.00601 0.022807531204867972 36.75
    0.006131282953644492 0.02301 36.75
    0.006761471239321569 0.02401 36.75
    0.00701 0.024387893796847342 36.75
    0.0074623310800552255 0.02501 36.75
    0.00801 0.025735827577500484 36.75
    0.008239955122561473 0.02601 36.75
    0.00901 0.026899649839370505 36.75
    0.009116957068547218 0.02701 36.75
    0.01001 0.027907952482635685 36.75
    0.01012462834520235 0.02801 36.75
    0.01101 0.028779118514297423 36.75
    0.011312598169938482 0.02901 36.75
    0.01201 0.029529659244679003 36.75
    0.01275178157135479 0.03001 36.75
    0.01301 0.03017400420589915 36.75
    0.01401 0.030717232142195787 36.75
    0.01465023300399116 0.03101 36.75
    0.01501 0.031171744461019824 36.75
    0.01601 0.03154158299544435 36.75
    0.01701 0.03183801250125576 36.75
    0.01776569774527 0.03201 36.75
    0.01801 0.03206488511429466 36.75
    0.01901 0.032226997182980985 36.75
    0.02001 0.03233024099185403 36.75
    0.02101 0.03237889827807641 36.75
    0.02201 0.03237809689196406 36.75
    0.02301 0.032331506151746035 36.75
    0.02401 0.03224211797067675 36.75
    0.02501 0.03211344808889823 36.75
    0.025640063488217294 0.03201 36.75
    0.02601 0.03194848765527479 36.75
    0.02701 0.031751344327930635 36.75
    0.02801 0.03152705572672668 36.75
    0.02901 0.031278219149197595 36.75
    0.03000445585639791 0.03101 36.75
    0.03001 0.03100847949001228 36.75
    0.03101 0.0307129525300096 36.75
    0.03201 0.03039904916740187 36.75
    0.03301 0.0300692868184619 36.75
    0.033182618753951024 0.03001 36.75
    0.03401 0.029720215819199435 36.75
    0.03501 0.029357181213175175 36.75
    0.03593384927739114 0.02901 36.75
    0.03601 0.028980727237524474 36.75
    0.03701 0.02858390105351237 36.75
    0.03801 0.0281784763840384 36.75
    0.038417223083651385 0.02801 36.75
    0.03901 0.027759011324383343 36.75
  
    0.004748629672978253 1.0e-5 37.5
    0.00401 0.0009038542234987277 37.5
    0.003929456712682747 0.00101 37.5
    0.0032578135071128604 0.00201 37.5
    0.00301 0.0024318240691661073 37.5
    0.0026966795535052416 0.00301 37.5
    0.002229342825638387 0.00401 37.5
    0.00201 0.0045630143215278935 37.5
    0.0018459541814744868 0.00501 37.5
    0.0015429201601924601 0.00601 37.5
    0.001301810552834196 0.00701 37.5
    0.001120379252261141 0.00801 37.5
    0.00101 0.0088893026552851 37.5
    0.0009959345008284455 0.00901 37.5
    0.000928769323651756 0.01001 37.5
    0.0009091332209243677 0.01101 37.5
    0.0009329076786859516 0.01201 37.5
    0.0009984719159856065 0.01301 37.5
    0.00101 0.013120260207984194 37.5
    0.0011102224345029517 0.01401 37.5
    0.0012632602480053623 0.01501 37.5
    0.0014530831694864847 0.01601 37.5
    0.0016766141393686927 0.01701 37.5
    0.0019330212431847903 0.01801 37.5
    0.00201 0.01827934423553882 37.5
    0.0022357331272723417 0.01901 37.5
    0.002576293718317635 0.02001 37.5
    0.0029458147458549418 0.02101 37.5
    0.00301 0.021172077380996687 37.5
    0.0033698845745700444 0.02201 37.5
    0.003826766407157416 0.02301 37.5
    0.00401 0.02339120209837316 37.5
    0.004334095340668639 0.02401 37.5
    0.004880718648318497 0.02501 37.5
    0.00501 0.025237864533303234 37.5
    0.005489439349097785 0.02601 37.5
    0.00601 0.02682247595904384 37.5
    0.006142077052395136 0.02701 37.5
    0.006864696353361883 0.02801 37.5
    0.00701 0.028206218234551262 37.5
    0.007668056646714168 0.02901 37.5
    0.00801 0.029417866129777658 37.5
    0.00856185497637382 0.03001 37.5
    0.00901 0.03048154269775407 37.5
    0.009572342679489393 0.03101 37.5
    0.01001 0.031414299514861584 37.5
    0.010738385488433085 0.03201 37.5
    0.01101 0.03222928167250572 37.5
    0.01201 0.032938503610093844 37.5
    0.01212598924781745 0.03301 37.5
    0.01301 0.033550001482248897 37.5
    0.013890466466229948 0.03401 37.5
    0.01401 0.034072121278873685 37.5
    0.01501 0.03450957543027688 37.5
    0.01601 0.03486941973978109 37.5
    0.016497411489316792 0.03501 37.5
    0.01701 0.03515698029852624 37.5
    0.01801 0.03537709876895978 37.5
    0.01901 0.035535886025015104 37.5
    0.02001 0.03563693156710627 37.5
    0.02101 0.03568445880386299 37.5
    0.02201 0.035683524289814625 37.5
    0.02301 0.03563774959247264 37.5
    0.02401 0.0355500318182035 37.5
    0.02501 0.035423855080946705 37.5
    0.02601 0.0352628655378763 37.5
    0.02701 0.03507207702164278 37.5
    0.027295687372305602 0.03501 37.5
    0.02801 0.03485387427188277 37.5
    0.02901 0.034611321482078586 37.5
    0.03001 0.03434830294750546 37.5
    0.03101 0.0340648430435686 37.5
    0.031192104604533175 0.03401 37.5
    0.03201 0.03376253227758316 37.5
    0.03301 0.03344470973455305 37.5
    0.03401 0.03311364739922438 37.5
    0.034312102879838144 0.03301 37.5
    0.03501 0.032768461387353696 37.5
    0.03601 0.03241014366055289 37.5
    0.03701 0.03204026863854766 37.5
    0.03709006659444664 0.03201 37.5
    0.03801 0.03165785528515047 37.5
    0.03901 0.031267603972908224 37.5
  
    1.0e-5 0.016961452932767854 38.25
    2.0071792952561853e-5 0.01701 38.25
    0.0002580564527841987 0.01801 38.25
    0.0005233056549106467 0.01901 38.25
    0.0008156223227332526 0.02001 38.25
    0.00101 0.020622683687074077 38.25
    0.0011423885010649823 0.02101 38.25
    0.001508686123586409 0.02201 38.25
    0.001898324114154359 0.02301 38.25
    0.00201 0.023282415894614383 38.25
    0.0023324425919612635 0.02401 38.25
    0.002795022476975577 0.02501 38.25
    0.00301 0.025457785529579956 38.25
    0.0032975222571871065 0.02601 38.25
    0.0038348373892179714 0.02701 38.25
    0.00401 0.027327737432871282 38.25
    0.004419797033031991 0.02801 38.25
    0.00501 0.028968902442848107 38.25
    0.005037684380126838 0.02901 38.25
    0.005727479386049517 0.03001 38.25
    0.00601 0.030411596946550514 38.25
    0.006472771181696954 0.03101 38.25
    0.00701 0.0316928394543737 38.25
    0.007285863930472392 0.03201 38.25
    0.00801 0.03283180136483487 38.25
    0.00818453989833504 0.03301 38.25
    0.00901 0.033845109420176954 38.25
    0.00919246353991468 0.03401 38.25
    0.01001 0.03474500660265807 38.25
    0.010342928916795231 0.03501 38.25
    0.01101 0.03553793493279817 38.25
    0.011689045993131973 0.03601 38.25
    0.01201 0.036231594278364684 38.25
    0.01301 0.036831655414245316 38.25
    0.013357510847033021 0.03701 38.25
    0.01401 0.03734331364650483 38.25
    0.01501 0.03777332375559851 38.25
    0.015679054266167145 0.03801 38.25
    0.01601 0.03812707604747762 38.25
    0.01701 0.038410561781359265 38.25
    0.01801 0.038628174107137675 38.25
    0.01901 0.03878521076661362 38.25
    0.02001 0.038885229717990265 38.25
    0.02101 0.03893233325958894 38.25
    0.02201 0.03893151171307299 38.25
    0.02301 0.03888638433418688 38.25
    0.02401 0.03879963649368849 38.25
    0.02501 0.038674674356661116 38.25
    0.02601 0.03851554763429825 38.25
    0.02701 0.038326861757995066 38.25
    0.02801 0.03811193260474257 38.25
    0.028437086224378162 0.03801 38.25
    0.02901 0.037873259686927764 38.25
    0.03001 0.03761436982814868 38.25
    0.03101 0.03733543589764034 38.25
    0.03201 0.03703898016195186 38.25
    0.032103007299826465 0.03701 38.25
    0.03301 0.036726213209531754 38.25
    0.03401 0.036400427823914845 38.25
    0.03501 0.036062981085547896 38.25
    0.035161634481881594 0.03601 38.25
    0.03601 0.03571160357560958 38.25
    0.03701 0.03534860478463739 38.25
    0.037923009684495634 0.03501 38.25
    0.03801 0.03497755438555571 38.25
    0.03901 0.03459707389531132 38.25
  
    0.0026874755613002874 1.0e-5 38.25
    0.00201 0.0009776829242438075 38.25
    0.0019890622379518776 0.00101 38.25
    0.0014156381874175243 0.00201 38.25
    0.00101 0.0028187477732049293 38.25
    0.0009209449317753181 0.00301 38.25
    0.0005193933249243835 0.00401 38.25
    0.00017856419961922635 0.00501 38.25
    1.0e-5 0.0056089733229603205 38.25
  
    1.0e-5 0.02241374570205252 39.0
    0.00022559707425995982 0.02301 39.0
    0.0006060508751931188 0.02401 39.0
    0.0010031847164450824 0.02501 39.0
    0.00101 0.025026536445360147 39.0
    0.0014467162435449334 0.02601 39.0
    0.0019050260332217968 0.02701 39.0
    0.00201 0.02723322500418287 39.0
    0.0024048834625504016 0.02801 39.0
    0.0029258074968040422 0.02901 39.0
    0.00301 0.029167839549308756 39.0
    0.0034971596705210784 0.03001 39.0
    0.00401 0.030879323899845323 39.0
    0.004093996610683515 0.03101 39.0
    0.004747959677802095 0.03201 39.0
    0.00501 0.03240554345784989 39.0
    0.005448269945419301 0.03301 39.0
    0.00601 0.033777733052239974 39.0
    0.006196828489755208 0.03401 39.0
    0.007005318491625848 0.03501 39.0
    0.00701 0.03501575750767083 39.0
    0.00790369924717573 0.03601 39.0
    0.00801 0.03612746389461668 39.0
    0.00889773326886783 0.03701 39.0
    0.00901 0.03712109236381506 39.0
    0.01001 0.03800490360175638 39.0
    0.010016514262211701 0.03801 39.0
    0.01101 0.038787225008460484 39.0
    0.011334180184736895 0.03901 39.0
  
    0.035859714077912164 0.03901 39.0
    0.03601 0.038957703555240246 39.0
    0.03701 0.03859858827726108 39.0
    0.03801 0.038231557299495934 39.0
    0.038601710721184485 0.03801 39.0
    0.03901 0.03785712699042312 39.0
  
    0.0009136772988265323 1.0e-5 39.0
    0.0003121083722235394 0.00101 39.0
    1.0e-5 0.001577386094880089 39.0
  
    1.0e-5 0.02648876288522713 39.75
    0.00023171017664341595 0.02701 39.75
    0.0006681880309039983 0.02801 39.75
    0.00101 0.028774184791624306 39.75
    0.0011236440012777165 0.02901 39.75
    0.0016170976861097463 0.03001 39.75
    0.00201 0.030790819623052045 39.75
    0.0021292259846832817 0.03101 39.75
    0.0026826309407937464 0.03201 39.75
    0.00301 0.032593983984677256 39.75
    0.0032629156460954106 0.03301 39.75
    0.003876468045018508 0.03401 39.75
    0.00401 0.03422648545630513 39.75
    0.004536578261967864 0.03501 39.75
    0.00501 0.035710426698421165 39.75
    0.005231601116539121 0.03601 39.75
    0.005976356589729233 0.03701 39.75
    0.00601 0.03705496697388248 39.75
    0.006795417122040673 0.03801 39.75
    0.00701 0.03827088085044119 39.75
    0.007682090797872263 0.03901 39.75
  
    1.0e-5 0.030107298336985125 40.5
    0.00043157822215903396 0.03101 40.5
    0.0009066940357889154 0.03201 40.5
    0.00101 0.03222463360517526 40.5
    0.0014172796030490966 0.03301 40.5
    0.0019406338875106988 0.03401 40.5
    0.00201 0.03414182742793743 40.5
    0.002503825808651769 0.03501 40.5
    0.00301 0.03589483850267001 40.5
    0.0030814431187074624 0.03601 40.5
    0.0037060117035669616 0.03701 40.5
    0.00401 0.03749443343155764 40.5
    0.004362497033581176 0.03801 40.5
    0.00501 0.038956924886103375 40.5
    0.005049717324896006 0.03901 40.5
  
    0.002935526055850745 0.03901 41.25
    0.00235678358450138 0.03801 41.25
    0.00201 0.03741073338125293 41.25
    0.0017954793859531462 0.03701 41.25
    0.0012626924262048361 0.03601 41.25
    0.00101 0.03553252396350412 41.25
    0.00075333769603054 0.03501 41.25
    0.0002649288983745463 0.03401 41.25
    1.0e-5 0.03348520203420637 41.25
  
    1.0e-5 0.036766545032044934 42.0
    0.00013040199062343486 0.03701 42.0
    0.0006272910061547682 0.03801 42.0
    0.00101 0.038780120247376884 42.0
    0.001133080787078403 0.03901 42.0
  
  };

  \addplot+[
] coordinates {
  (0.021581060791015627, 0.010907940673828128)
};
  
  \end{axis}
  \end{tikzpicture}
  
  

%% file: experiments/faces_train_128_10/faces_train_128_10_cost_plot.tex
\begin{tikzpicture}[scale=0.6]
\begin{axis}[
  ylabel = {$\|u-\bar{u}\|^2$},
  title = {Scalar Cost},
  ymode = {log},
  xlabel = {$\alpha$},
  grid=both,
  xmode = {log}
]

\addplot+[
  mark = {none}
] coordinates {
  (1.0e-5, 753.469343580276)
  (0.01001, 536.6644342600229)
  (0.02001, 381.0343216160397)
  (0.03001, 275.8512313710588)
  (0.04001, 210.14444485031302)
  (0.05001, 173.52152750513756)
  (0.06001, 156.73988687100208)
  (0.07001, 152.46392937555888)
  (0.08001, 155.62508118422144)
  (0.09001, 163.0254559992431)
  (0.11, 184.0827775870888)
  (0.13, 208.99057333050303)
  (0.15, 235.4488174252568)
  (0.17, 262.47719435982197)
  (0.19, 289.9762307565396)
  (0.21, 317.5468878881827)
  (0.23, 344.9762874184793)
  (0.25, 372.03953244615167)
  (0.27, 398.6204858873613)
  (0.29, 424.8022936539642)
  (0.31, 450.7935686229026)
  (0.33, 476.36262691901374)
  (0.35, 501.7655504732943)
  (0.37, 526.9648249192692)
  (0.39, 551.9138652186489)
  (0.41, 576.4234359681707)
  (0.43, 600.5309672692035)
  (0.45, 624.5713000976382)
  (0.47, 648.4936556340472)
  (0.49, 672.1568326447839)
  (0.51, 695.9240418038837)
  (0.53, 719.6530550158086)
  (0.55, 743.1132532355006)
  (0.57, 766.6378341147921)
  (0.59, 790.1335425142015)
  (0.61, 813.6041373643867)
  (0.63, 837.067686235155)
  (0.65, 860.6581251096322)
  (0.67, 884.3099485593881)
  (0.69, 907.9881463569114)
  (0.71, 931.5327134474526)
  (0.73, 954.8994559321992)
  (0.75, 978.2306000153282)
  (0.77, 1001.3211713748867)
  (0.79, 1023.9509909535843)
  (0.81, 1046.3504097046252)
  (0.83, 1068.6023919227202)
  (0.85, 1090.5856601187954)
  (0.87, 1112.3942726793862)
  (0.89, 1134.2488368581767)
  (0.91, 1155.8529397524328)
  (0.93, 1176.8000475544115)
  (0.95, 1197.3924649575013)
  (0.97, 1217.55204957252)
  (0.99, 1237.2945888308288)
  (1.01, 1256.8098303097322)
  (1.03, 1276.3640127633012)
  (1.05, 1295.8858126994699)
  (1.07, 1315.304100922705)
  (1.09, 1334.6268711715904)
  (1.11, 1353.4659411142047)
  (1.13, 1372.3121944330983)
  (1.15, 1391.1892701909649)
  (1.17, 1410.145525556456)
  (1.19, 1428.9343622941935)
  (1.21, 1447.6791015426406)
  (1.23, 1466.485644690028)
  (1.25, 1485.353761340444)
  (1.27, 1504.1844611890294)
  (1.29, 1522.8352783171567)
  (1.31, 1541.3547367205142)
  (1.33, 1559.8084501225567)
  (1.35, 1578.0645663736132)
  (1.37, 1596.46934483546)
  (1.39, 1614.911378748345)
  (1.41, 1633.2852064553836)
  (1.43, 1651.7158054863098)
  (1.45, 1670.2551185315635)
  (1.47, 1688.7637661601539)
  (1.49, 1707.2573746333908)
  (1.51, 1725.6010975338531)
  (1.53, 1743.7110957027226)
  (1.55, 1761.5663033914232)
  (1.57, 1779.5151855266663)
  (1.59, 1797.442353293292)
  (1.61, 1815.0060821392158)
  (1.63, 1832.5897235470622)
  (1.65, 1850.0793647654248)
  (1.67, 1867.3890930960283)
  (1.69, 1884.7228875558708)
  (1.71, 1902.1175754002627)
  (1.73, 1919.5556848541216)
  (1.75, 1937.0029545199127)
  (1.77, 1954.4668222072664)
  (1.79, 1971.9839600292996)
  (1.81, 1989.5792727975002)
  (1.83, 2007.1859035228404)
  (1.85, 2024.3981630359542)
  (1.87, 2041.3689179368048)
  (1.89, 2058.3141324012327)
  (1.91, 2075.1448674800695)
  (1.93, 2091.936737783555)
  (1.95, 2108.744963054511)
  (1.97, 2125.566080328848)
  (1.99, 2142.4594822514528)
  (2.01, 2159.4269976523556)
  (3.01, 2962.482365416407)
  (4.01, 3693.852424681464)
  (5.01, 4217.566685950918)
  (6.01, 4527.555410005075)
  (7.01, 4698.797472690587)
  (8.01, 4835.927029715684)
  (9.01, 4925.157872151898)
  (10.01, 4995.693902240593)
  (11.01, 4999.113813367361)
  (12.01, 4990.759960749009)
  (13.01, 4985.87538493477)
  (14.01, 4983.903065038801)
  (15.01, 4983.058310027304)
  (16.01, 4982.629458276466)
  (17.01, 4982.659160639976)
  (18.01, 4982.75517508182)
  (19.01, 4982.717213635151)
  (20.01, 4982.639686740012)
  (21.01, 4982.588698541161)
  (22.01, 4982.5860720071)
  (23.01, 4982.5860720071)
  (24.01, 4982.5860720071)
  (25.01, 4982.5860720071)
  (26.01, 4982.5860720071)
  (27.01, 4982.5860720071)
  (28.01, 4982.5860720071)
  (29.01, 4982.5860720071)
  (30.01, 4982.5860720071)
  (31.01, 4982.5860720071)
  (32.01, 4982.5860720071)
  (33.01, 4982.5860720071)
  (34.01, 4982.5860720071)
  (35.01, 4982.5860720071)
  (36.01, 4982.5860720071)
  (37.01, 4982.5860720071)
  (38.01, 4982.5860720071)
  (39.01, 4982.5860720071)
  (40.01, 4982.5860720071)
  (41.01, 4982.5860720071)
  (42.01, 4982.5860720071)
  (43.01, 4982.5860720071)
  (44.01, 4982.5860720071)
  (45.01, 4982.5860720071)
  (46.01, 4982.5860720071)
  (47.01, 4982.5860720071)
  (48.01, 4982.5860720071)
  (49.01, 4982.5860720071)
};

\addplot+[
] coordinates {
  (0.07032249999999998, 152.37267449446088)
};

\end{axis}
\end{tikzpicture}

%% file: experiments/faces_train_128_10/faces_train_128_10_cost_plot_2d.tex
\begin{tikzpicture}[scale=0.6]
\begin{axis}[
  ylabel = {$\alpha_2$},
  title = {2D Cost},
  xlabel = {$\alpha_1$},
  grid=both,
  view = {{0}{90}}
]

\addplot3[
  contour prepared,
  dashed
] table {
  0.09847041989229792 0.05001 170.0
  0.10001 0.05081875843691039 170.0
  0.11001 0.05688103264564913 170.0
  0.11476851572841304 0.06001 170.0
  0.11808550501328491 0.07001 170.0
  0.11418111063119718 0.08001 170.0
  0.11001 0.08513128372305724 170.0
  0.10570181437107506 0.09001 170.0
  0.10001 0.09493294503661648 170.0
  0.09322782388443991 0.10001 170.0
  0.09001 0.10211634266017196 170.0
  0.08001 0.10692319693039935 170.0
  0.07001 0.10861600058099619 170.0
  0.06001 0.10467295223905491 170.0
  0.05656138189696877 0.10001 170.0
  0.05006913340921828 0.09001 170.0
  0.05001 0.08989093691702311 170.0
  0.04771008558002157 0.08001 170.0
  0.0465985910784612 0.07001 170.0
  0.04754986570356008 0.06001 170.0
  0.05001 0.05436447891759793 170.0
  0.05405006982347606 0.05001 170.0
  0.06001 0.04712087499186308 170.0
  0.07001 0.04570407672044372 170.0
  0.08001 0.046302192173352956 170.0
  0.09001 0.04802069264634275 170.0
  0.09847041989229792 0.05001 170.0

  0.13287296720578254 0.10001 190.0
  0.14001 0.09289513196027618 190.0
  0.14284516482627235 0.09001 190.0
  0.15001 0.08043603108978174 190.0
  0.15032493796385096 0.08001 190.0
  0.1538157031914987 0.07001 190.0
  0.1507076806718994 0.06001 190.0
  0.15001 0.059511199174191516 190.0
  0.14001 0.05241122753827819 190.0
  0.13656992202618934 0.05001 190.0
  0.13001 0.04795539000827628 190.0
  0.12001 0.04487121909685768 190.0
  0.11001 0.041923335679639274 190.0
  0.10296173573392124 0.04001 190.0
  0.10001 0.0395657970661404 190.0
  0.09001 0.038264778616547895 190.0
  0.08001 0.03731584925573433 190.0
  0.07001 0.036990134109678136 190.0
  0.06001 0.037784183431376114 190.0
  0.051763437206058485 0.04001 190.0
  0.05001 0.04086155759783135 190.0
  0.041138745211550594 0.05001 190.0
  0.04001 0.052597226050856954 190.0
  0.03819571475080987 0.06001 190.0
  0.037661565809820455 0.07001 190.0
  0.038285240995241714 0.08001 190.0
  0.03959298773340565 0.09001 190.0
  0.04001 0.09244819216242822 190.0
  0.0423154655300602 0.10001 190.0
  0.04579351033764435 0.11001 190.0
  0.049555239477673374 0.12001 190.0
  0.05001 0.12115054005130176 190.0
  0.05751624515327396 0.13001 190.0
  0.06001 0.13283891593799074 190.0
  0.07001 0.13618445087099745 190.0
  0.08001 0.13481243460237466 190.0
  0.09001 0.13082054301061707 190.0
  0.09149687695982264 0.13001 190.0
  0.10001 0.1251693484883965 190.0
  0.10788058818288376 0.12001 190.0
  0.11001 0.1185284629762552 190.0
  0.12001 0.11098056753655955 190.0
  0.12124074296168426 0.11001 190.0
  0.13001 0.10250200017175357 190.0
  0.13287296720578254 0.10001 190.0

  0.06539491960722503 0.16001 210.0
  0.07001 0.16151871248874647 210.0
  0.08001 0.16025533109841913 210.0
  0.08066001278182541 0.16001 210.0
  0.09001 0.15644812040318576 210.0
  0.10001 0.15120471409143887 210.0
  0.1019862509534287 0.15001 210.0
  0.11001 0.1450974658445507 210.0
  0.11766727147975058 0.14001 210.0
  0.12001 0.13841752185657127 210.0
  0.13001 0.1313178225858618 210.0
  0.13181876529032574 0.13001 210.0
  0.14001 0.12383282720664895 210.0
  0.14497436160858312 0.12001 210.0
  0.15001 0.11588894196619018 210.0
  0.15710675744470803 0.11001 210.0
  0.16001 0.10740010549782207 210.0
  0.1681517653506392 0.10001 210.0
  0.17001 0.09807726466021113 210.0
  0.17771978863322366 0.09001 210.0
  0.18001 0.08685678788632006 210.0
  0.1849476675544487 0.08001 210.0
  0.18832530556393481 0.07001 210.0
  0.1852070159433549 0.06001 210.0
  0.18001 0.05625081971403292 210.0
  0.17135961553679774 0.05001 210.0
  0.17001 0.049571891367143725 210.0
  0.16001 0.04633637571565087 210.0
  0.15001 0.04312031900756355 210.0
  0.14026629763412757 0.04001 210.0
  0.14001 0.03996458292833455 210.0
  0.13001 0.03822574762956521 210.0
  0.12001 0.036514342042805056 210.0
  0.11001 0.03487979911155587 210.0
  0.10001 0.03337642763633126 210.0
  0.09001 0.032077829022788974 210.0
  0.08001 0.03113254901271828 210.0
  0.07001 0.030811613498410106 210.0
  0.06001 0.03161322256395182 210.0
  0.05001 0.034319907496300636 210.0
  0.040087614325060625 0.04001 210.0
  0.04001 0.04009003466584128 210.0
  0.034604648084807586 0.05001 210.0
  0.03215537503595617 0.06001 210.0
  0.031620970300204054 0.07001 210.0
  0.032244940899480386 0.08001 210.0
  0.03355432041634975 0.09001 210.0
  0.03526631696768134 0.10001 210.0
  0.03722017835149128 0.11001 210.0
  0.03933357064813804 0.12001 210.0
  0.04001 0.12302728857791687 210.0
  0.042797523623471015 0.13001 210.0
  0.04695462044287765 0.14001 210.0
  0.05001 0.14720815844537616 210.0
  0.05253103859997223 0.15001 210.0
  0.06001 0.15824199551578422 210.0
  0.06539491960722503 0.16001 210.0

  0.19001 0.04492525126755969 230.0
  0.18001 0.041671521634320675 230.0
  0.17489261762989566 0.04001 230.0
  0.17001 0.039129480810599686 230.0
  0.16001 0.0373306283354128 230.0
  0.15001 0.03554352228243551 230.0
  0.14001 0.0337715635934655 230.0
  0.13001 0.032032624792311296 230.0
  0.12001 0.030322362553806652 230.0
  0.11809827308674933 0.03001 230.0
  0.11001 0.02918529393132917 230.0
  0.10001 0.028248184103811254 230.0
  0.09001 0.027439251669309234 230.0
  0.08001 0.026851487359550494 230.0
  0.07001 0.0266543231123054 230.0
  0.06001 0.027158249952186908 230.0
  0.05001 0.028852861910970293 230.0
  0.04678720089405064 0.03001 230.0
  0.04001 0.033900641088592594 230.0
  0.034021118810140964 0.04001 230.0
  0.03001 0.047365226234279814 230.0
  0.0291070609170098 0.05001 230.0
  0.027574051879190813 0.06001 230.0
  0.027239599104249414 0.07001 230.0
  0.02763009729842057 0.08001 230.0
  0.028449779001726724 0.09001 230.0
  0.029521889096538724 0.10001 230.0
  0.03001 0.10399798802288268 230.0
  0.03118580300865687 0.11001 230.0
  0.033301008965016504 0.12001 230.0
  0.0355445076021079 0.13001 230.0
  0.03788244564641366 0.14001 230.0
  0.04001 0.14891496602601917 230.0
  0.04047541932897902 0.15001 230.0
  0.04476978246348897 0.16001 230.0
  0.04908414805154472 0.17001 230.0
  0.05001 0.17213265656680155 230.0
  0.05728056252204622 0.18001 230.0
  0.06001 0.18295476067512859 230.0
  0.07001 0.18619940942322175 230.0
  0.08001 0.1850132324130157 230.0
  0.09001 0.1813553568629352 230.0
  0.09267816814607384 0.18001 230.0
  0.10001 0.1762817012658925 230.0
  0.11001 0.17040427171578554 230.0
  0.11062900490406133 0.17001 230.0
  0.12001 0.1639351611578169 230.0
  0.12580107494561313 0.16001 230.0
  0.13001 0.15713296613095637 230.0
  0.14001 0.15016341055356205 230.0
  0.14022595738262023 0.15001 230.0
  0.15001 0.1429646843175625 230.0
  0.15407103351476772 0.14001 230.0
  0.16001 0.1355845863884488 230.0
  0.16741258606205764 0.13001 230.0
  0.17001 0.1279680232504327 230.0
  0.18001 0.12003982901535884 230.0
  0.18004731378470964 0.12001 230.0
  0.19001 0.11155326060367511 230.0
  0.19182368010606413 0.11001 230.0
  0.20001 0.10246206718646642 230.0
  0.20269610447183437 0.10001 230.0
  0.21001 0.09236315563359752 230.0
  0.2122679758020341 0.09001 230.0
  0.21956563488625286 0.08001 230.0
  0.22001 0.07870533108375835 230.0
  0.22299311729540988 0.07001 230.0
  0.22001 0.0607539633152649 230.0
  0.21977225798133704 0.06001 230.0
  0.21001 0.05307323227721153 230.0
  0.20572464868826482 0.05001 230.0
  0.20001 0.048170660659047244 230.0
  0.19001 0.04492525126755969 230.0

  0.13325422727119973 0.18001 250.0
  0.13001 0.18219124429740227 250.0
  0.12001 0.18878449907393496 250.0
  0.11806533593736229 0.19001 250.0
  0.11001 0.1950596713801965 250.0
  0.10145130825424045 0.20001 250.0
  0.10001 0.20084534117212133 250.0
  0.09001 0.2058722407805527 250.0
  0.08001 0.20952109684428472 250.0
  0.075885346124722 0.21001 250.0
  0.07001 0.2107052971827966 250.0
  0.06788014630767462 0.21001 250.0
  0.06001 0.207439315660499 250.0
  0.05314884258148309 0.20001 250.0
  0.05001 0.1966269770804006 250.0
  0.04711069572184899 0.19001 250.0
  0.0427318009736136 0.18001 250.0
  0.04001 0.1737709862367294 250.0
  0.03908765646262663 0.17001 250.0
  0.03666157794485096 0.16001 250.0
  0.03424575432013965 0.15001 250.0
  0.03185495991986552 0.14001 250.0
  0.03001 0.1321273992991547 250.0
  0.029699755331674585 0.13001 250.0
  0.028294576019990567 0.12001 250.0
  0.026970064053871423 0.11001 250.0
  0.02574545201097799 0.10001 250.0
  0.02467257561955485 0.09001 250.0
  0.02385244431619955 0.08001 250.0
  0.023461656181415012 0.07001 250.0
  0.023796365715009284 0.06001 250.0
  0.02533052089555747 0.05001 250.0
  0.028745231052630327 0.04001 250.0
  0.03001 0.03795133424592514 250.0
  0.0377949773510835 0.03001 250.0
  0.04001 0.02859626584974615 250.0
  0.05001 0.025004067205454686 250.0
  0.06001 0.023305569089241758 250.0
  0.07001 0.022798366819923626 250.0
  0.08001 0.02299237018731368 250.0
  0.09001 0.02357796864391859 250.0
  0.10001 0.024385359346313154 250.0
  0.11001 0.02532127022049057 250.0
  0.12001 0.02633958898623299 250.0
  0.13001 0.027406515818987204 250.0
  0.14001 0.028492007172362804 250.0
  0.15001 0.02959818707255453 250.0
  0.15369840726728182 0.03001 250.0
  0.16001 0.031138540888731733 250.0
  0.17001 0.03293852469701967 250.0
  0.18001 0.03474109658453456 250.0
  0.19001 0.03654788757820026 250.0
  0.20001 0.03835053532328339 250.0
  0.2092913891276815 0.04001 250.0
  0.21001 0.040241344527903344 250.0
  0.22001 0.04344263254072372 250.0
  0.23001 0.04661817368365848 250.0
  0.24001 0.0497724984738576 250.0
  0.24076523704682964 0.05001 250.0
  0.25001 0.05646828700869151 250.0
  0.2551040433319912 0.06001 250.0
  0.2583532595405985 0.07001 250.0
  0.2548583235455946 0.08001 250.0
  0.25001 0.08651916871368201 250.0
  0.24743299752006329 0.09001 250.0
  0.24001 0.09768246130257113 250.0
  0.23775612481778846 0.10001 250.0
  0.23001 0.10701495979703145 250.0
  0.22670643471030624 0.11001 250.0
  0.22001 0.1156145525922092 250.0
  0.21477640848944243 0.12001 250.0
  0.21001 0.12377230248141535 250.0
  0.20213148728534872 0.13001 250.0
  0.20001 0.13161936892507073 250.0
  0.19001 0.13927668050519248 250.0
  0.18905198002128176 0.14001 250.0
  0.18001 0.1467790444140608 250.0
  0.17564921293434405 0.15001 250.0
  0.17001 0.15413079667366916 250.0
  0.16190032731983997 0.16001 250.0
  0.16001 0.16136834384063556 250.0
  0.15001 0.16847435305676586 250.0
  0.1478285673504098 0.17001 250.0
  0.14001 0.17542644017189357 250.0
  0.13325422727119973 0.18001 250.0

  0.020018679550827756 0.06001 270.0
  0.021553980874105143 0.05001 270.0
  0.024970660123381744 0.04001 270.0
  0.03001 0.03180746114994899 270.0
  0.03177207094531824 0.03001 270.0
  0.04001 0.024752156863421276 270.0
  0.05001 0.02115527249993908 270.0
  0.05673968315072321 0.02001 270.0
  0.06001 0.019633904112302952 270.0
  0.07001 0.019289380894798216 270.0
  0.08001 0.01941836075702648 270.0
  0.09001 0.019812080601489774 270.0
  0.09365074247670603 0.02001 270.0
  0.10001 0.020522534588815054 270.0
  0.11001 0.021457246509651966 270.0
  0.12001 0.022474182453128985 270.0
  0.13001 0.023540383485605068 270.0
  0.14001 0.02462561603405524 270.0
  0.15001 0.025731670676088607 270.0
  0.16001 0.0268481539441906 270.0
  0.17001 0.027972200703754654 270.0
  0.18001 0.02909761180490446 270.0
  0.18809470552959584 0.03001 270.0
  0.19001 0.030356065051772407 270.0
  0.20001 0.03215927105541165 270.0
  0.21001 0.033947380546518094 270.0
  0.22001 0.03572599588039601 270.0
  0.23001 0.03749117823147596 270.0
  0.24001 0.03923903640775411 270.0
  0.2444321797090724 0.04001 270.0
  0.25001 0.041761559465982215 270.0
  0.26001 0.044884570300565446 270.0
  0.27001 0.04797425572189133 270.0
  0.276623834044627 0.05001 270.0
  0.28001 0.05232828856279225 270.0
  0.29001 0.05911285096625798 270.0
  0.29133724406268297 0.06001 270.0
  0.29464353481959715 0.07001 270.0
  0.2909889273285666 0.08001 270.0
  0.29001 0.08127545612974622 270.0
  0.2832840798005893 0.09001 270.0
  0.28001 0.09325359236268621 270.0
  0.2732993512798234 0.10001 270.0
  0.27001 0.10291661665520968 270.0
  0.2620267683788132 0.11001 270.0
  0.26001 0.11166727440218578 270.0
  0.25001 0.12000612478145438 270.0
  0.25000537860654326 0.12001 270.0
  0.24001 0.1278940788698566 270.0
  0.237313443021077 0.13001 270.0
  0.23001 0.13549874828091846 270.0
  0.2240131992662762 0.14001 270.0
  0.22001 0.14295485263584162 270.0
  0.2104708770967769 0.15001 270.0
  0.21001 0.15034652205331592 270.0
  0.20001 0.15766645444658312 270.0
  0.1968337985264983 0.16001 270.0
  0.19001 0.16499917972365513 270.0
  0.18315107453270144 0.17001 270.0
  0.18001 0.1722677564620623 270.0
  0.17001 0.17939366433665893 270.0
  0.16913903638579833 0.18001 270.0
  0.16001 0.1864184074389037 270.0
  0.15482957506860803 0.19001 270.0
  0.15001 0.19333461391870796 270.0
  0.14020964072066827 0.20001 270.0
  0.14001 0.200145867294592 270.0
  0.13001 0.20681111762280943 270.0
  0.1251244198671862 0.21001 270.0
  0.12001 0.21335815217460258 270.0
  0.11001 0.2196216325690599 270.0
  0.10933936142983246 0.22001 270.0
  0.10001 0.22542437243480518 270.0
  0.09091007740008829 0.23001 270.0
  0.09001 0.2304679747928876 270.0
  0.08001 0.23415948030265615 270.0
  0.07001 0.2353566893253881 270.0
  0.06001 0.23204055717596703 270.0
  0.058159721369928405 0.23001 270.0
  0.05001 0.22114871787032903 270.0
  0.04951516642128372 0.22001 270.0
  0.0451521861572694 0.21001 270.0
  0.040788191309413684 0.20001 270.0
  0.04001 0.19823478330393784 270.0
  0.03798246505278153 0.19001 270.0
  0.03551889793299678 0.18001 270.0
  0.033065014574953255 0.17001 270.0
  0.030637400409838318 0.16001 270.0
  0.03001 0.1574141763624702 270.0
  0.028889510864979186 0.15001 270.0
  0.027392269086587616 0.14001 270.0
  0.025926463946898782 0.13001 270.0
  0.024519926383026594 0.12001 270.0
  0.02319444577438777 0.11001 270.0
  0.02196901492541726 0.10001 270.0
  0.020895372237382977 0.09001 270.0
  0.02007479133397853 0.08001 270.0
  0.02001 0.07835315736504009 270.0
  0.019789206202523016 0.07001 270.0
  0.02001 0.06026913429697949 270.0
  0.020018679550827756 0.06001 270.0

  0.27001 0.03823088948355407 290.0
  0.26001 0.03651803331894341 290.0
  0.25001 0.03478645424809962 290.0
  0.24001 0.033044525050259045 290.0
  0.23001 0.031300422391822735 290.0
  0.22269811865989453 0.03001 290.0
  0.22001 0.029713674226936965 290.0
  0.21001 0.02860233884470443 290.0
  0.20001 0.027485513995279806 290.0
  0.19001 0.026359073577204517 290.0
  0.18001 0.025230277078570378 290.0
  0.17001 0.024105168910442527 290.0
  0.16001 0.02298160960684685 290.0
  0.15001 0.02186515427962268 290.0
  0.14001 0.020759224895747677 290.0
  0.13310467735854678 0.02001 290.0
  0.13001 0.01978344844512401 290.0
  0.12001 0.01906445606658559 290.0
  0.11001 0.018379017609230754 290.0
  0.10001 0.017749326763278794 290.0
  0.09001 0.017206607218194836 290.0
  0.08001 0.016814182811061278 290.0
  0.07001 0.016686624081624163 290.0
  0.06001 0.017033030288490632 290.0
  0.05001 0.018184696690558274 290.0
  0.042505770777598414 0.02001 290.0
  0.04001 0.020908047877096405 290.0
  0.03001 0.02729383483251105 290.0
  0.027342072254758536 0.03001 290.0
  0.021196089194133162 0.04001 290.0
  0.02001 0.04348065003388443 290.0
  0.018498970597555578 0.05001 290.0
  0.017459502309929195 0.06001 290.0
  0.017232723640727043 0.07001 290.0
  0.01749735728594354 0.08001 290.0
  0.01805281418894341 0.09001 290.0
  0.018779813089504174 0.10001 290.0
  0.01960980785766101 0.11001 290.0
  0.02001 0.11446617066207533 290.0
  0.02074527674606263 0.12001 290.0
  0.022153172562122975 0.13001 290.0
  0.023619827158954516 0.14001 290.0
  0.0251176206654927 0.15001 290.0
  0.026631429555008638 0.16001 290.0
  0.028151828346705897 0.17001 290.0
  0.029689490544486802 0.18001 290.0
  0.03001 0.18208565461762824 290.0
  0.03196327423201605 0.19001 290.0
  0.034429835531859354 0.20001 290.0
  0.03688459248160758 0.21001 290.0
  0.039339064764698686 0.22001 290.0
  0.04001 0.22275408982404013 290.0
  0.04316433587144143 0.23001 290.0
  0.04746795756863181 0.24001 290.0
  0.05001 0.24601486041606518 290.0
  0.05358900239783913 0.25001 290.0
  0.06001 0.25727952493000616 290.0
  0.06796204605066444 0.26001 290.0
  0.07001 0.26072277064286814 290.0
  0.07576362433218715 0.26001 290.0
  0.08001 0.2594899607939668 290.0
  0.09001 0.25567905100717314 290.0
  0.10001 0.25042853153843847 290.0
  0.10070156761703006 0.25001 290.0
  0.11001 0.24446474555481879 290.0
  0.11692910071216056 0.24001 290.0
  0.12001 0.23806080303967322 290.0
  0.13001 0.23147532995661207 290.0
  0.1321978433865256 0.23001 290.0
  0.14001 0.2248067696485102 290.0
  0.14705292737415063 0.22001 290.0
  0.15001 0.21800149393583532 290.0
  0.16001 0.21113095459255699 290.0
  0.161620604062293 0.21001 290.0
  0.17001 0.20415390340943898 290.0
  0.1758928221547271 0.20001 290.0
  0.18001 0.19711319367409785 290.0
  0.19001 0.19002229780016686 290.0
  0.1900273467559782 0.19001 290.0
  0.20001 0.18289498695263967 290.0
  0.20409224572638904 0.18001 290.0
  0.21001 0.17579317240638442 290.0
  0.21812921881569303 0.17001 290.0
  0.22001 0.16864864937854196 290.0
  0.23001 0.16144698373814365 290.0
  0.23201099642673473 0.16001 290.0
  0.24001 0.1542138152215548 290.0
  0.24578152544154275 0.15001 290.0
  0.25001 0.14688937425155357 290.0
  0.2593419234821283 0.14001 290.0
  0.26001 0.13950602574453122 290.0
  0.27001 0.1320362739641976 290.0
  0.27274990898542023 0.13001 290.0
  0.28001 0.12442055198821675 290.0
  0.28581172023962215 0.12001 290.0
  0.29001 0.11662674806305295 290.0
  0.29825984499132174 0.11001 290.0
  0.30001 0.10849446799088383 290.0
  0.3098544987697947 0.10001 290.0
  0.31001 0.09985731061683037 290.0
  0.32001 0.09009885450021299 290.0
  0.32010050353299274 0.09001 290.0
  0.32787223017664396 0.08001 290.0
  0.33001 0.07420367877919612 290.0
  0.3315573259185039 0.07001 290.0
  0.33001 0.0653132186290027 290.0
  0.3282636329748176 0.06001 290.0
  0.32001 0.05444432800226551 290.0
  0.313388879902479 0.05001 290.0
  0.31001 0.048993128094341205 290.0
  0.30001 0.04596743634642208 290.0
  0.29001 0.04292872343892843 290.0
  0.28043208653699153 0.04001 290.0
  0.28001 0.039938641280442395 290.0
  0.27001 0.03823088948355407 290.0

};

\addplot+[
] coordinates {
  (0.07145531249999999, 0.06918968749999997)
};

\end{axis}
\end{tikzpicture}